\documentclass{amsart}
\usepackage{amssymb}
\usepackage[all]{xy}
\usepackage{mathrsfs}
\usepackage{url}

\newcommand{\overto}[1]{\xrightarrow{\;#1\;}}
\newcommand{\phat}{^{\sma}_{p}}
\mathchardef\nhyphen=45

\newcommand{\ssdot}{\bullet}

\newcommand{\subdot}{_\ssdot}

\newcommand{\Com}{\aC{om}}
\newcommand{\Ass}{\aA{ss}}
\newcommand{\Alg}{\nhyphen\aA{lg}}
\newcommand{\CAlg}{\nhyphen\aC{om}}
\newcommand{\Mod}{\nhyphen\aM{od}}
\newcommand{\Spec}{\aS}

\newcommand{\opTR}{{^{op}TR}}
\newcommand{\opTC}{{^{op}TC}}

\newcommand{\AN}[1][{A}]{{{}_{#1}N}}
\newcommand{\gAN}[1][{A}]{{{}^{g}_{#1}N}}
\newcommand{\ATC}[1][{A}]{{{}_{#1}TC}}
\newcommand{\ATR}[1][{A}]{{{}_{#1}TR}}
\newcommand{\ATF}[1][{A}]{{{}_{#1}TF}}
\newcommand{\ATHH}[1][{A}]{{{}_{#1}THH}}

\newcommand{\uTor}{\mathop{\rm \underline{Tor}\mathstrut}\nolimits}
\newcommand{\upi}{\underline\pi\mathstrut}
\newcommand{\uE}{\underline E\mathstrut}

\newcommand{\dersma}{\sma^{\mathbf{L}}}

\let\iso\cong
\let\sma\wedge
\newcommand{\htp}{\simeq}
\renewcommand{\to}{\mathchoice{\longrightarrow}{\rightarrow}{\rightarrow}{\rightarrow}}
\newcommand{\from}{\mathchoice{\longleftarrow}{\leftarrow}{\leftarrow}{\leftarrow}}

\newcommand{\sI}{\mathscr{I}}
\newcommand{\sJ}{\mathscr{J}}

\let\catsymbfont\mathcal
\newcommand{\aA}{{\catsymbfont{A}}}
\newcommand{\aB}{{\catsymbfont{B}}}
\newcommand{\aC}{{\catsymbfont{C}}}

\newcommand{\aF}{{\catsymbfont{F}}}
\newcommand{\aFin}{{\catsymbfont{F}}_{\textup{Fin}}}
\newcommand{\aI}{{\catsymbfont{I}}}

\newcommand{\aM}{{\catsymbfont{M}}}

\newcommand{\aP}{{\catsymbfont{P}}}

\newcommand{\aS}{{\catsymbfont{S}}}
\newcommand{\aT}{{\catsymbfont{T}}}

\newcommand{\aV}{{\catsymbfont{V}}}

\newcommand{\bC}{\mathbb{C}}

\newcommand{\bF}{\mathbb{F}}

\newcommand{\bP}{{\mathbb{P}}}
\newcommand{\bQ}{{\mathbb{Q}}}
\newcommand{\bR}{{\mathbb{R}}}

\newcommand{\bT}{\mathbb{T}}

\newcommand{\bZ}{{\mathbb{Z}}}

\def\quickop#1{\expandafter\DeclareMathOperator\csname
#1\endcsname{#1}}
\quickop{id}\quickop{Id}\quickop{colim}\quickop{hocolim}\quickop{op}
\quickop{co}\quickop{Ar}\quickop{sing}\quickop{Hom}\quickop{w}\quickop{Ho}
\quickop{ob}\quickop{diag}\quickop{Stab}\quickop{Cat}\quickop{Mot}
\quickop{map}\quickop{Cone}\quickop{End}\quickop{Idem}\quickop{Perf}\quickop{Ind}
\quickop{Gap}\quickop{Shv}\quickop{Spaces}\quickop{cof}\quickop{ex}\quickop{perf}
\quickop{tri}\quickop{Set}\quickop{Fib}\quickop{Nat}\quickop{haut}\quickop{holim}\quickop{Tor}\quickop{Ext}\quickop{Diff}\quickop{lax}\quickop{Map}\quickop{Comm}\quickop{Top}\quickop{Tr}\quickop{Res}\quickop{cyc}\quickop{sd}\quickop{fib}\quickop{cf}\quickop{GL}
\DeclareMathOperator*\lcolim{colim}
\quickop{Fix}\quickop{im}

\numberwithin{equation}{section}

\newtheorem{theorem}[equation]{Theorem}
\newtheorem*{theorem*}{Theorem}
\newtheorem{corollary}[equation]{Corollary}
\newtheorem{lemma}[equation]{Lemma}
\newtheorem{proposition}[equation]{Proposition}
\theoremstyle{definition}
\newtheorem{definition}[equation]{Definition}
\newtheorem{notation}[equation]{Notation}
\newtheorem{remark}[equation]{Remark}
\newtheorem{example}[equation]{Example}
\newtheorem{warning}[equation]{Warning}

\xyoption{arrow}
\xyoption{matrix}
\xyoption{cmtip}
\SelectTips{cm}{}

\newcommand{\term}[1]{\textit{#1}}

\newdir{ >}{{}*!/-5pt/\dir{>}}

\begin{document}

\title{Topological cyclic homology via the norm}

\author[V.Angeltveit]{Vigleik Angeltveit}
\address{Australian National University, Canberra, Australia}
\email{vigleik.angeltveit@anu.edu.au}
\thanks{Angeltveit was supported in part by an NSF All-Institutes Postdoctoral Fellowship administered by the Mathematical Sciences Research Institute through its core grant DMS-0441170, NSF grant DMS-0805917, and an Australian Research Council Discovery Grant}
\author[A.Blumberg]{Andrew~J. Blumberg}
\address{University of Texas, Austin, TX 78712}
\email{blumberg@math.utexas.edu}
\thanks{Blumberg was supported in part by NSF grant DMS-1151577}
\author[T.Gerhardt]{Teena Gerhardt}
\address{Michigan State University, East Lansing, MI 48824}
\email{teena@math.msu.edu}
\thanks{Gerhardt was supported in part by NSF grants DMS-1007083 and DMS-1149408}
\author[M.Hill]{Michael~A. Hill}
\address{University of California Los Angeles\\Los Angeles, CA 90025}
\email{mikehill@math.ucla.edu}
\thanks{Hill was supported in part by NSF grant DMS-0906285, DARPA
grant FA9550-07-1-0555, and the Sloan Foundation}
\author[T.Lawson]{Tyler Lawson}
\address{University of Minnesota, Minneapolis, MN 55455}
\thanks{Lawson was supported in part by NSF grant DMS-1206008}
\email{tlawson@math.umn.edu}
\author[M.Mandell]{Michael~A. Mandell}
\address{Indiana University\\Bloomington, IN 47405}
\email{mmandell@indiana.edu}
\thanks{Mandell was supported in part by NSF grants DMS-1105255, DMS-1505579}


\begin{abstract}
We describe a construction of the cyclotomic structure on topological
Hochschild homology ($THH$) of a ring spectrum using the
Hill--Hopkins--Ravenel multiplicative norm.  Our analysis takes place
entirely in the category of equivariant orthogonal spectra, avoiding
use of the B\"okstedt coherence machinery.  We are also able to define
two relative versions of topological cyclic homology ($TC$) and
$TR$-theory: one starting with a ring $C_n$-spectrum and one starting
with an algebra over a cyclotomic commutative ring spectrum $A$.  We
describe spectral sequences computing the relative theory $\ATR$ in
terms of $TR$ over the sphere spectrum and vice versa.  Furthermore,
our construction permits a straightforward definition of the Adams
operations on $TR$ and $TC$.
\end{abstract}

\maketitle \setcounter{tocdepth}{1} \tableofcontents

\section{Introduction}

Over the last two decades, the calculational study of algebraic
$K$-theory has been revolutionized by the development of trace
methods.  In analogy with the Chern character from topological
$K$-theory to ordinary cohomology, there exist ``trace maps'' from
algebraic $K$-theory to various more homological approximations, which
also can be more computable.  For a ring $R$, Dennis constructed a map to Hochschild homology 
\[K(R) \to HH(R)
\]
 that generalizes the trace of a matrix.  Goodwillie
lifted this trace map to negative cyclic homology
\[
K(R) \to HC^{-}(R) \to HH(R)
\]
and showed that, rationally, this map can often be used to compute
$K(R)$.

In his 1990 ICM address, Goodwillie conjectured that there should be a
``brave new'' version of this story involving ``topological''
analogues of Hochschild and cyclic homology defined
by changing the ground ring from $\bZ$ to the sphere spectrum.
Although the modern symmetric monoidal categories of spectra had not
yet been invented, B\"okstedt developed coherence machinery that
enabled a definition of topological Hochschild homology ($THH$) along these lines. Further, he constructed a
``topological'' Dennis trace map ~\cite{Bokstedt}
\[K(R) \to THH(R).
\]
Subsequently, B\"okstedt--Hsiang--Madsen~\cite{BHM} defined topological cyclic homology ($TC$) and
constructed the cyclotomic trace map to $TC$, lifting the topological Dennis trace
\[
K(R) \to TC(R) \to THH(R).
\]
They did this in the course of resolving the $K$-theory Novikov conjecture for
groups satisfying a mild finiteness hypothesis.  Subsequently, seminal work of
McCarthy~\cite{McCarthy} and Dundas~\cite{Dundas} showed that when
working at a prime $p$, $TC$ often captures a great deal of
information about $K$-theory.  Hesselholt and Madsen (inter alia,
\cite{HMannals}) then used $TC$ to make extensive computations in
$K$-theory, including a computational resolution of the
Quillen--Lichtenbaum conjecture for certain fields.

The calculational power of trace methods depends on the ability to
compute $TC(R)$, which ultimately derives from the methods of equivariant
stable homotopy theory. B\"okstedt's definition of $THH(R)$ closely
resembles a cyclic bar construction, and as a consequence $THH(R)$ is an
$S^1$-spectrum. Topological cyclic homology is constructed from this $S^1$-action
on $THH(R)$, via fixed point spectra $TR^n(R) = THH(R)^{C_{p^{n}}}$. In fact, $THH(R)$ has a very special
equivariant structure: $THH(R)$ is a \emph{cyclotomic} spectrum, which
is an $S^1$-spectrum equipped with additional data that
models the structure of a free loop space $\Lambda X$.

The cyclic bar construction can be formed in any symmetric monoidal
category $(A, \boxtimes, 1)$; we will let $N^{\cyc}_{\boxtimes}$
denote the resulting simplicial (or cyclic) object.  Recall that in
the category of spaces, for a group-like monoid $M$, there is a
natural map
\[
|N^{\cyc}_{\times} M| \to \Map(S^1, BM) = \Lambda BM
\]
(where $|\cdot|$ denotes geometric realization)
that is a weak equivalence on fixed points for any finite subgroup
$C_n < S^1$.  Moreover, for each such $C_n$, the free loop space
is equipped with equivalences (in fact homeomorphisms) 
\[
(\Lambda BM)^{C_n} \cong \Lambda BM
\]
of $S^1$-spaces, where $(\Lambda
BM)^{C_n}$ is regarded as an $S^1$-space (rather than an
$S^{1}/C_{n}$-space) via pullback along the $n$th root isomorphism
\[
\rho_n \colon S^1 \cong S^1/C_n.
\]
In analogy, a cyclotomic spectrum is an $S^1$-spectrum
equipped with compatible equivalences of $S^1$-spectra
\[
t_n \colon \rho^*_n L\Phi^{C_n} X \to X,
\]
where $L\Phi^{C_{n}}$ denotes the (left derived) ``geometric'' fixed point
functor.

The construction of the cyclotomic structure on $THH$ has classically
been one of the more subtle and mysterious parts of the construction
of $TC$.  In a modern symmetric monoidal category of
spectra (e.g., symmetric spectra or EKMM $S$-modules), one can simply
define $THH(R)$ as
\[
THH(R) = |N^{\cyc}_{\sma} R|,
\]
but the resulting equivariant spectrum did not have the correct
homotopy type.  Only B\"okstedt's original construction of $THH$
seemed to produce the cyclotomic structure.

Although this situation has not impeded the calculational
applications, reliance on the B\"okstedt construction has limited
progress in certain directions.  For one thing, it does not seem to be
possible to use the B\"okstedt construction to define $TC$ relative to
a ground ring that is not the sphere spectrum $S$.  
Moreover, the details of the B\"okstedt construction make
it difficult to understand the equivariance (and therefore relevance
to $TC$) of various additional algebraic structures that arise on
$THH$, notably the Adams operations and the coalgebra structures.

The purpose of this paper is to introduce a new approach to the
construction of the cyclotomic structure on $THH$ using an
interpretation of $THH$ in terms of the Hill--Hopkins--Ravenel
multiplicative norm.  Our point of departure is the observation that
the construction of the cyclotomic structure on $THH(R)$ ultimately
boils down to having good models of the smash powers
\[
R^{\sma n} = \underbrace{R \sma R \sma \ldots \sma R}_{n}
\]
of a spectrum $R$ as a $C_n$-spectrum such that there is a
suitably compatible collection of diagonal equivalences
\[
R \to \Phi^{C_n} R^{\sma n}.
\]
The recent solution of the Kervaire invariant one problem involved the
detailed analysis of a multiplicative norm construction in equivariant
stable homotopy theory that has precisely this type of behavior.  Although
Hill--Hopkins--Ravenel studied the norm construction $N_{H}^G$ for
a finite group $G$ and subgroup $H$, using the cyclic bar construction one can
extend this construction to a norm $N_e^{S^1}$ on associative ring
orthogonal spectra; such a construction first appeared in the thesis
of Martin Stolz~\cite{Stolz}. 

For the following definition, we need to introduce some notation.  Let
$\Spec$ denote the category of orthogonal spectra and let
$\Spec^{S^1}_{U}$ denote the category of orthogonal $S^1$-spectra
indexed on the complete universe $U$.  Finally, let $\Ass$ 
denote the category of associative ring orthogonal
spectra.

\begin{definition}\label{defn:Tnorm}
Define the functor
\[
N_e^{S^1} \colon \Ass \to \Spec^{S^1}_{U}
\]
to be the composite functor
\[
R \mapsto \aI_{\bR^{\infty}}^{U} |N^{\cyc}_{\sma} R|,
\]
with $|N^{\cyc}_{\sma} R|$ regarded as an orthogonal $S^1$-spectrum
indexed on the standard trivial universe $\bR^{\infty}$. Here $\aI_{\bR^{\infty}}^{U}$ denotes the change of universe functor (see Definition \ref{def:changeuniverse}).
\end{definition}

Since both the cyclic bar construction and the change of universe
functor preserve commutative ring orthogonal spectra, the norm above
also preserves commutative ring orthogonal spectra.  In the following
proposition, proved in Section \ref{sec:THH}, $\Com$ and
$\Com^{S^1}_{U}$ denote the categories of commutative ring orthogonal
spectra and commutative ring orthogonal $S^1$-spectra, respectively.

\begin{proposition}
$N_e^{S^1}$ restricts to a functor
\[
N_e^{S^1} \colon\Com \to \Com^{S^{1}}_{U}
\]
that is the left
adjoint to the forgetful functor from commutative ring orthogonal
$S^1$-spectra to commutative ring orthogonal spectra.
\end{proposition}

The forgetful functor from commutative ring orthogonal
$S^1$-spectra to commutative ring orthogonal spectra is the composite
of the change of universe functor $\aI_{U}^{\bR^{\infty}}$ and the
functor that forgets equivariance. The proof of the above proposition
identifies $N_{e}^{S^{1}}\colon\Com \to\Com^{S^{1}}_{U}$
as the composite functor
\[
R \mapsto \aI_{\bR^{\infty}}^{U} (R \otimes S^1),
\]
which is left adjoint to the forgetful functor. Here $\otimes$ denotes the tensor of a commutative ring orthogonal
spectrum with an unbased space, and we regard $(-)\otimes S^{1}$ as a
functor from commutative ring orthogonal spectra to commutative ring
orthogonal spectra with an action of $S^{1}$.

The deep aspect of the Hill--Hopkins--Ravenel treatment of the norm
functor is their analysis of the left derived functors of the norm.
As part of this analysis they show that the norm $N_{H}^{G}$ preserves
certain weak equivalences.  For our norm $N_{e}^{S^{1}}$ into
$\Spec^{S^1}_{U}$, we work with the homotopy theory defined by the 
\emph{$\aF$-equivalences} of orthogonal $S^{1}$-spectra, where an
$\aF$-equivalence is a map that induces an isomorphism on all the
homotopy groups at the fixed point spectra for the finite subgroups of
$S^{1}$.  We prove the following theorem in Section~\ref{sec:THH}.

\begin{proposition}\label{prop:derive}
Assume that $R$ is a cofibrant associative ring orthogonal
spectrum and $R'$ is either a cofibrant associative ring orthogonal
spectrum or a cofibrant commutative ring orthogonal spectrum.  If
$R \to R'$ is a weak equivalence, then $N_{e}^{S^{1}} R \to
N_{e}^{S^{1}} R'$ is an $\aF$-equivalence in $\Spec^{S^1}_{U}$.
\end{proposition}

Of course the conclusion holds if $R$ is a cofibrant commutative ring
orthogonal spectrum as well; the point of
Proposition~\ref{prop:derive} is to compare cofibrant replacements in
associative and commutative ring orthogonal spectra.

As a consequence we obtain the following additional observation about
the adjunction in the commutative case.  See
Proposition~\ref{prop:Tcommquillen} for a more precise statement. 

\begin{proposition}\label{prop:dercom}
The functor
\[
N_{e}^{S^{1}}\colon \Com\to \Com^{S^{1}}_{U}
\]
is Quillen left adjoint to the forgetful functor 
(for an appropriate model structure with weak equivalences the
$\aF$-equivalences on the codomain); in particular, its left
derived functor exists and is left
adjoint to the right derived forgetful functor.
\end{proposition}

Our first main theorem is that when $R$ is a cofibrant associative
ring orthogonal spectrum, $N_{e}^{S^{1}}R$ is a cyclotomic spectrum.
To be precise, we use the point-set model of cyclotomic spectra
from~\cite{BM}, which provides a definition entirely in terms of the
category of orthogonal $S^1$-spectra.

\begin{theorem}\label{thm:maincyc}
Let $R$ be a cofibrant associative or cofibrant commutative ring
orthogonal spectrum.  Then $N_{e}^{S^{1}}R$ has a natural structure of
a cyclotomic spectrum.
\end{theorem}

Experts will recognize that one can give a direct construction of the
cyclotomic trace induced by the inclusion of objects in a spectral
category enriched in orthogonal spectra (e.g., see~\cite{BMtw}).  We
review this construction in Section~\ref{sec:trace}.

Proposition~\ref{prop:dercom}, which describes $N_{e}^{S^{1}}$ as the
homotopical left adjoint to the forgetful functor, suggests a
generalization of our construction of $THH$ that takes ring orthogonal
$C_{n}$-spectra as input.  For commutative ring orthogonal
$C_{n}$-spectra, we can define $N_{C_{n}}^{S^{1}}$ as the left adjoint
to the forgetful functor.  However, to extend to the non-commutative
case, we need an explicit construction. We give such a construction in
Section~\ref{sec:cncyclo} in terms of a cyclic bar construction, which
we denote as $N^{\cyc,C_{n}}_{\sma} R$. Its geometric realization
$|N^{\cyc,C_{n}}_{\sma} R|$ has an $S^{1}$-action, and by promoting it
to the complete universe we obtain a genuine 
orthogonal $S^{1}$-spectrum that we denote as $N_{C_{n}}^{S^{1}}R$.  The
following proposition is a consistency check.

\begin{proposition}
Let $R$ be a commutative ring orthogonal $C_{n}$-spectrum.  Then
$N_{C_{n}}^{S^{1}} R$ is isomorphic to the left adjoint of the
forgetful functor from commutative ring orthogonal $S^{1}$-spectra to
commutative ring orthogonal $C_{n}$-spectra.
\end{proposition}

Again, we can describe the left adjoint in terms of a tensor
\[
N_{C_{n}}^{S^{1}} R = \aI_{\bR^{\infty}}^{U} (R \otimes_{C_{n}} S^{1}),
\]
where the relative tensor $R \otimes_{C_{n}} S^{1}$ may be explicitly
constructed as the coequalizer
\[ 
(i^{*} R) \otimes C_{n} \otimes S^{1} \rightrightarrows (i^{*} R)
\otimes S^{1}
\]
of the canonical action of $C_{n}$ on $S^{1}$ and the action map
$(i^{* }R) \otimes C_{n} \to i^{*}R$, where $i^{*}$ denotes the
change-of-group functor to the trivial group.  Choosing an
appropriately subdivided model of the circle produces the isomorphism
between the two descriptions.  

As above, by cofibrantly replacing $R$ we can compute the left-derived
functor of $N_{C_{n}}^{S^{1}}$, and in this case $N_{C_{n}}^{S^{1}}R$
is a $p$-cyclotomic spectrum (see Definition~\ref{def:pcyclotomic}) provided either $n$ is prime to $p$ or
$R$ is ``$C_{n}$-cyclotomic'' (q.v.~Definition~\ref{defn:cncyclo}
below).  This leads to the obvious definition of 
$TC_{C_{n}} R$.  This
$C_{n}$-relative $THH$ (and the associated constructions of $TR$ and
$TC$) is expected to be both interesting and comparatively easy to
compute for some of the equivariant spectra that arise in
Hill--Hopkins--Ravenel, in particular the real cobordism spectrum
$MU_{\bR}$.

We can also consider another kind of relative construction, namely in
the situation where $R$ is an algebra over an arbitrary commutative
ring orthogonal spectrum $A$.  Definition~\ref{defn:Tnorm} can be
extended to the relative setting; the equivariant indexed product can be carried out
in any symmetric monoidal category, and the homotopical analysis
in the case of $A$-modules is given in Section~\ref{sec:relTHH}.

\begin{definition}
Let $A$ be a cofibrant commutative ring orthogonal spectrum, and
denote by $A\Alg$ the category of $A$-algebras.  We define
the $A$-relative norm functor
\[
\AN_e^{S^1} \colon A\Alg \to A_{S^{1}}\Mod^{S^{1}}_{U}
\]
by
\[
R\mapsto \aI_{\bR^{\infty}}^{U}|N^{\cyc}_{\sma_{A}}R|.
\]
Here $A_{S^{1}}$ denotes $\aI_{\bR^{\infty}}^{U}A$, constructed by
applying the point-set change of universe functor
$\aI_{\bR^{\infty}}^{U}$ to $A$ regarded  as
a commutative ring orthogonal $S^{1}$-spectrum (on the universe
$\bR^{\infty}$) with trivial $S^{1}$-action.  Then $A_{S^{1}}$ is a
commutative ring orthogonal 
$S^{1}$-spectrum (on the universe $U$) and
$A_{S^1}\Mod^{S^1}_{U}$ denotes the category of
$A_{S^{1}}$-modules in $\Spec^{S^{1}}_{U}$.
\end{definition}

We write $\ATHH(R)$ for the underlying non-equivariant spectrum of
$\AN_{e}^{S^{1}}R$; this spectrum was denoted $thh^{A}(R)$ in
\cite[IX.2.1]{EKMM}.
When $R$ is a commutative $A$-algebra, $\AN_{e}^{S^{1}}R$ is naturally
a commutative $A_{S^{1}}$-algebra.  The functor
\[
\AN_{e}^{S^{1}}\colon A\CAlg\to
A_{S^1}\CAlg^{S^{1}}_{U}
\]
is again left adjoint to the forgetful functor.
However, due to the subtleties of the behavior of
$\aI_{\bR^{\infty}}^U$ when applied to cofibrant commutative ring
orthogonal spectra regarded as $S^1$-spectra with trivial action,
$\AN_e^{S^1}R$ is not in general cyclotomic.  Instead, we must settle
for the following weaker analogue of Theorem~\ref{thm:maincyc}, which
we prove in Section~\ref{sec:amodcyc}.

\begin{theorem}\label{thm:relwhencyc}
Let $A$ be a cofibrant commutative ring orthogonal spectrum that is
$\iota_e^* \underbar{A}$ for a cofibrant $p$-cyclotomic commutative
ring orthogonal $S^1$-spectrum $\underbar{A}$.  Moreover, assume that
the canonical counit map $N_e^{S^1} A \to \underbar{A}$ is a
$p$-cyclotomic map.  Let $R$ be a cofibrant $A$-algebra.  Then
\[
\AN_e^{S^1} R \cong N_e^{S^1} R \sma_{N_e^{S^1} A} \underbar{A}
\]
is a $p$-cyclotomic spectrum.
\end{theorem}

In fact, we have a slightly more general version of this result.

\begin{theorem}\label{thm:relwhencycmod}
Let $A$ be a cofibrant commutative ring orthogonal spectrum and $R$ a
cofibrant $A$-algebra.  Let $M$ be a $p$-cyclotomic object in
$N_e^{S_1} A$-modules.  Then the smash product
\[
N_e^{S^1} R \sma_{N_e^{S^1} A} M
\]
is a $p$-cyclotomic spectrum.
\end{theorem}

When these theorems apply, we can form relative topological cyclic
homology $\ATC(R)$, which is the target of an $A$-relative cyclotomic
trace $K(R)\to \ATC(R)$, factoring though the usual cyclotomic trace
$K(R)\to TC(R)$, essentially by construction.

\begin{theorem}\label{thm:trace}
Under the hypotheses above, there is an $A$-relative
cyclotomic trace map $K(R)\to \ATC(R)$ making the
following diagram commute in the stable category
\[
\xymatrix{%
K(R)\ar[r]\ar[dr]&TC(R)\ar[r]\ar[d]&THH(R)\ar[d]\\
& \ATC(R)\ar[r]& \ATHH(R).  }
\]
\end{theorem}

Using the identification $N_{e}^{S^{1}}A\iso\aI_{\bR^{\infty}}^{U} (A
\otimes S^1)$ in the commutative context, the map $S^{1}\to *$ induces
a map of equivariant commutative ring orthogonal spectra
$N_{e}^{S^{1}}A\to A_{S^{1}}$.  Just as in the
non-equivariant case, we can identify $\AN_{e}^{S^{1}}R$ as
extension of scalars along this map.

\begin{proposition}\label{prop:introsmader}
Let $R$ be an associative $A$-algebra.
There is a natural isomorphism
\[
\AN_{e}^{S^{1}}R\iso N_{e}^{S^{1}}R\sma_{N_{e}^{S^{1}}A}
A_{S^{1}}.
\]
When $R$ is a cofibrant associative $A$-algebra or cofibrant
commutative $A$-algebra, this induces a natural isomorphism in the
stable category
\[
\AN_{e}^{S^{1}} R\iso N_{e}^{S^{1}} R \dersma_{N_{e}^{S^{1}}A}
A_{S^{1}}.
\]
\end{proposition}

The equivariant homotopy groups $\pi^{C_{n}}_{*}(N_{e}^{S^{1}}R)$ are the
$TR$-groups $TR^{n}_{*}(R)$ and so $\pi^{C_{n}}_{*}(\AN_{e}^{S^{1}}R)$
are by definition the relative $TR$-groups $\ATR^{n}_{*}(R)$.  The
K\"unneth spectral sequence of \cite{LewisMandell2} can be combined
with the previous theorem to compute the relative $TR$-groups from the
absolute $TR$-groups and Mackey functor $\uTor$.  
More often we expect to use the
relative theory to compute the absolute theory.  Non-equivariantly,
the isomorphism
\begin{equation}\label{eq:changebasering}
THH(R) \sma A \cong \ATHH(R \sma A)
\end{equation}
gives rise to a K\"unneth spectral sequence
\[
\Tor^{A_*(R \sma_{S} R^{\op})}_{*,*}(A_*(R), A_*(R))
\quad\Longrightarrow\quad A_* (THH(R)).
\]
An Adams spectral sequence can then in theory be used to compute the
homotopy groups of $THH(R)$.  For formal reasons, the
isomorphism~\eqref{eq:changebasering} still holds equivariantly, but
now we have three different versions of the non-equivariant K\"unneth
spectral sequence (none of which have quite as elegant an
$E^{2}$-term) which we use in conjunction with
equation~\eqref{eq:changebasering}.  We discuss these in
Section~\ref{sec:spectralsequences}.  

A further application of our model of $THH$ and $TC$ is a construction,
when $R$ is commutative, of Adams operations on $N_{e}^{S^{1}}R$ and
$\AN_{e}^{S^{1}} R$ that are compatible (in the absolute case) with
the cyclotomic structure.  McCarthy explained how Adams operations can
be constructed on any cyclic object that, when viewed as a functor from
the cyclic category, factors through the category of finite sets (and
all maps).  As a consequence, it is possible to construct Adams
operations on $THH$ of a commutative monoid object in any symmetric
monoidal category of spectra.  An advantage of our formulation is
that we can easily verify the equivariance of these operations and in
particular show they descend to $TC$.  We prove the following theorem
in Section~\ref{sec:adams}.

\begin{theorem}
Let $A$ be a commutative ring orthogonal spectrum and $R$ a
commutative $A$-algebra.  There are Adams operations $\psi^r
\colon \AN_{e}^{S^{1}} R \to \AN_{e}^{S^{1}} R$.  When $r$ is prime to
$p$, the operation $\psi^{r}$ is compatible with the restriction and
Frobenius maps on the $p$-cyclotomic spectrum $THH(R)$ and so induces
a corresponding operation on $TR(R)$ and $TC(R)$.
\end{theorem}



We have organized the paper to contain a brief review with references
to much of the background needed here.  Section~\ref{sec:background}
is mostly review of \cite{MM} and \cite[App.~B]{HHR}, and
Section~\ref{sec:cyclosec} is in part a review of~\cite[\S 4]{BM}.  In
addition, the main results in Section~\ref{sec:THH} overlap
significantly with~\cite{Stolz}, although our treatment is very
different: we rely on~\cite{HHR} to study the absolute $S^1$-norm
whereas \cite{Stolz} directly analyzes the construction by using a
somewhat different model structure and focuses on the case of
commutative ring orthogonal spectra.

\subsection*{Acknowledgments.}  The authors would like to thank Lars
Hesselholt, Mike Hopkins, and Peter May for many helpful
conversations.  We thank Lars Hesselholt, Aaron Royer, and Ernie Fontes
for helping to identify serious errors in previous drafts.  We thank
Cary Malkiewich for many helpful suggestions regarding a previous
draft.  This project was made possible by the hospitality of AIM, the
IMA, MSRI, and the Hausdorff Research Institute for Mathematics at the
University of Bonn.

\section{Background on equivariant stable homotopy
theory}\label{sec:background}

In this section, we briefly review necessary details about the
category of orthogonal $G$-spectra and the geometric fixed point and
norm functors.  Our primary sources for this material are the
monograph of Mandell-May~\cite{MM} and the appendices to
Hill--Hopkins--Ravenel~\cite{HHR}.  See also~\cite[\S 2]{BM} for a
review of some of these details.  We begin with two subsections
discussing the point-set theory followed by two subsections on
homotopy theory and derived functors.

\subsection{The point-set theory of equivariant orthogonal
spectra}\label{sec:pointset}

Let $G$ be a compact Lie group.  We denote by
$\aT^{G}$ the category of based $G$-spaces and based $G$-maps.  The smash
product of $G$-spaces makes this a closed symmetric monoidal category,
with function object $F(X,Y)$ the based space of (non-equivariant)
maps from $X$ to $Y$ with the conjugation $G$-action.  In particular,
$\aT^{G}$ is enriched over $G$-spaces.  We will denote by $U$ a fixed
universe of $G$-representations~\cite[\S II.1.1]{MM}, by which we mean
a countable dimensional vector space with linear $G$-action and
$G$-fixed inner product that contains
$\bR^{\infty}$, is the sum of finite dimensional $G$-representations,
and that has the property that any $G$-representation that occurs in
$U$ occurs infinitely often.
Let $\aV^G(U)$ denote the set of finite dimensional $G$-inner product spaces which are isomorphic to a $G$-vector subspace of $U$.
%
Except in this
section, we always assume that $U$ is a \term{complete} $G$-universe,
meaning that all finite dimensional irreducible $G$-representations are
in $U$.  For $V$, $W$ in
$\aV^{G}(U)$, denote by $\sI_G(V,W)$ the space of (non-equivariant)
isometric isomorphisms $V \to W$, regarded as a $G$-space via
conjugation.  Let $\sI^{U}_{G}$ be the category enriched in
$G$-spaces with $\aV^{G}(U)$ as its objects and $\sI_G(V,W)$ as its
morphism $G$-spaces; we write just $\sI_{G}$ when $U$ is understood.

\begin{definition}[{\cite[II.2.6]{MM}}] An orthogonal $G$-spectrum is
a $G$-equivariant continuous functor $X \colon \sI_{G} \to \aT^{G}$
equipped with a structure map
\[
\sigma_{V,W} \colon X(V) \sma S^W \to X(V \oplus W)
\]
that is a natural transformation of enriched functors $\sI_{G}\times
\sI_{G}\to \aT^{G}$ and that is associative and unital in the obvious
sense.  A map of orthogonal $G$-spectra $X \to X'$ is a
natural transformation that commutes with the structure map.
\end{definition}

We denote the category of orthogonal $G$-spectra by $\Spec^{G}$.  When
necessary to specify the universe $U$, we include it in the notation
as $\Spec^{G}_{U}$.

The category of orthogonal $G$-spectra is enriched over based
$G$-spaces, where the $G$-space of maps consists of all natural
transformations (not just the equivariant ones).  Tensors and
cotensors are computed levelwise.  The category of orthogonal
$G$-spectra is a closed symmetric monoidal category with unit the
equivariant sphere spectrum $S_G$ (with $S_{G}(V)=S^{V}$).

For technical reasons, it is often convenient to give an equivalent
formulation of orthogonal $G$-spectra as diagram spaces.
Following~\cite[\S II.4]{MM}, we consider the category $\sJ_G$ which
has the same objects as $\sI_G$ but morphisms from $V$ to $W$ given by
the Thom space of the complement bundle of linear isometries from $V$
to $W$.

\begin{proposition}[{\cite[II.4.3]{MM}}]\label{prop:orthodesc}
The category $\Spec^{G}$ of orthogonal $G$-spectra is equivalent to
the category of $\sJ_G$-spaces, i.e., the continuous equivariant
functors from $\sJ_G$ to $\aT_G$.  The symmetric monoidal structure is
given by the Day convolution.
\end{proposition}

This description provides simple formulas for suspension spectra and
desuspension spectra in orthogonal $G$-spectra.

\begin{definition}[{\cite[II.4.6]{MM}}] For any finite-dimensional
$G$-inner product space $V$ we have the shift desuspension spectrum
functor
\[
F_V \colon \aT^{G} \to \Spec^{G}
\]
defined by
\[
(F_V A)(W) = \sJ_G(V,W) \sma A.
\]
This is the left adjoint to the evaluation functor which evaluates an
orthogonal $G$-spectrum at $V$.
\end{definition}

\begin{remark}
In \cite{HHR}, the desuspension spectrum $F_{V}S^{0}$ is denoted as
$S^{-V}$ and $F_{0}A$ is denoted as $\Sigma^{\infty}A$ in a nod to the
classical notation.  (They write $S^{-V}\sma A$ for $F_{V}A\iso
F_{V}S^{0}\sma A$.)
\end{remark}

Since the category $\Spec^{G}_{U}$ is symmetric monoidal under the smash
product, we have categories of associative and commutative monoids,
i.e., algebras over the monads $\bT$ and $\bP$ that create associative
and commutative monoids in symmetric monoidal categories (e.g.,
see~\cite[\S II.4]{EKMM} for a discussion).

\begin{notation}\label{notn:rings}
Let $\Ass^{G}$ and $\Com^{G}$ denote the categories of
associative and commutative ring orthogonal $G$-spectra.
\end{notation}

For a fixed object $A$ in $\Com^{G}$, there is an associated
symmetric monoidal category $A\Mod^{G}$ of $A$-modules in orthogonal
$G$-spectra, with product the $A$-relative smash product $\sma_{A}$.
As in Notation~\ref{notn:rings}, there are categories $A\Alg^{G}$ of
$A$-algebras, and $A\CAlg^{G}$ of commutative
$A$-algebras~\cite[III.7.6]{MM}.

We now turn to the description of various useful functors on
orthogonal $G$-spectra.  We begin by reviewing the change of universe
functors.  In contrast to the classical framework of
``coordinate-free'' equivariant spectra~\cite{LMS}, orthogonal
$G$-spectra disentangle the point-set and homotopical roles of the
universe $U$.  A first manifestation of this occurs in the behavior of
the point-set ``change of universe'' functors.

\begin{definition}[{\cite[V.1.2]{MM}}]\label{def:changeuniverse} For any pair of universes $U$
and $U'$, the point-set change of universe functor
\[
\aI_{U}^{U'} \colon \Spec^{G}_{U} \to \Spec^{G}_{U'}
\]
is defined by $\aI_{U}^{U'}X(V)=\sJ(\bR^{n},V)\sma_{O(n)}X(\bR^{n})$
for $V$ in $\aV^{G}(U')$, where $n=\dim V$.
\end{definition}

These functors are strong symmetric monoidal equivalences of
categories:

\begin{proposition}[{\cite[V.1.1,V.1.5]{MM}}] Given universes
$U,U',U''$,
\begin{enumerate}
\item $\aI_{U}^{U}$ is naturally isomorphic to the identity.
\item $\aI_{U'}^{U''}\circ \aI_{U}^{U'}$ is naturally isomorphic to
$\aI_{U}^{U''}$.
\item $\aI_{U}^{U'}$ is strong symmetric monoidal.
\end{enumerate}
\end{proposition}

We are particularly interested in the change of universe functors
associated to the universes $U$ and $U^{G}$.  The latter of these
universes is isomorphic to the standard trivial universe
$\mathbb{R}^\infty$.  Note that the category of orthogonal $G$-spectra
on $\bR^{\infty}$ is just the category of orthogonal spectra with
$G$-actions.

Given a subgroup $H < G$, we can regard a $G$-space $X(V)$ as an
$H$-space $\iota_H^* X(V)$.  The space-level construction gives rise
to a spectrum-level change-of-group functor.

\begin{definition}[{\cite[V.2.1]{MM}}]\label{def:changegroup}
For a subgroup $H < G$,
define the functor
\[
\iota_H^* \colon \Spec^{G}_{U} \to \Spec^{H}_{\iota_{H}^* U}
\]
by 
\[
(\iota_H^* X)(V) = \sJ_{H}(\bR^{n},V)\sma_{O(n)}\iota_H^* (X(\bR^{n}))
\]
for $V$ in $\aV^{H}(\iota_{H}^{*}U)$, where $n=\dim(V)$.
\end{definition}

As observed in \cite[V.2.1, V.1.10]{MM}, for $V$ in $\aV^{G}(U)$,
\[
(\iota_H^* X)(\iota_H^* V) \iso \iota_H^* (X(V)).
\]

In contrast to the category of $G$-spaces, there are two reasonable
constructions of fixed-point functors: the ``categorical'' fixed
points, which are based on the description of fixed points as
$G$-equivariant maps out of $G/H$, and the ``geometric'' fixed points,
which commute with suspension and the smash product (on the level of
the homotopy category).  Again, the description of orthogonal
$G$-spectra as $\sJ_G$-spaces in Proposition~\ref{prop:orthodesc}
provides the easiest way to construct the categorical and geometric
fixed point functors~\cite[\S V]{MM}.

For any normal $H \lhd G$, let $\sJ_G^H(U,V)$ denote the
$G/H$-space of $H$-fixed points of $\sJ_G(U,V)$.  Given any orthogonal
spectrum $X$, the collection of fixed points $\{X(V)^H\}$ forms a
$\sJ_G^H$-space.  We can turn this collection into a $\sJ_{G/H}$-space
in two ways. There is a functor $q \colon \sJ_{G/H} \to \sJ_G^H$ induced
by regarding $G/H$-representations as $H$-trivial
$G$-representations via the quotient map $G \to G/H$.


\begin{definition}[{\cite[\S V.3]{MM}}] For $H$ a normal subgroup of
$G$, the categorical fixed point functor
\[
(-)^H \colon \Spec^{G}_{U} \to \Spec^{G/H}_{U^{H}}
\]
is computed by regarding the $\sJ_{G}^H$-space $\{X(V)^H\}$ as a
$\sJ_{G/H}$-space via $q$.
\end{definition}

On the other hand, there is an equivariant continuous functor $\phi
\colon \sJ_G^H \to \sJ_{G/H}$ induced by taking a $G$-representation
$V$ to the $G/H$-representation $V^H$.

\begin{definition}[{\cite[\S V.4]{MM}}] For $H$ a normal subgroup
of $G$, let $\Fix^{H}$ denote the functor from orthogonal $G$-spectra
(=$\sJ_{G}$-spaces) to $\sJ_{G}^{H}$-spaces defined by
$(\Fix^{H}X)(V)=(X(V))^{H}$.  The geometric fixed point functor
\[
\Phi^H(-) \colon \Spec^{G}_{U} \to \Spec^{G/H}_{U^{H}}
\]
is constructed by taking $\Phi^H(X)$ to be the left Kan extension of
the $\sJ_G^H$-space $\Fix^{H} X$ along $\phi$.
\end{definition}

\begin{remark}
Hill--Hopkins--Ravenel \cite[B.190]{HHR} call the point-set geometric
fixed point functor ``the monoidal geometric fixed point functor'' and
define it using the coequalizer
\[
\xymatrix@C-1pc{
\bigvee\limits_{V,W<U}
\sJ_{G}^{H}(V,W)\sma F_{W^{H}}S^{0} \sma (X(V))^{H}
\ar@<.5ex>[r]\ar@<-.5ex>[r]
&\bigvee\limits_{V<U}F_{V^{H}}S^{0}\sma (X(V))^{H},
}
\]
derived from applying the geometric fixed point functor above to the
``tautological presentation'' of $X$:
\[
\xymatrix@C-1pc{
\bigvee\limits_{V,W<U}
\sJ_{G}(V,W)\sma F_{W}S^{0}\sma X(V)
\ar@<.5ex>[r]\ar@<-.5ex>[r]
&\bigvee\limits_{V<U}F_{V}S^{0}\sma X(V),
}
\]
noting that $\Phi^{H}F_{V}A\iso F_{V^{H}}A^{H}$ for a $G$-space $A$.
Although $\Phi^{H}$ does not preserve coequalizers in general, it does
preserve the coequalizers preserved by $\Fix^{H}$, and $\Fix^{H}$
preserves the canonical coequalizer diagram since it is levelwise
split.   Thus, the definition above agrees with the definition
in~\cite[B.190]{HHR}.
\end{remark}

Both fixed-point functors are lax symmetric monoidal~\cite[V.3.8,
V.4.7]{MM} and so descend to categories of associative and commutative
ring orthogonal $G$-spectra. 

\begin{proposition}\label{prop:fixsymmon}
Let $H \lhd G$ be a normal subgroup.  Let $X$ and $Y$ be orthogonal
$G$-spectra.  There are natural maps
\[
\Phi^H X \sma \Phi^H Y \to \Phi^H (X \sma Y) \qquad \textrm{and}
\qquad X^H \sma Y^H \to (X \sma Y)^H
\]
that exhibit $\Phi^H$ and $(-)^H$ as lax symmetric monoidal functors.

Therefore, there are induced functors
\[
\Phi^H, (-)^H \colon \Ass^{G} \to \Ass^{G/H}
\]
and
\[
\Phi^H, (-)^H \colon \Com^{G} \to \Com^{G/H}.
\]
\end{proposition}

For a commutative ring orthogonal $G$-spectrum $A$, a corollary of
Proposition~\ref{prop:fixsymmon} is that the fixed-point functors
interact well with the category of $A$-modules.

\begin{corollary}
Let $A$ be a commutative ring orthogonal $G$-spectrum.  The
fixed-point functors restrict to functors
\[
\Phi^H \colon A\Mod^{G} \to ({\Phi^H A})\Mod^{G/H}
\]
and
\[
(-)^H \colon A\Mod^{G} \to A^{H}\Mod^{G/H}.
\]
\end{corollary}

\begin{remark}
We can extend these constructions to subgroups $H < G$ that are
not normal by considering the normalizer $NH$ and quotient $WH = NH /
H$.  However, since we do not need this generality herein, we do not
discuss it further.
\end{remark}

Let $z\in G$ be an element in the center of $G$.  Then multiplication
by $z$ is a natural automorphism on objects of
$\Spec^{G}_{\bR^{\infty}}$ or on objects of
$A\Mod^{G}_{\bR^{\infty}}$. In particular, it will induce a natural automorphism $\aI_{\bR^{\infty}}^{U}z$
 of $N_{H}^{G}X$ or of $\AN_{H}^{G}X$, as described in 
Sections~\ref{sec:THH} and~\ref{sec:amodcyc}.

\begin{proposition} \label{prop:diagdiag}
Let $z$ be an element in the center of $G$, and $K$ a normal
subgroup. Then for any $X \in \Spec^{G}_{\bR^{\infty}}$, we have an
identification
\[
  \Phi^K(\aI_{\bR^{\infty}}^{U}z) = \aI_{\bR^{\infty}}^{U^{K}}\bar z
\]
where $\bar z=zK\in G/K$.
In particular, for $z \in K$ the map $\Phi^K(\aI_{\bR^{\infty}}^{U}z)$
is the identity.
\end{proposition}

\begin{proof}
Using the tautological presentation of $\aI_{\bR^{\infty}}^{U}X$ and naturality, it suffices to verify this identity on orthogonal
spectra of the form $F_V Y$ for a $G$-representation $V \in \aV^{G}(U)$; on
such spectra, the map $\aI_{\bR^\infty}^{U} z\colon F_V Y \to F_V Y$ is
given by $f \sma y \mapsto (f \circ z^{-1}) \sma (z\cdot y)$. The
result follows from the fact the fixed point functor $(-)^K$ takes
multiplication by $z$ to multiplication by $\bar z$, and the functor
$\sJ_G^K \to \sJ_{G/K}$ induces maps $\sJ_G^K(V,V) \to
\sJ_{G/K}(V^K,V^K)$ taking $z$ to $\bar z$.
\end{proof}

\subsection{The point-set theory of the norm}

Central to our work is the realization by Hill, Hopkins, and
Ravenel~\cite{HHR} that a tractable model for the ``correct''
equivariant homotopy type of a smash power can be formed as a
point-set construction using the point-set change of universe
functors. It is ``correct'' insofar as there is a diagonal map which
induces an equivalence onto the geometric fixed points (see
Section~\ref{sec:horev} below).  They refer to this construction as
the norm after the norm map of Greenlees-May~\cite{GreenleesMayLoc},
which in turn is named for the norm map of Evens in group
cohomology~\cite[Chapter~6]{Evens}.

The point of departure for the construction of the norm is the use of
the change-of-universe equivalences to regard orthogonal $G$-spectra
on any universe $U$ as $G$-objects in orthogonal spectra.  (Good
explicit discussions of the interrelationship can be found in~\cite[\S
V.1]{MM} and~\cite[2.7]{schwequi}.)  We now give a point-set
description of the norm following~\cite{schwequi} and~\cite{bohmann};
these descriptions are equivalent to the description of~\cite[\S
A.3]{HHR} by the work of~\cite{bohmann}.

For the construction of the norm, it is convenient to use $\aB G$ to
denote the category with one object, whose monoid of endomorphisms is
the finite group $G$.  The category $\Spec^{\aB G}$ of functors from
$\aB G$ to the category $\Spec$ of 
(non-equivariant) orthogonal spectra indexed on the universe
$\bR^{\infty}$ is isomorphic to the category
$\Spec^{G}_{\bR^{\infty}}$ of orthogonal $G$-spectra indexed on the
universe $\bR^{\infty}$.  We can then use the change of universe functor
$\aI_{\bR^{\infty}}^{U}$ to give an equivalence of $\Spec^{\aB G}$
with the category $\Spec^{G}_{U}$ of orthogonal 
$G$-spectra indexed on $U$. 

\begin{definition}\label{def:indexedsmash}
Let $G$ be a finite group and $H < G$ be a finite index subgroup
with index $n$.  Fix an ordered set of coset representatives
$(g_{1},\dotsc,g_{n})$, and let $\alpha \colon G\to \Sigma_{n}\wr H$
be the homomorphism 
\[
\alpha(g)= (\sigma, h_1, \ldots, h_n)
\]
defined by the relation $g g_i =  g_{\sigma(i)} h_i$.
The indexed smash-power functor 
\[
\sma_H^G \colon \Spec^{\aB H} \to \Spec^{\aB G} 
\]
is defined as the composite
\[
\xymatrix{
\Spec^{\aB H} \ar[r]^-{\sma^{n}} & \Spec^{\aB (\Sigma_n \wr
H)} \ar[r]^-{\alpha^*} & \Spec^{\aB G}.
}
\]

The norm functor
\[
N_H^G \colon \Spec^{H}_{U} \to \Spec^{G}_{U'}
\]
is defined to be the composite
\[
X \mapsto \aI_{\mathbb{R}^{\infty}}^{U'} (\sma_H^G (\aI_{U}^{\mathbb{R}^{\infty}} X)).
\]
\end{definition}

This definition depends on the choice of coset representatives;
however, any other choice gives a canonically naturally isomorphic
functor (the isomorphism induced by permuting factors and multiplying
each factor by the appropriate element of $H$).  As observed
in~\cite[A.4]{HHR}, in fact it is possible to give a description of the
norm which is independent of any choices and is determined instead by
the universal property of the left Kan extension.  Alternatively,
Schwede~\cite[9.3]{schwequi} gives another way of avoiding the choice
above, using the set $\langle
G:H \rangle$ of all choices of ordered sets of coset representatives; 
$\langle
G:H \rangle$ is a free transitive $\Sigma_{n}\wr H$-set and the
inclusion of $(g_{1},\dotsc,g_{n})$ in $\langle G:H \rangle$ induces
an isomorphism 
\[
\sma_H^G X \iso \left<G : H\right>_+ \sma_{\Sigma_n \wr H} X^{\sma n}.
\]
In our work, $G$ will be the cyclic group $C_{nr}<S^{1}$ and $H=C_{r}$
(usually for $r=1$), and we have the obvious choice of coset
representatives $g_{k}=e^{2\pi (k-1)i/nr}$, letting us take advantage of
the explicit formulas.  In the case $r=1$, we have the following.

\begin{proposition}\label{prop:pointsetnorm}
Let $G$ be a finite group and $U$ a complete $G$-universe.  The norm
functor
\[
N_e^G \colon \Spec \to \Spec^{G}_{U}
\]
is given by the composite 
\[
X \mapsto \aI_{\bR^{\infty}}^{U} X^{\sma G},
\]
where $X^{\sma G}$ denotes the smash power indexed on the set $G$.
\end{proposition}

When dealing with commutative ring orthogonal $G$-spectra, the norm
has a particularly attractive formal description~\cite[A.56]{HHR}, which
is a consequence of the fact that the norm is a symmetric monoidal
functor.

\begin{theorem}\label{thm:normadj}
Let $G$ be a finite group and let $H$ be a subgroup of $G$.  The norm
restricts to the left adjoint in the adjunction 
\[
N_H^G \colon \Com^{H} \leftrightarrows \Com^{G} \colon \iota^*_{H},
\]
where $\iota^*_{H}$ denotes the change of group functor along $H<G$.
\end{theorem}

The relationship of the norm with the geometric fixed point functor is
encoded in the diagonal map~\cite[B.209]{HHR}.

\begin{proposition}\label{prop:diag}
Let $G$ be a finite group, $H < G$ a subgroup, and $K \lhd G$ a normal subgroup. Let $X$ be an orthogonal $H$-spectrum. Then there is a natural diagonal map of orthogonal $G/K$-spectra
\[
 \Delta \colon N_{HK/K}^{G/K} \Phi^{H \cap K} X \to \Phi^K N_H^G X.
\]
(Here we suppress the isomorphism $H/H \cap K \cong HK/K$ from the notation.)
\end{proposition}

\begin{proof}
The construction of $\Delta$ is the same as \cite[Proposition
B.209]{HHR} after generalizing the corresponding space-level
diagonal.   To do this, first note that for any based $H$-space $A$,
there is a natural isomorphism 
\[
 \Delta \colon N_{HK/K}^{G/K} A^{H \cap K} \overto{\cong} (N_H^G A)^K.
\]
For this, it is convenient to model the space-level norm as
follows. The space $N_H^G A$ is isomorphic to the subspace of tuples
$a = (a_g)_{g \in   G} \in \bigwedge_{g \in G} A$ such that $a_{hg} =
h a_g$. The left $G$-action is given by $(k \cdot a)_g = a_{gk}$.

Under this identification, $N_{HK/K}^{G/K} A^{H \cap K}$ consists of
tuples $b = (b_{[g]})_{[g] \in G/K}$ of elements in $A^{H \cap K}$
such that $b_{[hg]} = h b_{[g]}$ for $h \in H$. Similarly, $(N_{H}^{G}
A)^K$ consists of tuples $a = (a_{g})_{g \in G}$ such that $a_{hg} = h
a_g$ for $h \in H$ and $a_{gk} = a_g$ for $k \in K$. This allows us to
define the bijection $\Delta$ by $(\Delta b)_g = b_{[g]}$.
\end{proof}

For any particular commutative ring orthogonal spectrum $A$, the
indexed smash-power construction of Definition~\ref{def:indexedsmash}
can be carried out in the symmetric  monoidal category $A\Mod$.
Denote the $A$-relative indexed smash-power by $(\sma_A)^G_H$. For $X$ an $A$-module with $H$-action, we
understand $(\sma_A)^G_H X$ to be 
\[
(\sma_A)^G_H X:=\alpha^{*}X^{\sma n},
\]
where the $n$th smash power is over
$A$ and $\alpha^{*}$ is as in Definition~\ref{def:indexedsmash}.  This  
is an $A$-module (in $\Spec^{G}_{\bR^{\infty}}$).  
We then have the following definition of the $A$-relative norm
functor: 

\begin{definition}
Let $A$ be a commutative ring orthogonal spectrum.  Write $A_H$ for
the commutative ring orthogonal $H$-spectrum $\aI_{\bR^{\infty}}^{U}A$
obtained by regarding $A$ (with trivial $H$-action) as an object of
$\Spec^{\aB H}$ and applying the change of universe functor, and
similarly for $A_G$.
The $A$-relative norm functor 
\[
\AN_H^G \colon A_{H}\Mod^{H}_{U}\to A_{G}\Mod^{G}_{U'}
\]
is defined to be the composite
\[
X \mapsto \aI_{\bR^{\infty}}^{U'} ((\sma_A)_H^G (\aI_U^{\bR^\infty} X)).
\]
\end{definition}

The theory of the diagonal map in the $A$-relative context is somewhat
more complicated than in the absolute setting; we explain the details
in Section~\ref{sec:amodcyc}.

\subsection{Homotopy theory of orthogonal spectra}\label{sec:horev}

We now review the homotopy theory of orthogonal $G$-spectra with a
focus on discussing the derived functors associated to the point-set
constructions of the preceding section.  We begin by reviewing the
various model structures on orthogonal $G$-spectra.  All of these
model structures are ultimately derived from the standard model
structure on $\aT^{G}$ (the category of based $G$-spaces), which we begin
by reviewing. 

Following the notational conventions of~\cite{MM}, we start with the
sets of maps
\[
I = \{(G/H \times S^{n-1})_+ \to (G/H \times D^n)_+\}
\]
and
\[
J = \{(G/H \times D^n)_+ \to (G/H \times (D^n \times
I))_+\},
\] 
where $n\geq 0$ and $H$ varies over the closed subgroups of $G$.
Recall that there is a compactly generated model structure on the
category $\aT^{G}$ in which $I$ and $J$ are the generating cofibrations
and generating acyclic cofibrations (e.g.,~\cite[III.1.8]{MM}).  The
weak equivalences and fibrations are the maps $X \to Y$ such that $X^H
\to Y^H$ is a weak equivalence or fibration for each closed $H <
G$.  Transporting this structure levelwise in $\aV^{G}(U)$, we get the
level model structure in orthogonal $G$-spectra.

\begin{proposition}[{\cite[III.2.4]{MM}}]
Fix a $G$-universe $U$.
There is a compactly generated model structure on $\Spec^{G}_{U}$ in
which the weak equivalences and fibrations are the maps $X \to Y$ such
that each map $X(V) \to Y(V)$ is a weak equivalence or fibration of
$G$-spaces.  The sets of generating cofibrations and acyclic
cofibrations are given by $I_{G}^{U} = \{F_V i\mid i\in I\}$ and
$J_{G}^{U}=\{F_V j\mid j\in J\}$, where $V$ varies
over $\aV^{G}(U)$.
\end{proposition}

The level model structure is primarily scaffolding to construct the
stable model structures.  In order to specify the weak equivalences in
the stable model structures, we need to define equivariant homotopy
groups.   

\begin{definition}
Fix a $G$-universe $U$.
The homotopy groups of an orthogonal $G$-spectrum $X$ are defined for
a subgroup $H < G$ and an integer $q$ as
\[
\pi_q^H(X) =
\begin{cases}
\quad\displaystyle 
\lcolim_{V < U} \pi_{q}((\Omega^{V}X(V))^{H})&q\geq 0\\[1em]
\quad\displaystyle 
\lcolim_{\bR^{-q} < V < U} \pi_{0}((\Omega^{V-\bR^{-q}}X(V))^{H})&q< 0,\\
\end{cases}
\]
(see~\cite[\S III.3.2]{MM}).
\end{definition}

These are the homotopy groups of the underlying $G$-prespectrum
associated to $X$ (via the forgetful functor from orthogonal
$G$-spectra to prespectra).  We define the stable equivalences to be
the maps $X \to Y$ that induce isomorphisms on all homotopy groups.

\begin{proposition}[{\cite[III.4.2]{MM}}]
Fix a $G$-universe $U$.
The standard stable model structure on $\Spec^{G}_{U}$ is the compactly
generated symmetric monoidal model structure with the cofibrations
given by the level cofibrations, the weak equivalences the stable
equivalences, and the fibrations determined by the right lifting
property.  The generating cofibrations are given by $I_{G}^{U}$ as
above, and the generating acyclic cofibrations $K$ are the union of
$J_{G}^{U}$ and certain additional maps described in~\cite[III.4.3]{MM}.
\end{proposition}

This model structure lifts to a model
structure on the category $\Ass^{G}_{U}$ of associative monoids in
orthogonal $G$-spectra.

\begin{theorem}[{\cite[III.7.6.(iv)]{MM}}]
Fix a $G$-universe $U$.
There are compactly generated model structures on 
$\Ass^{G}_{U}$ in which the weak equivalences are the stable
equivalences of underlying orthogonal $G$-spectra indexed on $U$,
the fibrations are the maps which are stable fibrations of underlying
orthogonal $G$-spectra indexed on $U$, and the cofibrations are
determined by the left lifting property.
\end{theorem}

To obtain a model structure on commutative ring orthogonal spectra, we
also need the ``positive'' variant of the stable model structure.  We
define the positive level model
structures in terms of the generating cofibrations $(I_G^U)^+ \subset
I_G^U$ and $(J_G^U)^+ \subset J_G^U$, consisting of those maps $F_V i$
and $F_V j$ such that the representation $V$ contains a
nonzero trivial representation; these also extend to a positive stable
model structure.

\begin{theorem}[{\cite[III.5.3]{MM}}]
Fix a $G$-universe $U$.
There are compactly generated model structures on $\Com^{G}_{U}$ in
which the weak equivalences are the stable equivalences of the
underlying orthogonal $G$-spectra, the fibrations are
the maps which are positive stable fibrations of underlying
orthogonal $G$-spectra indexed on $U$, and 
the cofibrations are determined by the left lifting
property.
\end{theorem}

We will also use a variant of the standard stable model structure that
can be more convenient when working with the derived functors of the
norm.  We refer to this as the positive complete stable model
structure.  See~\cite[\S B.4]{HHR} for a comprehensive discussion of
this model structure, and~\cite[\S A]{Ullman} for a brief review.  In
order to describe this, denote by $(I^{\iota^*_{H}U}_H)^+$ and
$(J^{\iota^*_{H}U}_H)^+$ generating
cofibrations for the positive stable model structure on orthogonal
$H$-spectra indexed on the universe $\iota^{*}_{H}U$.

\begin{theorem}[{\cite[B.63]{HHR}}]\label{thm:normstable}
Fix a $G$-universe $U$.
There is a compactly generated symmetric monoidal model structure on
$\Spec^{G}$ with generating cofibrations and acyclic cofibrations the
sets $\{G_+ \sma_{H} i\mid  i\in (I^{\iota^{*}_{H}U}_H)^+,\; H<G\}$ and
$\{G_+ \sma_{H} j\mid j\in (J_H^{\iota^{*}_{H}U})^+,\; H<G\}$
respectively. 
The weak equivalences are the stable equivalences, and the fibrations
are determined by the right lifting property.
\end{theorem}

We then have corresponding positive complete model structures for
$\Com^{G}$ and $\Ass^{G}$.

\begin{theorem}[{\cite[B.130]{HHR}}, {\cite[B.136 (0908.3724v3)]{HHR}}]
\label{thm:normstablemult}
Fix a $G$-universe $U$.
There are compactly generated model structures on $\Ass^{G}_U$ and $\Com^{G}_{U}$ in
which the weak equivalences are the stable equivalences of the
underlying orthogonal $G$-spectra, the fibrations are
the maps which are positive complete stable fibrations of underlying
orthogonal $G$-spectra indexed on $U$, and 
the cofibrations are determined by the left-lifting
property.
\end{theorem}

For a fixed object $A$ in $\Com^{G}_U$, there are also lifted
model structures on the categories $A\Mod^{G}_{U}$ of $A$-modules,
$A\Alg^{G}_{U}$ of $A$-algebras, and $A\CAlg^{G}_{U}$ of
commutative $A$-algebras in both the stable and positive complete stable model
structures (\cite[III.7.6]{MM} and~\cite[B.137]{HHR}).  There are also
lifted model structures on the category $A\Mod^{G}_{U}$ of $A$-modules
when $A$ is an object of $\Ass^{G}_{U}$, but we will not need
these. Part of the following is \cite[B.137]{HHR}; the rest follows by
standard arguments.

\begin{theorem}
Fix a $G$-universe $U$.  Let $A$ be a commutative ring orthogonal
$G$-spectrum indexed on $U$.  There are compactly generated model
structures on the categories $A\Mod^{G}_{U}$ and $A\Alg^{G}_{U}$ 
in which the fibrations and weak equivalences are created by
the forgetful functors to the stable, complete stable, and positive complete stable
model structures on $\Spec^{G}_{U}$.  There are compactly generated
model structures on $A\CAlg^{G}_{U}$ in which the fibrations
and weak equivalences are created by the forgetful functors to the
positive stable and positive complete stable model structures on
$A\Mod^{G}_{U}$.
\end{theorem}

Finally, when dealing with cyclotomic spectra, we need to use variants
of these model structures where the stable equivalences are determined
by a family of subgroups of $G$.  Recall from~\cite[IV.6.1]{MM} the
definition of a family: a family $\aF$ is a collection of closed
subgroups of $G$ that is closed under taking closed subgroups (and
conjugation).  We say a map $X\to Y$ is an $\aF$-equivalence if it
induces an isomorphism on homotopy groups $\pi^{H}_{*}$ for all $H$ in
$\aF$.  All of the model structures described above have analogues
with respect to the $\aF$-equivalences (e.g., see~\cite[IV.6.5]{MM}),
which are built from sets $I$ and $J$ where the cells $(G / H
\times S^{n-1})_+ \to (G/H \times D^{n})_+$ and $ (G / H \times
D^{n})_+ \to  (G/H \times (D^{n}\times I))_+$ are restricted to $H
\in \aF$.  We record the situation in the following omnibus theorem.

\begin{theorem}\label{thm:fmodel}
There are stable, positive stable, and positive complete stable compactly generated
model structures on the categories $\Spec^{G}_{U}$ and
$\Ass^{G}_{U}$ where the weak equivalences are the
$\aF$-equivalences.  There are positive stable and positive complete stable compactly
generated model structures on the category $\Com^{G}_{U}$ where
the weak equivalences are the $\aF$-equivalences.

Let $A$ be a commutative ring orthogonal $G$-spectrum.  There are
stable, positive stable, and positive complete stable compactly generated model
structures on the categories $A\Mod^{G}_{U}$, $A\Alg^{G}_{U}$
where the weak equivalences are the $\aF$-equivalences.  There are positive stable and
positive complete stable compactly generated model structures on
$A\CAlg^{G}_{U}$ where the weak equivalences are the
$\aF$-equivalences. 
\end{theorem}

We are most interested in case of $G = S^1$ and the families $\aFin$ of 
finite subgroups of $S^1$ and $\aF_{p}$ of $p$-subgroups $\{C_{p^n}\}$
of $S^1$ for a fixed prime $p$.

\subsection{Derived functors of fixed points and the norm}

We now discuss the use of the model structures described in the
previous section to construct the derived functors of the categorical
fixed point, geometric fixed point, and norm functors.  We begin with
the categorical fixed point functor.  Since this is a right adjoint,
we have right-derived functors computed using fibrant replacement (in
any of our available stable model structures):

\begin{theorem}\label{thm:catder}
Let $H \lhd G$ be a normal subgroup.  Then the categorical fixed point
functor $(-)^{H}\colon \Spec^{G}_{U}\to \Spec^{G/H}_{U^{H}}$ is a
Quillen right adjoint; in particular, it preserves fibrations and weak
equivalences between fibrant objects in the stable and positive
complete stable model structures on $\Spec^{G}_{U}$.
\end{theorem}

As the fibrant objects in the model structures on associative and
commutative ring orthogonal spectra are fibrant in the underlying model
structures on orthogonal $G$-spectra, we can derive the categorical
fixed points by fibrant replacement in any of the settings in which we
work.

In contrast, the geometric fixed point functor admits a Quillen left derived
functor (see~\cite[V.4.5]{MM} and~\cite[B.197]{HHR}).

\begin{theorem}\label{thm:geoder}
Let $H$ be a normal subgroup of $G$.  The functor $\Phi^H(-)$
preserves cofibrations and weak equivalences between cofibrant objects
in the stable, positive stable, and positive complete stable model structures on
$\Spec^{G}_{U}$.
\end{theorem}

Since the cofibrant objects in the lifted model structures on
$\Ass^{G}_{U}$ are cofibrant when regarded as objects in
$\Spec^{G}_{U}$ \cite[III.7.6]{MM}, an immediate corollary of Theorem~\ref{thm:geoder} is
that we can derive $\Phi^H$ by cofibrant replacement when working with
associative ring orthogonal $G$-spectra.  In contrast, the underlying
orthogonal $G$-spectra associated to cofibrant objects in
$\Com^{G}$, in either of the model structures we study, are
essentially never cofibrant and the point-set functor $\Phi^{G}$ does
not always agree on these with the geometric fixed point functor on
the equivariant stable category.

The first part of the following theorem is \cite[B.104]{HHR}; the
statement in the case of $A$-modules is similar and discussed in
Section~\ref{sec:relTHH}.

\begin{theorem}\label{thm:derivenorm}
The norm $N_H^G(-)$ preserves weak equivalences between cofibrant
objects in any of the various stable model structures on $\Spec^{H}$,
$\Ass^{H}$, and $\Com^{H}$.

Let $A$ be a commutative ring orthogonal spectrum.  Then the
$A$-relative norm $\AN_e^G(-)$ preserves weak equivalences between
cofibrant objects in any of the various stable model structures on
$A\Mod$, $A\Alg$, and $A\CAlg$. 
\end{theorem}

The utility of the positive complete model structure is the following
homotopical version of Theorem~\ref{thm:normadj}~\cite[B.135]{HHR}.

\begin{theorem}
Let $H$ be a subgroup of $G$.  The adjunction
\[
N_H^G \colon \Com^{H} \leftrightarrows \Com^{G} \colon \iota^*_{H}
\]
is a Quillen adjunction for the positive complete stable structures. 
\end{theorem}

Finally, we have the following result about the derived version of the
diagonal map~\cite[B.209]{HHR}.  We note the strength of the conclusion:
the diagonal map is an isomorphism on cofibrant objects, not just a
weak equivalence.  

\begin{theorem}[{\cite[B.209]{HHR}}]\label{thm:normisobasic}
Let $H$ be a normal subgroup of $G$.  The diagonal map
\[
\Delta \colon \Phi^H X \to \Phi^G N_H^G X
\]
is an isomorphism of orthogonal spectra (and in particular a weak
equivalence) when $X$ is cofibrant in any of the stable model
structures on $\Spec^{H}$, or when $X$ is a cofibrant object in
$\Ass^{H}$.
\end{theorem}

Along the lines of Proposition~\ref{prop:diag}, we also need the
following more general statement, which essentially follows from the
argument of \cite[B.209]{HHR} using the isomorphism given in the proof
of Proposition~\ref{prop:diag} to start the induction.

\begin{theorem}\label{thm:normiso}
Let $G$ be a finite group, $H < G$ a subgroup, and $K \lhd G$ a normal
subgroup. Let $X$ be an orthogonal $H$-spectrum.  The diagonal map of
orthogonal $G/K$-spectra 
\[
 \Delta \colon N_{HK/K}^{G/K} \Phi^{H \cap K} X \to \Phi^K N_H^G X.
\]
is an isomorphism of orthogonal spectra (and in particular a weak
equivalence) when $X$ is cofibrant in any of the stable model
structures on $\Spec^{H}$ or when $X$ is a cofibrant object in
$\Ass^{H}$.
\end{theorem}

We also need the commutative ring orthogonal spectrum version of
Theorem~\ref{thm:normisobasic}. 

\begin{theorem}\label{thm:comnormisobasic}
The diagonal map
\[
\Delta \colon X \to \Phi^G N_e^G X
\]
is an isomorphism of orthogonal spectra when $X$ is a cofibrant
commutative ring orthogonal spectrum.
\end{theorem}

\begin{proof}
The induction in~\cite[B.209]{HHR}
and monoidality of both sides reduces the statement to
the case when $X=(F_{V}B_{+})^{(m)}/\Sigma_{m}$ where $V$ is a
finite-dimensional (non-equivariant) inner product space and $B$ is
the disk $D^{n}$ or sphere $S^{n-1}$---in particular, when $B$ is a
compact Hausdorff space. In general, for a (non-equivariant)
orthogonal spectrum $T$ the diagonal map is constructed
as follows: for every (non-equivariant) inner product space $Z$, the
universal property of the indexed smash product gives a map of based
$G$-spaces $N_{e}^{G}(T(Z))\to (N_{e}^{G}T)(\Ind_{e}^{G} Z)$, which restricts on
the diagonal to a map 
\begin{equation}\label{eq:fix}
T(Z)\to (N_{e}^{G}T(\Ind_{e}^{G} Z))^{G}=(\Fix^{G}(N_{e}^{G}T))(\Ind_{e}^{G} Z),
\end{equation}
and then (passing to the left Kan extension $\bP_{\phi}$ along the
fixed point functor $\phi \colon \sJ^{G}_{G}\to
\sJ_{e}$ on the right) induces a map
\[
T(Z)\to (\Phi^{G}(N_{e}^{G}T))((\Ind_{e}^{G} Z)^{G}) =(\Phi^{G}(N_{e}^{G}T))(Z).
\]
When $T$ is a cell of the form $F_{V}B_{+}$, the map in~\eqref{eq:fix}
factors as 
\begin{multline*}
T(Z)=\sJ_{e}(V,Z)\sma B_{+}\to
\sJ_{G}^{G}(\Ind_{e}^{G}V,\Ind_{e}^{G}Z)\sma B_{+}
\to \\
(\sJ_{G}(\Ind_{e}^{G}V,\Ind_{e}^{G}Z)\sma N_{e}^{G}(B)_{+})^{G}
=(\Fix^{G}(N_{e}^{G}T))(\Ind_{e}^{G} Z).
\end{multline*}
The first map $T(Z)=\sJ_{e}(V,Z)\sma B_{+}\to
\sJ_{G}^{G}(\Ind_{e}^{G}V,\Ind_{e}^{G}Z)\sma B_{+}$
induces an isomorphism 
\[
T\to \bP_{\phi}(\sJ_{G}^{G}(\Ind_{e}^{G}V,-)\sma B_{+})\iso 
\sJ_{e}((\Ind_{e}^{G}V)^{G},-)\sma B_{+}.
\]
By passing to quotients, we see that likewise in the case of interest, 
\[
T=X=(F_{V}B_{+})^{(m)}/\Sigma_{m}\iso F_{V^{m}}B^{m}_{+}/\Sigma_{m},
\]
the diagonal map factors as an isomorphism
\[
X\to \bP_{\phi}(\sJ_{G}^{G}(\Ind_{e}^{G}V^{m},-)\sma_{\Sigma_{m}} B^{m}_{+})\iso 
\sJ_{e}((\Ind_{e}^{G}V^{m})^{G},-)\sma_{\Sigma_{m}} B^{m}_{+}
\]
followed by a map
\[
\bP_{\phi}(\sJ_{G}^{G}(\Ind_{e}^{G}V^{m},-)\sma_{\Sigma_{m}}
B^{m}_{+})\to \Phi^{G}(N_{e}^{G} X)
\]
that is the induced map on left Kan extension from a map of $\sJ_{G}^{G}$-spaces
\[
\sJ_{G}^{G}(\Ind_{e}^{G}V^{m},-)\sma_{\Sigma_{m}}
B^{m}_{+} \to
(\sJ_{G}(\Ind_{e}^{G}V^{m},-)\sma_{\Sigma_{m}^{\times G}} N_{e}^{G}(B^{m})_{+})^{G}.
\]
Thus, it suffices to show that the latter map is an isomorphism.  This
amounts to showing that for each $G$-inner product space $W$, the map
\[
\sJ_{G}^{G}(\Ind_{e}^{G}V^{m},W)\sma_{\Sigma_{m}}
B^{m}_{+} \to
(\sJ_{G}(\Ind_{e}^{G}V^{m},W)\sma_{\Sigma_{m}^{\times G}} N_{e}^{G}(B^{m})_{+})^{G}
\]
is a homeomorphism, but since both sides are compact Hausdorff spaces,
it amounts to showing that the map is a bijection.  The map is clearly
an injection.  To see that it is a surjection, we note that any
non-basepoint $x$ of
$\sJ_{G}(\Ind_{e}^{G}V^{m},W)\sma_{N_{e}^{G}\Sigma_{m}}
N_{e}^{G}(B^{m})_{+}$
is represented by a collection of points $\vec b_{h}\in B^{m}$
(indexed on $h\in G$) and isometries $\phi_{h}\colon V^{m}\to
W$ (indexed on $h\in G$) such that $\bigoplus_{h}\phi_{h}\colon
\Ind_{e}^{G}V^{m}\to W$ is injective.  The point $x$ is $G$-fixed if
for every $g\in G$, there exist an element $\sigma(g)$ in
$N_{e}^{G}\Sigma_{m}$ such that 
\begin{equation}\label{eq:gfix}
g\cdot((\phi_{h}),(\vec b_{h}))=
((\phi_{h})\circ \sigma(g)^{-1},\sigma(g)\cdot(\vec b_{h})).
\end{equation}
If we write $\sigma(g)$ also in coordinates $\sigma(g)=(\sigma_{h}(g))$,
where 
\[
(\phi_{h})\circ \sigma(g)^{-1}=(\phi_{h}\circ \sigma_{h}(g)^{-1})
\qquad \text{and}\qquad 
\sigma(g)\cdot(\vec b_{h})=(\sigma_{h}(g)\cdot \vec b_{h}),
\]
then \eqref{eq:gfix} becomes
\begin{align*}
g\circ \phi_{g^{-1}h}&=\phi_{h}\circ \sigma_{h}(g)^{-1}\\
\vec b_{g^{-1}h}&=\sigma_{h}(g)\vec b_{h}.
\end{align*}
for all $g,h\in G$, 
where we have written $h\circ (-)$ to denote the action of $h$ on $W$
(and likewise we use $(-)\circ h$ below to denote the action of $h$
on $\Ind_{e}^{G}V^{m}$).  Let
\begin{align*}
\phi'_{h}&=h\circ \phi_{1}=\phi_{h}\circ \sigma_{h}(h)^{-1}\\
\vec b'_{h}&=\sigma_{h}(h)\cdot \vec b_{h}=\vec b_{1},
\end{align*}
Then $((\phi'_{h}),(\vec b'_{h}))$ also represents the element $x$,
with $(\vec b'_{h})$ clearly a diagonal element.  
Since 
\begin{align*}
(g\cdot \phi')_{h}&=(g\circ \phi'\circ g^{-1})_{h}\\
&=g\circ \phi'_{g^{-1}h}= g\circ g^{-1}h\circ \phi_{1}\\
&=h\circ \phi_{1}=\phi'_{h},
\end{align*}
we also have $(\phi'_{h})$ in the image of
$\sJ_{G}^{G}(\Ind_{e}^{G}V^{m},W)$. 
\end{proof}

\section{Cyclotomic spectra and topological cyclic homology}\label{sec:cyclosec}

In this section, we review the details of the category of $p$-cyclotomic
spectra and the construction of topological cyclic homology ($TC$).
The diagonal maps that naturally arise in the context of the norm go
in the opposite direction to the usual cyclotomic structure maps, and so we also
explain how to construct $TC$ from these ``op''-cyclotomic spectra.
In the following, fix a prime $p$ and a complete $S^1$-universe
$U$. 

\subsection{Background on $p$-cyclotomic spectra}\label{sec:cyclo}

In this section, we briefly review the point-set description of
$p$-cyclotomic spectra from~\cite[\S 4]{BM}; we refer the reader to
that paper for more detailed discussion.  

\begin{definition}[{\cite[4.5]{BM}}]\label{def:pcyclotomic}
A \term{$p$-precyclotomic spectrum} $X$ consists of an orthogonal
$S^1$-spectrum $X$ together with a map of orthogonal $S^1$-spectra 
\[
t_p \colon \rho_{p}^{*}\Phi^{C_{p}} X \to X.
\]
Here $\rho_{p}$ denotes the $p$-th root isomorphism $S^1 \to
S^1/C_{p}$. A $p$-precyclotomic spectrum is a \term{$p$-cyclotomic spectrum} when
the induced map on the derived functor
$\rho_{p}^{*}L\Phi^{C_{p}} X \to X$ is an $\aF_{p}$-equivalence.
A morphism of $p$-cyclotomic spectra consists of a map of
orthogonal $S^1$-spectra $X\to Y$ such that the diagram
\[
\xymatrix@-1pc{%
\rho_{p}^{*}\Phi^{C_{p}}X\ar[r]\ar[d]
&X\ar[d]\\
\rho_{p}^{*}\Phi^{C_{p}}Y\ar[r]
&Y
}
\]
commutes. 
\end{definition}

\begin{remark}
A \term{cyclotomic spectrum} is an orthogonal spectrum with
$p$-cyclotomic structures for all primes $p$ satisfying certain
compatibility relations; see~\cite[4.7--8]{BM} for details.
\end{remark}

Following \cite[5.4--5]{BM}, we have the following weak equivalences for
$p$-precyclotomic spectra.

\begin{definition}
A map of $p$-precyclotomic spectra is a \term{weak equivalence} when
it is an $\aF_{p}$-equivalence of the underlying orthogonal $S^{1}$-spectra.
\end{definition}

\begin{proposition}[{\cite[5.5]{BM}}]\label{prop:wkeqivcyc}
A map of $p$-cyclotomic spectra is a weak
equivalence if and only if is a weak equivalence of the underlying
(non-equivariant) orthogonal spectra.
\end{proposition}

\subsection{Constructing $TR$ and $TC$ from a cyclotomic spectrum}

In this section, we give a very rapid review of the definition of
$TR$ and $TC$ in terms of the point-set category of cyclotomic
spectra described above.  The interested reader is referred to the
excellent treatment in Madsen's CDM notes \cite{MadsenTraces} for more details on the
construction in terms of the classical (homotopical) definition of a
cyclotomic spectrum.

For a $p$-precyclotomic spectrum $X$, the collection $\{X^{C_{p^n}}\}$ of
(point-set) categorical fixed points is equipped with functors
\[
F, R \colon X^{C_{p^n}} \to X^{C_{p^{n-1}}}
\]
for all $n$, defined as follows.
The Frobenius maps $F$ are simply the obvious inclusions of
fixed points, and the restriction maps $R$ are constructed as
the composites
\[
X^{C_{p^n}} \cong (\rho_{p}^{*}X^{C_{p}})^{C_{p^{n-1}}} \xrightarrow{(\rho_p^* \omega )^{C_{p^{n-1}}}} 
(\rho_{p}^{*}\Phi^{C_p} X)^{C_{p^{n-1}}} \xrightarrow{(t_{p})^{C_{p^{n-1}}}}
X^{C_{p^{n-1}}},
\]
where the map $\omega$ is the usual map from categorical to geometric fixed points \cite[V.4.3]{MM}.
The Frobenius and restriction maps satisfy the identity $F\circ
R=R\circ F$. When $X$ is fibrant in the $\aF_{p}$-model structure (of
Theorem~\ref{thm:fmodel}), we then define
\[
TR(X) = \holim_{R} X^{C_{p^n}} \qquad \textrm{and} \qquad TC(X)
= \holim_{R,F} X^{C_{p^n}}.
\]

In general, we define $TR$ and $TC$ using a fibrant replacement that
preserves the $p$-precyclotomic structure; such a functor is provided by
the main theorems of~\cite[\S 5]{BM}, which construct model
structures on $p$-precyclotomic and $p$-cyclotomic spectra where the fibrations are the
fibrations of the underlying orthogonal $S^{1}$-spectra in the
$\aF_{p}$-model structure.  Alternatively, an explicit construction of
a fibrant replacement functor on orthogonal spectra that preserves
precyclotomic structures is given in~\cite[4.6--7]{BM2}.

\begin{proposition}[cf.~{\cite[1.4]{BM}}]
A weak equivalence $X \to Y$ of $p$-precyclotomic spectra induces weak
equivalences $TR(X_{f}) \to TR(Y_{f})$ and $TC(X_{f}) \to TC(Y_{f})$
of orthogonal spectra, where $(-)_{f}$ denotes any fibrant replacement
functor in $p$-cyclotomic spectra.
\end{proposition}

\begin{remark}
We do not yet have an abstract homotopy theory for multiplicative
objects in cyclotomic spectra, and the explicit fibrant replacement
functor $Q^{\aI}$ of \cite[4.6]{BM2} is lax monoidal but not lax
symmetric monoidal.  As a consequence, at present we do not know how
to convert a $p$-cyclotomic spectrum which is also a commutative ring
orthogonal $S^1$-spectrum into a cyclotomic spectrum that is a fibrant
commutative ring orthogonal $S^{1}$-spectrum.
\end{remark}

\subsection{Op-precyclotomic spectra}

For our construction of $THH$ based on the norm (in the next section),
the diagonal map $X\to \Phi^G N_e^G X$ is in the opposite direction of
the cyclotomic structure map needed in the definition of a cyclotomic
spectrum.  In the case when $X$ is cofibrant (or a cofibrant ring or
cofibrant commutative ring orthogonal spectrum), the diagonal map is
an isomorphism and so presents no difficulty; in the case when $X$ is
just of the homotopy type of a cofibrant orthogonal spectrum, the fact
that the structure map goes the wrong way necessitates some technical
maneuvering in order to construct $TR$ and $TC$.

\begin{definition}
An \term{op-$p$-precyclotomic spectrum} $X$ consists of an orthogonal
$S^1$-spectrum $X$ together with a map of orthogonal $S^1$-spectra  
\[
\gamma \colon X \to \rho_{p}^{*}\Phi^{C_{p}} X.
\]
An \term{op-$p$-cyclotomic spectrum} is an op-$p$-precyclotomic spectrum
where the structure map is an $\aF_{p}$-equivalence.  
A map of op-$p$-precyclotomic spectra
is a map of orthogonal $S^{1}$-spectra that commutes with the
structure map.  A map of op-$p$-precyclotomic spectra is a weak
equivalence when it is an $\aF_{p}$-equivalence of the underlying
orthogonal $S^{1}$-spectra. 
\end{definition}

Note that the definition above uses a condition on the point-set
geometric fixed point functor rather than the derived geometric fixed
point functor.  Such a definition works well  when we restrict to those
op-$p$-cyclotomic spectra $X$ where the canonical map in the
$S^{1}$-equivariant stable category $\rho_{p}^{*}L\Phi^{C_{p}}X\to
\rho^{*}_{p}\Phi^{C_{p}}X$ is an $\aF_{p}$-equivalence.  For
op-$p$-cyclotomic spectra in this subcategory, a 
map is a weak equivalence if and 
only if it is a weak equivalence of the underlying (non-equivariant)
orthogonal spectra.

Rather than study the category of op-$p$-precyclotomic spectra in detail,
we simply explain an approach to constructing $TR$ and $TC$ from this
data.  In what follows, let $(-)_{f}$ denote a fibrant replacement
functor in the $\aF_{p}$-model structure on orthogonal
$S^{1}$-spectra; to be clear, we assume the given natural
transformation $X\to X_{f}$ is always an acyclic cofibration.
Then for an op-$p$-precyclotomic spectrum $X$, we get a commutative
diagram 
\[
\xymatrix{%
X \ar[r]^-{\gamma} \ar@{ >->}[d]_{\simeq}
&\rho^{*}_{p}\Phi^{C_{p}}X\ar@{ >->}[d]_{\simeq}
\\
X_{f}\ar[r]_-{\gamma_{f}}&(\rho^{*}_{p}\Phi^{C_{p}}X)_{f}\ar[r]_-{\simeq}
&(\rho^*_{p} \Phi^{C_p}(X_{f}))_{f}
}
\]
where the bottom right horizontal map is a weak equivalence because
$\rho^{*}_{p}$ and $\Phi^{C_{p}}$ preserve acyclic cofibrations. In place of the restriction map $R$, we have a zigzag
\[
R\colon (X_{f})^{C_{p^{n}}}\to 
((\rho^{*}_{p}\Phi^{C_{p}}(X_{f}))_{f})^{C_{p^{n-1}}}
\from (X_{f})^{C_{p^{n-1}}}
\]
constructed as the following composite
\[
\xymatrix{
(X_{f})^{C_{p^n}} \ar[r]^-{\cong} 
&(\rho^{*}_{p}(X_{f})^{C_p})^{C_{p^{n-1}}} \ar[r]^-{\htp} 
& ((\rho^{*}_{p}(X_{f})^{C_p})_{f})^{C_{p^{n-1}}} \ar[dll] \\
((\rho^*_{p} \Phi^{C_p}(X_{f}))_{f})^{C_{p^{n-1}}} 
& \ar[l]^-{\htp} ((\rho^*_{p} \Phi^{C_p} X)_{f})^{C_{p^{n-1}}} 
& \ar[l] (X_{f})^{C_{p^{n-1}}}.\\
}
\]
We can use this as an analogue of $TR$.
\begin{definition}\label{defn:zTR}
Define $\opTR(X)$ as the homotopy limit of the diagram
\begin{multline*}
\cdots  \longleftarrow
(X_{f})^{C_{p^{n}}}\longrightarrow 
((\rho^{*}_{p}\Phi^{C_{p}}(X_{f}))_{f})^{C_{p^{n-1}}}
\longleftarrow(X_{f})^{C_{p^{n-1}}}
 \longrightarrow  \cdots \\
\cdots
\longleftarrow
(X_{f})^{C_{p}}\longrightarrow 
(\rho^{*}_{p}\Phi^{C_{p}}(X_{f}))_{f}
\longleftarrow X_{f}.
\end{multline*}
\end{definition}

The zigzags $R$ are compatible with the 
inclusion maps 
\[
F \colon (X_{f})^{C_{p^{n}}} \to (X_{f})^{C_{p^{n-1}}}
\]
in the sense that the following diagram commutes:
\[
\xymatrix@-1pc{%
(X_{f})^{C_{p^{n+1}}}\ar[r]\ar[dr]_{F}
&((\rho^{*}_{p}\Phi^{C_{p}}(X_{f}))_{f})^{C_{p^{n}}}\ar@{.>}[dr]^{F}
&\ar[l](X_{f})^{C_{p^{n}}}\ar[dr]^{F}\\
&(X_{f})^{C_{p^{n}}}\ar[r]
&((\rho^{*}_{p}\Phi^{C_{p}}(X_{f}))_{f})^{C_{p^{n-1}}}
&\ar[l](X_{f})^{C_{p^{n-1}}}\\
}
\]
We can therefore form an analogue of $TC$.

\begin{definition}\label{defn:zTC}
Define $\opTC(X)$ by taking the homotopy limit over the
diagram
\vtop{\vskip1.5em}
\[
\xymatrix{
\cdots \ar@/^1.5pc/[rr] \ar@/_1pc/[r] & 
\ar[l] (X_{f})^{C_{p^n}} \ar@/_1.5pc/[rr] \ar[r] & 
((\rho^{*}_{p}\Phi^{C_{p}}(X_{f}))_{f})^{C_{p^{n-1}}} \ar@/^1.5pc/[rr] &
(X_{f})^{C_{p^{n-1}}} \ar[r] \ar[l] \ar@/_1pc/[r] & \cdots \\
}\vtop{\vskip1.5em}
\]
where the middle parts are the $R$ zigzags and the top and bottom the
$F$ maps.
\end{definition}

This has the expected homotopy invariance property.

\begin{proposition}
Let $X\to Y$ be a weak equivalence of op-$p$-precyclotomic spectra.  The
induced maps $\opTR(X)\to \opTR(Y)$ and $\opTC(X)\to \opTC(Y)$ are
weak equivalences. 
\end{proposition}

Although we have nothing to say in general about the relationship
between $p$-cyclotomic spectra and op-$p$-cyclotomic spectra or
between $\opTC$ and $TC$, we have the following comparison result in
the case when $X$ has compatible $p$-cyclotomic and
op-$p$-precyclotomic structures.  This in particular applies when $X$
has the homotopy type of a cofibrant orthogonal spectrum, as we
explain in Section~\ref{sec:THH}.  We apply it in
Section~\ref{sec:amodcyc} to prove Theorem~\ref{thm:trace}.

\begin{proposition}\label{prop:zconsist}
Let $X$ be an op-$p$-precyclotomic spectrum and a $p$-cyclotomic spectrum
and assume that the composite of the two structure maps
\[
\rho^{*}_{p}\Phi^{C_{p}}X\to X\to \rho^{*}_{p}\Phi^{C_{p}}X
\] 
is homotopic to the identity.
Then there is a zig-zag of weak equivalences
connecting $TR(X)$ and $\opTR(X)$ and a 
zig-zag of weak equivalences connecting $TC(X)$ and $\opTC(X)$. 
\end{proposition}

\begin{proof}
In the case of the comparison of $TR(X)$ and
$\opTR(X)$, we can use a fibrant replacement of $X$ in the category of cyclotomic spectra to compute both $TR(X)$ and $\opTR(X)$. It follows that it suffices to show that the homotopy limits of
diagrams of fibrant objects of the form
\begin{equation}\label{eq:comptr1}
\xymatrix{
\ldots & \ar[l] Y_n \ar[r]^{f_n} & Y'_n &
Y_{n-1} \ar[l]_{g_n^{-1}} \ar[r] & \ldots 
}
\end{equation}
and
\begin{equation}\label{eq:comptr2}
\xymatrix{
\ldots \ar[r] & Y_n \ar[r]^{f_n} & Y'_n \ar[r]^{g_n} & Y_{n-1} \ar[r] & \ldots
}
\end{equation}
are equivalent, where $g_n$ is an equivalence and $g_n^{-1}\circ
g_{n}$ is homotopic to the identity.  This kind of rectification
argument is standard, although we are not sure of a place in the
literature where the precise fact we need is spelled out.  We argue as
follows.  Choosing a homotopy $H$ from the identity to
$g_{n}^{-1}\circ g_{n}$, we get strictly commuting diagrams of the
form
\[
\xymatrix@C=4pc{
Y_n \ar[r]^{f_n} & Y_n' & Y_{n-1} \ar[l]_{g_n^{-1}} \ar[r]^{\id} & Y_{n-1} \\
Y_n \ar[u]^{\id} \ar[d]_{\id} \ar[r]^-{f_n \times \{0\}} &
Y_n' \times I \ar[d]_{\pi_1} \ar[u]^{H} &
Y_n' \ar[d]_{\id} \ar[u]^{g_n} \ar[l]_-{\id \times \{1\}} \ar[r]^{g_n} &
Y_{n-1} \ar[u]^{\id} \ar[d]_{\id} \\
Y_n \ar[r]^{f_n} & Y_n'& Y_n' \ar[l]_{\id}  \ar[r]^{g_n} & Y_{n-1}. \\ 
}
\]
Note that all the vertical maps are weak equivalences, and therefore
the induced maps between the homotopy limits of the rows are both weak
equivalences.  The homotopy limit of the top row is weakly equivalent
to the homotopy limit of \eqref{eq:comptr1} and the homotopy limit of
the bottom row is weakly equivalent to the homotopy limit of
\eqref{eq:comptr2}.  This completes the comparison of $TR(X)$ and
$\opTR(X)$; the argument for comparing $TC(X)$ and $\opTC(X)$ is
analogous using ``ladders'' in place of rows.
\end{proof}

\begin{remark}
The following sketches a reformulation of the above argument, showing
the equivalence of homotopy limits of (\ref{eq:comptr1}) and
(\ref{eq:comptr2}), using the more general-purpose machinery of
coherent diagrams. All numbered references in the following are to
\cite{HTT}.

As homotopy limits are invariant up to equivalence, we can assume that
the objects in the diagram are cofibrant-fibrant and hence that $g_n$
is a homotopy equivalence. If $N(\Spec^\circ)$ denotes the
``simplicial nerve'' [1.1.5.5] of the simplicial category of
cofibrant-fibrant orthogonal spectra, homotopy limits can be computed
in the quasicategory $N(\Spec^\circ)$ [4.2.4.8].

There is a simplicial set $K$ whose $0$-simplices correspond to the
objects $Y_n$ and $Y_n'$, whose $1$-simplices correspond to the maps
$f_n$, $g_n$, and $g_n^{-1}$, and whose $2$-simplices express the
composition homotopies $g_n^{-1} \circ g_n \Rightarrow \id$. We have a
homotopy coherent diagram of orthogonal spectra indexed on $K$ in the
sense of Vogt (or [1.2.6]) expressed as follows:
\[
\xymatrix{
\cdots Y_{n+1} \ar[r]^{f_{n+1}} &
Y'_{n+1} \ar@/^1pc/[r]^{g_{n+1}} &
Y_n \ar@/^1pc/[l]^{g_{n+1}^{-1}} \ar@{=>}[l] \ar[r]^{f_n} &
Y'_n \ar@/^1pc/[r]^{g_{n}} &
Y_{n-1} \ar@/^1pc/[l]^{g_{n}^{-1}} \ar@{=>}[l] \ar[r]^{f_{n-1}} &
Y'_{n-1} \ar@/^1pc/[r]^{g_{n-1}} &
Y_{n-2} \cdots \ar@/^1pc/[l]^{g_{n-1}^{-1}} \ar@{=>}[l]
}
\]
We write $K^+$ for the upper subcomplex containing the edges $f_n$ and
$g_n$, and similarly write $K^-$ for the lower subcomplex containing
the $f_n$ and $g_n^{-1}$.

The inclusion $K^+ \to K$ is an iterated pushout along horn-filling
maps $\Lambda^2_0 \to \Delta^2$, so this map is left anodyne [2.0.0.3]
and hence final [4.1.1.3]. The restriction from $K$-diagrams to
$K^+$-diagrams therefore preserves all homotopy limits [4.1.1.8].

We now consider the inclusion $K^- \to K$, which is an iterated
pushout along horn-filling maps $\Lambda^2_2 \to \Delta^2$ whose last
edges are $g_n^{-1}$. Because the maps $g_n^{-1}$ are equivalences,
the space of extensions of a diagram indexed on $K^-$ to a diagram
indexed on $K$ is contractible because the map $\Lambda^2_2 \to
\Delta^2$, with the final edge marked as an equivalence, is marked
anodyne [3.1.1.1, 3.1.3.4]. In addition, the subspace of homotopy
right Kan extensions is also contractible [4.2.4.8,
4.3.2.15]. Therefore, any extension of this $K^-$-diagram to a
$K$-diagram is a homotopy right Kan extension, and the homotopy limit
of a homotopy right Kan extension is equivalent to the homotopy limit
of the original diagram [4.3.2.8].

The comparison between $TC$ and $\opTC$ follows by a similar
argument. There is a diagram indexed by $K \times \Delta^1$,
representing the natural transformation $F$ on the comparison diagram
for $TR$: we
define a simplicial set $L$ by identifying $K \times \{1\}$ with $K
\times \{0\}$ after a shift. There are subcomplexes $L^+$ and $L^-$,
generated by $K^+ \times \Delta^1$ and $K^- \times \Delta^1$
respectively, representing the diagrams defining $TR$ and $\opTR$. As
before, the inclusion $L^+ \to L$ is left anodyne and the inclusion
$L^- \to L$ only involves extension along equivalences.
\end{remark}

\section{The construction and homotopy theory of the
$S^1$-norm}\label{sec:THH} 

In this section, we construct the norm from the trivial group to
$S^{1}$ and study its basic point-set and homotopy properties.  In
particular, we prove that under mild hypotheses it gives a model for
$THH$ which is cyclotomic.  Unlike norms for finite groups, the
$S^{1}$-norm does not apply to arbitrary orthogonal spectra; instead
we need an associative ring structure. In the case when $R$ is
commutative, we identify the $S^{1}$-norm as the left adjoint of the
forgetful functor from commutative ring orthogonal $S^{1}$-spectra
indexed on a complete universe to (non-equivariant) commutative ring
orthogonal spectra.

Throughout this section, we fix a complete $S^1$-universe $U$.
As in the definition of the norm for finite groups, the (point-set)
equivalence of categories $\aI_{\bR^{\infty}}^{U}$ discussed in
Section~\ref{sec:pointset} will play a key technical role.

For a ring orthogonal spectrum $R$, let $N^{\cyc}_{\sma} R$
denote the cyclic bar construction with respect to the smash product;
i.e., the cyclic object in orthogonal spectra with $k$-simplices 
\[
[k] \to \underbrace{R \sma R \sma \ldots \sma R}_{k+1}
\]
and the usual cyclic structure maps induced from the ring structure on
$R$.

\begin{lemma}\label{lem:cycl}
Let $R$ be an object in $\Ass$.  Then the geometric realization of
the cyclic bar construction $|N^{\cyc}_{\sma} R|$ is naturally an
object in $\Spec^{S^1}_{\bR^{\infty}}$.
\end{lemma}

\begin{proof}
It is well known that the geometric realization of a cyclic space has
a natural $S^{1}$-action \cite[3.1]{JonesCyclic}.  Since geometric
realization of an orthogonal spectrum is computed levelwise, it
follows by continuous naturality that the geometric realization of a
cyclic object in orthogonal spectra has an $S^{1}$-action.  As noted
in Section~\ref{sec:pointset}, the category
$\Spec^{S^1}_{\bR^{\infty}}$ of orthogonal $S^{1}$-spectra indexed on 
$\bR^{\infty}$ is isomorphic to the category of orthogonal spectra
with $S^{1}$-actions.
\end{proof}

Using the point-set change of universe functors we can regard this as
indexed on the complete universe $U$.  The following definition
repeats Definition~\ref{defn:Tnorm} from the introduction.

\begin{definition}
Let $R$ be a ring orthogonal spectrum.  Define the functor
\[
N_e^{S^1} \colon \Ass \to \Spec^{S^1}_{U}
\]
to be the composite functor
\[
R \mapsto \aI_{\bR^{\infty}}^{U} |N^{\cyc}_\sma R|.
\]
\end{definition}

When $R$ is a commutative ring orthogonal spectrum, the usual
tensor homeomorphism~\cite[IX.3.3]{EKMM} 
\[
|N^{\cyc}_{\sma} R| \cong R \otimes S^1
\]
yields the following characterization:

\begin{proposition}\label{prop:commuT}
The restriction of $N_e^{S^1}$ to $\Com$ lifts to a functor
\[
N_e^{S^1} \colon \Com \to \Com^{S^{1}}_{U}
\]
that is left adjoint to the forgetful functor 
\[
\iota^{*}\colon \Com^{S^{1}}_{U}\to \Com.
\]
\end{proposition}

\begin{proof}
To obtain the refinement of $N_{e}^{S^{1}}$ to a functor 
$\Com \to \Com^{S^{1}}_{U}$, it suffices to construct a
refinement of $|N^{\cyc}_{\sma}|$ to a functor
\[
|N^{\cyc}_{\sma}|\colon \Com \to \Com^{S^1}_{\bR^{\infty}}.
\]
We obtain this immediately from the strong symmetric monoidal
isomorphism 
\[
|X\subdot|\sma |Y\subdot|\iso |X\subdot \sma Y\subdot|
\]
for simplicial objects $X\subdot$,$Y\subdot$ in orthogonal spectra and
the easy observation that the map is $S^{1}$-equivariant for cyclic
objects.  Indeed, using the isomorphism 
\[
\bP |X\subdot|\iso |\bP X\subdot|,
\]
we can identify $|N^{\cyc}_{\sma}\bP X|$ as $\bP(X\sma S^{1}_{+})$.
Now using the canonical reflexive coequalizer 
\[
\xymatrix@C-1pc{%
\bP \bP R\ar@<.5ex>[r]\ar@<-.5ex>[r]&\bP R\ar[r]&R
}
\]
we can identify $|N^{\cyc}_{\sma}R|$ as the reflexive coequalizer 
\[
\xymatrix@C-1pc{
\bP\bP (R \sma S^{1}_{+})\ar@<.5ex>[r]\ar@<-.5ex>[r]
&\bP (R\sma S^{1}_{+})\ar[r]&R\otimes S^{1},
}
\]
constructing the tensor of $R$ with the unbased space $S^{1}$ in the
category of commutative ring orthogonal spectra.  A formal argument
now identifies this as the left adjoint to the forgetful functor 
\[
\iota^* \colon \Com^{S^1}_{\bR^{\infty}} \to \Com
\]
and it follows that $N_{e}^{S^{1}}$ is the left adjoint to the
forgetful functor indicated in the statement.
\end{proof}

We now show that the $S^{1}$-norm $N_e^{S^1} R$ is a cyclotomic
spectrum in orthogonal $S^1$-spectra.  For this, we need to work with
the $C_{n}$ geometric fixed points. 
Since
$|N^{\cyc}_{\sma} R|$ is the geometric realization of a cyclic
spectrum, the $C_n$-action can be computed in terms of the edgewise
subdivision of the cyclic spectrum $N^{\cyc}_{\sma} R$~\cite[\S
1]{BHM}.  Specifically, the $n$th edgewise subdivision $\sd_n
N^{\cyc}_{\sma} R$ is a simplicial orthogonal spectrum with a
simplicial $C_n$-action such that there is a natural isomorphism of
orthogonal $S^1$-spectra
\[
|\sd_n N^{\cyc}_{\sma} R| \cong |N^{\cyc}_{\sma} R|,
\]
where the $S^1$-action on the left extends the $C_n$-action induced
from the simplicial structure.  For $N_{e}^{S^{1}}$ then, taking
$\tilde{U}=\iota^{*}_{C_{n}}U$, a complete $C_n$-universe, there is
an isomorphism of orthogonal 
$C_n$-spectra indexed on $\tilde{U}$  
\[
\iota_{C_n}^* N_e^{S^1}R \cong \aI_{\bR^{\infty}}^{\tilde{U}} (\iota_{C_n}^*
|N^{\cyc}_{\sma} R|). 
\]
This allows us to understand the $C_n$-action on $N_e^{S^1} R$ in
terms of the $C_n$-action on $|N^{\cyc}_{\sma} R|$.  

Writing this out, the orthogonal $C_n$-spectrum
$\iota^{*}_{C_{n}}N_{e}^{S^1}(R)$ has a description as the geometric
realization of a simplicial orthogonal $C_n$-spectrum having
$k$-simplices given by norms
\[
(N_e^{C_n} R)^{\sma
(k+1)} \cong \aI_{\bR^{\infty}}^{\tilde{U}}(R^{\sma n(k+1)}),
\] 
where $C_n$ acts by block permutation on $R^{\sma n(k+1)}$ and
$\tilde{U}=\iota^{*}_{C_n}U$ (for $U$ a complete $S^{1}$-universe).  The
faces are also given blockwise, with $d_{i}$ for $0\leq i\leq k-1$ the
map
\[
N_{e}^{C_{n}}(R^{\sma (k+1)})\to N_{e}^{C_{n}}(R^{\sma k})
\]
on norms induced by the multiplication of the $(i+1)$st and $(i+2)$nd
factors of $R$. The face map $d_{k}$ is a bit more complicated and
uses both an internal cyclic permutation inside the last
$N_{e}^{C_{n}}R$ factor (as in Proposition~\ref{prop:diagdiag}) and a
permutation of the $(k+1)$ factors of $(N_{e}^{C_{n}}R)^{\sma (k+1)}$
together with the multiplication $d_{0}$.  Writing $g=e^{2\pi i/n}$
for the canonical generator of $C_{n}<S^{1}$ and $\alpha$ for the
natural cyclic permutation on $X^{\sma (k+1)}$, $d_{k}$ is the
composite
\[
(N_{e}^{C_{n}}R)^{\sma (k+1)}
\overto{\id^{\sma k}\sma \aI_{\bR^{\infty}}^{\tilde U}g}
(N_{e}^{C_{n}}R)^{\sma (k+1)}
\overto{\alpha}
(N_{e}^{C_{n}}R)^{\sma (k+1)}
\overto{d_{0}}
(N_{e}^{C_{n}}R)^{\sma k}.
\]

In fact, we have the following concise description of the $C_n$-action
in $N_e^{S^1}$-bimodule terms.  We obtain a ($N^{C_n}_e R$,$N^{C_n}_e
R$)-bimodule ${^{g}N_e^{C_n} R}$, using the standard right action
but twisting the left action using $\aI_{\bR^{\infty}}^{\tilde U}g$.  In
the following statement, we use the cyclic bar construction with
coefficients in a bimodule, q.v.~\cite[\S 2]{BHM}.

\begin{theorem}\label{thm:sdcyc}
Let $R$ be a ring orthogonal spectrum.  For any $C_n < S^1$,
there is an isomorphism of orthogonal $C_n$-spectra
\[
\iota_{C_n}^* N_e^{S^1}(R) \cong |N^{\cyc}_{\sma} (N_e^{C_n} R,
{^{g}N_e^{C_n} R})|,
\]
where the cyclic bar construction is taken in the symmetric monoidal
category $\Spec^{C_n}_{\tilde U}$.
\end{theorem}

Next we assemble the diagonal maps into a map $N_e^{S^1}
R \to \rho^*_n \Phi^{C_n} N_e^{S^1} R$ of orthogonal $S^1$-spectra.
The following lemma (which is just a specialization of
Proposition~\ref{prop:diagdiag}) provides the basic compatibility we
need.  (The lemma also follows as an immediate consequence of the
much more general rigidity theorem of
Malkiewich~\cite[\S 3]{MalkiewichDX}.) 

\begin{lemma}\label{lem:multcompat}
Let $R$ be an orthogonal spectrum, let $H<S^1$ be a finite
subgroup, and let $h \in H$.  Then the map
$\Phi^H(\aI_{\bR^{\infty}}^{\tilde U}h)\colon \Phi^H N_e^H R \to \Phi^H
N_e^H R$ is the identity.
\end{lemma}

We now prove the main theorem about the diagonal map cyclotomic
structure.

\begin{theorem}\label{thm:cyclodiag}
Let $R$ be a ring orthogonal spectrum.  The 
diagonal maps 
\[
\Delta_n \colon R^{\sma(k+1)} \to \Phi^{C_n} N_e^{C_n} R^{\sma (k+1)}
\]
assemble into natural maps of $S^1$-spectra 
\[
 \tau_n \colon N_e^{S^1} R \to \rho_n^* \Phi^{C_n} \aI_{\bR^{\infty}}^{U}|N^{\cyc}_{\sma} R|\iso\rho_n^* \Phi^{C_n} N^{S^{1}}_{e}R.
\]
If $R$ is cofibrant or cofibrant as a commutative ring orthogonal
spectrum, then these maps are isomorphisms.
\end{theorem}

\begin{proof}
Varying $k$, we get a map of cyclic objects 
\[
N^{\cyc}_{\sma}R\to \Phi^{C_n} \aI_{\bR^{\infty}}^{\tilde{U}} \sd_n N^{\cyc}_{\sma} R
\]
and on realization and change of universe, a map
\[
N^{S^{1}}_{e}R\to
\aI_{\bR^{\infty}}^{U}|\Phi^{C_n} \aI_{\bR^{\infty}}^{\tilde{U}} \sd_n N^{\cyc}_{\sma} R|
\]
of orthogonal $S^{1}$-spectra.  The map $\tau_{n}$ is the composite
with the evident isomorphism of orthogonal $S^{1}$-spectra
\[
\aI_{\bR^{\infty}}^{U}|\Phi^{C_n} \aI_{\bR^{\infty}}^{\tilde{U}} \sd_n N^{\cyc}_{\sma} R|
\iso
\rho_n^* \Phi^{C_n} \aI_{\bR^{\infty}}^{U} |\sd_n N^{\cyc}_{\sma} R|
\iso\rho_n^* \Phi^{C_n} N^{S^{1}}_{e}R.
\]
When $R$ is cofibrant, the maps $\Delta_{n}$ are isomorphisms, and so
therefore are the maps $\tau_{n}$.
\end{proof}

The previous theorem establishes a precyclotomic structure.  For the
cyclotomic structure, we now just need to compare the pointset
geometric fixed point functors with their derived functors.

\begin{theorem}\label{thm:levelwisefp}
Let $R$ be a cofibrant ring orthogonal spectrum or a cofibrant
commutative ring orthogonal spectrum. Then for any $C_{n}<S^{1}$, the point-set
geometric fixed point functor on $N_{e}^{S^{1}}R$ computes the left derived geometric
fixed point functor
\[
L\Phi^{C_{n}}N_{e}^{S^{1}}R\overto{\simeq}
\Phi^{C_{n}}N_{e}^{S^{1}}R.
\]
Moreover, 
\[
\Phi^{C_{n}}N_{e}^{S^{1}}R\iso
\aI_{\bR^{\infty}}^{U}|\Phi^{C_{n}}\aI_{\bR^{\infty}}^{\tilde
U}\sd_{n}N^{\cyc}_{\sma}R|.
\]
\end{theorem}

Theorem~\ref{thm:maincyc}, the assertion of the cyclotomic structure
on $N_{e}^{S^{1}}R$ for $R$ a cofibrant ring orthogonal spectrum or
cofibrant commutative ring orthogonal spectrum, is now an immediate
consequence of the previous theorem and Theorem~\ref{thm:cyclodiag}.  
If $R$ only has the homotopy type of a cofibrant object, application of
Proposition~\ref{prop:zconsist} allows us to functorially work with
$\opTR$ and $\opTC$ as models of $TR$ and $TC$.

For the proof of the previous theorem, recall that a simplicial object
in a category enriched in spaces is said to be \term{proper} when for
each $n$ the map from the $k$th latching object to the $k$th level is an
$h$-cofibration.  (Recall that an $h$-cofibration is a map $f\colon
X\to Y$ with the homotopy extension property: Any map $\phi\colon Y\to
Z$ and any path in the space of maps from $X$ to $Z$ starting at $\phi
\circ f$ comes from the restriction of a path in the space of maps
from $Y$ to $Z$ starting at $\phi$.)  The geometric realization of a
proper simplicial object (in a topologically cocomplete category) is
the colimit of a sequence of pushouts of $h$-cofibrations.  This is
relevant to the situation above because of the following lemma.

\begin{lemma}\label{lem:proper}
Let $R$ be a cofibrant ring orthogonal spectrum or a cofibrant
commutative ring orthogonal spectrum. Then for any $C_{n}<S^{1}$, 
\[
\aI_{\bR^{\infty}}^{\tilde U}\sd_{n}N^{\cyc}_{\sma}R
\]
is proper as a simplicial object in $\Spec^{C_{n}}_{\tilde U}$.
\end{lemma}

\begin{proof}
Since $\aI_{\bR^{\infty}}^{\tilde U}$ is a topological left adjoint,
it preserves pushouts and homotopies, and therefore preserves
properness.  Thus, it suffices to show that 
\[
\sd_{n}N^{\cyc}_{\sma}R
\]
is a proper simplicial object in $\Spec^{C_{n}}_{\bR^{\infty}}$.  In
the case when $R$ is a cofibrant ring orthogonal spectrum, each level
is cofibrant as an orthogonal $C_{n}$-spectrum and the inclusion of
the latching object is a cofibration.  In the case when $R$ is
cofibrant as a commutative ring orthogonal spectrum, an argument similar
to~\cite[VII.7.5]{EKMM} shows that the iterated pushouts that form the
latching objects are $h$-cofibrations and the inclusion of the
latching object is an $h$-cofibration.
\end{proof}

\begin{proof}[Proof of Theorem~\ref{thm:levelwisefp}]
Given the discussion above, we see that under the hypotheses of the
theorem, the point-set geometric fixed point functor $\Phi^{C_{n}}$
commutes with geometric realization, giving us the isomorphism 
\[
\Phi^{C_{n}}N_{e}^{S^{1}}R\iso
\aI_{\bR^{\infty}}^{U}|\Phi^{C_{n}}\aI_{\bR^{\infty}}^{\tilde
U}\sd_{n}N^{\cyc}_{\sma}R|.
\]
To prove that the point-set geometric fixed point functor computes the
derived geometric fixed point functor, we just need to see that it
does so on each of the objects involved in the sequence of pushouts
that constructs the geometric realization.  This happens on the levels
of $N\subdot=\aI_{\bR^{\infty}}^{\tilde U}\sd_{n}N^{\cyc}_{\sma}R$ because each
$N_{k}$ is the smash product of copies of $N_{e}^{C_{n}}R$ and it
happens on $N_{e}^{C_{n}}R$ by Theorems~\ref{thm:normisobasic}
and~\ref{thm:comnormisobasic}.  The other pieces are the orthogonal $C_{n}$-spectra
$P_{k}$ defined by the pushout diagram
\[
\xymatrix@-1pc{%
L_{k}\sma \partial \Delta^{k}_{+}\ar[r]\ar[d]&L_{k}\sma \Delta^{k}_{+}\ar[d]\\
N_{k}\sma \partial \Delta^{k}_{+}\ar[r]&P_{k},
}
\]
where $L_{k}$ denotes the latching object.  The point-set geometric
fixed point functor computes the derived geometric fixed point functor
for each $P_{k}$ because it does so for each $N_{k}$ and for each
latching object (by induction).
\end{proof}


Finally, we turn to the question of understanding the derived functors of
$N_e^{S^1}$.  Recall that when dealing with cyclic sets, the
$S^1$-fixed points do not usually carry homotopically meaningful
information.  As a consequence, we will work with the model structure
on $\Spec^{S^{1}}_{U}$ provided by Theorem~\ref{thm:fmodel} with weak
equivalences the $\aFin$-equivalences, i.e., the maps which are
isomorphisms on the homotopy groups of the (categorical or geometric)
fixed point spectra for the 
finite subgroups of $S^{1}$ (irrespective of what happens on the
fixed points for $S^{1}$).  We will now write $\Spec^{S^1,\aFin}_{U}$
for $\Spec^{S^{1}}_{U}$ to emphasize that we are using the 
$\aFin$-equivalences.  We use analogous notation for the categories of
ring orthogonal $S^{1}$-spectra and commutative ring orthogonal
$S^{1}$-spectra. 

We now observe that $N_e^{S^1}$ admits (left) derived functors when
regarded as landing in $\Spec^{S^1,\aFin}_{U}$ and (in the commutative
case) $\Com^{S^{1},\aFin}_{U}$. 
Theorems~\ref{thm:cyclodiag} and~\ref{thm:levelwisefp} have the
following consequence.

\begin{theorem}\label{thm:Tder}
Let $R \to R'$ be a weak equivalence of ring orthogonal
spectra where $R$ and $R'$ is each either a cofibrant ring orthogonal
spectra or a cofibrant commutative ring orthogonal spectra (four cases).  Then the
induced map $N_e^{S^1} R \to N_e^{S^1} R'$ is an $\aFin$-equivalence.
\end{theorem}

\begin{proof}
Since we have shown that $N_{e}^{S^{1}}R$ and $N_{e}^{S^{1}}R'$ are
cyclotomic spectra and the map is a map of cyclotomic spectra, it
suffices to prove that it is a weak equivalence of the underlying
non-equivariant spectra, where we are looking at the map
$|N^{\cyc}_{\sma}R|\to |N^{\cyc}_{\sma}R'|$.  At each simplicial level,
the map $R^{\sma (k+1)}\to R^{\prime {\sma (k+1)}}$is a weak equivalence and the simplicial objects are proper,
so the map on geometric realizations is a weak equivalence.
\end{proof}

In the commutative case, we have the following derived functor
result.

\begin{proposition}\label{prop:Tcommquillen}
Regarded as a functor on commutative ring orthogonal spectra, the
functor $N_e^{S^1}$ is a left Quillen functor with respect to the positive
complete model structure on $\Com$ and the $\aFin$-model structure
on $\Com^{S^1}_{U}$.
\end{proposition}

\begin{proof}
The forgetful functor preserves fibrations and acyclic fibrations.
\end{proof}

\section{The cyclotomic trace}\label{sec:trace}

The modern importance of $THH$ and $TC$ derives from the application
of the trace maps $K \to TC$ and $K \to TC \to THH$ to computing
algebraic $K$-theory.  In this section, we give a construction of the
cyclotomic trace in terms of the norm construction of $THH$.

First, observe that the constructions of Section~\ref{sec:THH}
and~\ref{sec:relTHH} generalize without modification to the setting of
categories enriched in orthogonal spectra:  Specifically, given a
spectral category $\aC$ we define the cyclic bar construction as the
geometric realization of the cyclic orthogonal spectrum with
$k$-simplices 
\[
[k] \mapsto \bigvee_{c_0, \ldots c_k} \aC(c_{1}, c_0) \sma \aC(c_{2},
c_{1}) \sma \ldots \sma \aC(c_{k},c_{k-1})\sma\aC(c_0, c_k).
\]
This construction gives rise to an orthogonal $S^1$-spectrum; we
have the following analogue of Lemma~\ref{lem:cycl}.

\begin{lemma}\label{lem:catcycl}
Let $\aC$ be a category enriched in orthogonal spectra.  Then the
geometric realization of the cyclic bar construction $|N^{\cyc}_{\sma}
\aC|$ is naturally an object in $\Spec^{S^1}_{\bR^{\infty}}$.
\end{lemma}

In order to obtain a cyclotomic structure, as in
Theorem~\ref{thm:maincyc}, we need to 
arrange for the mapping spectra in $\aC$ to be cofibrant.  Such a
spectral category is called ``pointwise cofibrant''~\cite[2.5]{BM2}.
Following~\cite[2.7]{BM2}, we have a cofibrant replacement functor on
spectral categories with a fixed object set that in particular
produces pointwise cofibrant spectral categories.

\begin{theorem}\label{thm:hmcyc}
Let $\aC$ be a pointwise cofibrant spectral category, then 
$\aI_{\bR^{\infty}}^{U}|N^{\cyc}_{\sma}\aC|$ has a natural structure
of a cyclotomic spectrum.
\end{theorem}

\begin{proof}
Much of this goes through just as in Section~\ref{sec:THH}.  The only
real divergence is that although levelwise
\[
\aI_{\bR^{\infty}}^{\tilde U}\sd_{n}N^{\cyc}_{\sma}\aC
\]
is no longer given as a smash of norms, the diagonal isomorphisms
\begin{multline*}
\bigvee_{c_0, \ldots c_k} \aC(c_{1}, c_0) \sma \aC(c_{2},
c_{1}) \sma \ldots \sma \aC(c_{k},c_{k-1})\sma\aC(c_0, c_k)\\
\to
\Phi^{C_{n}}\aI_{\bR^{\infty}}^{\tilde U}\biggl(
\bigvee_{c_0, \ldots c_{q}} \aC(c_{1}, c_0) \sma \aC(c_{2},
c_{1}) \sma \ldots \sma \aC(c_{q},c_{q-1})\sma\aC(c_0, c_{q})
\biggr)
\end{multline*}
(where $q=n(k+1)-1$)
arise as the composite of the diagonal isomorphism
\begin{multline*}
\bigvee_{c_0, \ldots c_k} \aC(c_{1}, c_0) \sma \aC(c_{2},
c_{1}) \sma \ldots \sma \aC(c_{k},c_{k-1})\sma\aC(c_0, c_k)\\
\to
\Phi^{C_{n}}N_{e}^{C_{n}}\biggl(
\bigvee_{c_0, \ldots c_{k}} \aC(c_{1}, c_0) \sma \aC(c_{2},
c_{1}) \sma \ldots \sma \aC(c_{k},c_{k-1})\sma\aC(c_0, c_{k})
\biggr)
\end{multline*}
and the isomorphism
\begin{multline*}
\Phi^{C_{n}}N_{e}^{C_{n}}\biggl(
\bigvee_{c_0, \ldots c_{k}} \aC(c_{1}, c_0) \sma \aC(c_{2},
c_{1}) \sma \ldots \sma \aC(c_{k},c_{k-1})\sma\aC(c_0, c_{k})
\biggr)\\
\to
\Phi^{C_{n}}\aI_{\bR^{\infty}}^{\tilde U}\biggl(
\bigvee_{c_0, \ldots c_{q}} \aC(c_{1}, c_0) \sma \aC(c_{2},
c_{1}) \sma \ldots \sma \aC(c_{q},c_{q-1})\sma\aC(c_0, c_{q})
\biggr)
\end{multline*}
induced by the inclusion
\begin{multline*}
\biggl(
\bigvee_{c_0, \ldots c_{k}} \aC(c_{1}, c_0) \sma \aC(c_{2},
c_{1}) \sma \ldots \sma \aC(c_{k},c_{k-1})\sma\aC(c_0, c_{k})
\biggr)^{\sma (n)}\\
\to
\bigvee_{c_0, \ldots c_{q}} \aC(c_{1}, c_0) \sma \aC(c_{2},
c_{1}) \sma \ldots \sma \aC(c_{q},c_{q-1})\sma\aC(c_0, c_{q})
\end{multline*}
of the summands where $c_{i(k+1)+j}=c_{j}$ for all $0<i< n$, $0\leq j<k+1$.
\end{proof}

We simplify notation by writing $THH(\aC)$ for the orthogonal
$S^{1}$-spectrum or cyclotomic spectrum
$\aI_{\bR^{\infty}}^{U}|N^{\cyc}_{\sma}\aC|$. 
From this point, the construction of $TR$ and $TC$ proceeds
identically with the case of ring orthogonal spectra.

We now turn to the construction of the cyclotomic trace.  The trace
map is induced from the inclusion of objects map
\[
\ob(\aC) \to |N^{\cyc}_{\sma} \aC|
\]
that takes $x$ to the identity map $x \to x$ in the zero-skeleton of
the cyclic bar construction.  To make use of this for $K$-theory, we
use the Waldhausen construction of $K$-theory as the geometric
realization of the nerve of the multisimplicial spectral category $w\subdot
S^{(n)}\subdot\aC$ and consider the bispectrum $THH(w\subdot
S^{(n)}\subdot \aC)$.  The construction now proceeds in the usual way
(e.g., see~\cite[1.2.5]{BMtw}).

\section{A description of relative $THH$ as the relative $S^1$-norm}\label{sec:relTHH}

In this section, we extend the work of Section~\ref{sec:THH} to the
setting of $A$-algebras for a commutative ring orthogonal spectrum
$A$.
The category of $A$-modules is a symmetric monoidal category 
with respect to $\sma_{A}$, the smash product over $A$.  As explained
in~\cite[\S A.3]{HHR}, the construction of the indexed smash product can
be carried out in the symmetric monoidal category of $A$-modules.  Our
construction of relative $THH$ will use the associated
$A$-relative norm.

We will write $A_G$ to denote the commutative 
ring orthogonal $G$-spectrum obtained by regarding $A$ as having
trivial $G$-action; i.e., $A_G = \aI_{\bR^{\infty}}^{U} A$.
This is a commutative ring orthogonal $G$-spectrum since
$\aI_{\bR^{\infty}}^{U}$ is a 
symmetric monoidal functor.  For example, if $A$ is the sphere
spectrum then $A_G$ is the $G$-equivariant sphere spectrum. 

\begin{warning}
Although $\aI_{\bR^{\infty}}^{U}$ performs the (derived) change of universe on
stable categories for cofibrant orthogonal spectra, and
$\aI_{\bR^{\infty}}^{U}$ has a left derived functor on commutative
ring orthogonal spectra (Proposition~\ref{prop:comringderI} below), the
underlying object in the stable category 
of $A_{G}$ is not the derived change of universe applied to $A$ except
in rare cases like $A=S$; see Example~\ref{ex:Isux} below. 
As a consequence, in the following result the
comparison map between the left derived functor and the left derived
functor of $\aI_{\bR^{\infty}}^{U}\colon \Spec\to \Spec^{G}_{U}$ is not
an isomorphism.
\end{warning}

\begin{proposition}\label{prop:comringderI}
The functor $\aI_{\bR^{\infty}}^{U}\colon \Com\to
\Com^{G}_{U}$ is a Quillen left adjoint.
\end{proposition}

\begin{proof}
The functor in question is the composite of the inclusion of $\Com$ in
$\Com^{G}_{\bR^{\infty}}$ as the objects with trivial $G$-action
(which is Quillen left adjoint to the $G$-fixed point functor) and the
Quillen left adjoint $\aI_{\bR^{\infty}}^{U}\colon
\Com^{G}_{\bR^{\infty}}\to \Com^{G}_{U}$.  The Quillen right adjoint
is the composite $(-)^G \circ \aI_{U}^{\bR^{\infty}}$.
\end{proof}

\begin{example}\label{ex:Isux}
For $X$ a non-equivariant positive cofibrant orthogonal spectrum, $\bP
X$ is a cofibrant commutative ring orthogonal spectrum.  We have that $\aI_{\bR^{\infty}}^{U}\bP X=\bP\aI_{\bR^{\infty}}^{U} X$,
whose underlying object in the equivariant stable category is
isomorphic to $\bigvee
E_{G}\Sigma_{n+}\sma_{\Sigma_{n}}\aI_{\bR^\infty}^{U} X^{\sma n}$ by
\cite[III.8.4]{MM}, \cite[B.117]{HHR}.  On the other hand, the
underlying object of $\bP X$ in the non-equivariant stable category is
isomorphic to $\bigvee E\Sigma_{n+}\sma_{\Sigma_{n}}X^{\sma n}$, which
the derived functor on stable categories takes to $\bigvee
E\Sigma_{n+}\sma_{\Sigma_{n}}\aI_{\bR^{\infty}}^{U}X^{\sma n}$. In
general, the commutative ring derived functor is related to the stable
category derived functor by change of operads along $E\Sigma_{*}\to
E_{G}\Sigma_{*}$, cf.~\cite{BlumbergHill}.
\end{example}

For an $A$-algebra $R$, let $N^{\cyc}_{\sma_A} R$ denote the
cyclic bar construction with respect to the smash product over $A$.
The same proof as Lemma~\ref{lem:cycl} implies the following.

\begin{lemma}
Let $R$ be an object in $A\Alg$.  Then the geometric realization of
the cyclic bar construction $|N^{\cyc}_{\sma_A} R|$ is naturally an
object in $A\Mod^{S^1}_{\bR^{\infty}}$.
\end{lemma}

Using the point-set change of universe functors we can turn this into
an orthogonal $S^{1}$-spectrum indexed on the complete universe $U$.

\begin{definition}
Let $R$ be a ring orthogonal spectrum.  Define the functor
\[
\AN_e^{S^1} \colon A\Alg \to A\Mod^{S^{1}}_{U}
\]
as the composite
\[
\AN_e^{S^1} R = \aI_{\bR^{\infty}}^{U} |N^{\cyc}_{\sma_A} R|.
\]
\end{definition}

The argument for Proposition~\ref{prop:commuT} also proves the
following relative version.

\begin{proposition}\label{prop:relcommuT}
The restriction of $\AN_e^{S^1}$ to commutative $A$-algebras lifts to
a functor
\[
\AN_e^{S^1} \colon A\CAlg \to A_{S^1}\CAlg^{S^{1}}_{U}
\]
that is left adjoint to the forgetful functor 
\[
\iota^{*}\colon  A_{S^1}\CAlg^{S^{1}}_{U} \to  A\CAlg
\]
\end{proposition}

We now make a non-equivariant observation about relative $THH$
(ignoring the group action temporarily) that informs our description of
the equivariant structure.  Similar theorems have appeared previously
in the literature, e.g., \cite[\S 5]{McCarthyMinasian}.

\begin{lemma}\label{lem:nonequirel}
Let $R$ be an $A$-algebra in orthogonal spectra.  Then there is an
isomorphism
\[
\ATHH[S](R) \sma_{\ATHH[S](A)} A \cong \ATHH(R).
\]
\end{lemma}

\begin{proof}
Commuting the smash product with geometric
realization reduces the lemma to verifying the formula
\[
(R \sma R \sma \ldots \sma R) \sma_{A \sma A \sma \ldots \sma A} A \cong
R \sma_A R \sma_A \ldots \sma_A R,
\]
which is a straightforward calculation. 
\end{proof}

We now generalize Lemma~\ref{lem:nonequirel} to take advantage of the
equivariant structure. 

\begin{proposition}\label{prop:basechange}
Let $G$ be a finite group.  Let $A$ be a commutative ring orthogonal spectrum and
$M$ an $A$-module.  The $A$-relative norm is obtained by base-change from the
usual norm:
\[
\AN_e^G M \cong N_e^G M \sma_{N_e^G A} A_G
\]
\end{proposition}

\begin{proof}
Since $M$ is an $A$-module, we know that $N_e^G M$ is an $N_e^G
A$-module (in the category $\Spec^{G}_{U}$), using the fact that the norm
is a symmetric monoidal functor~\cite[A.53]{HHR}.  The right hand side
is the extension of scalars along the canonical map $N_e^G A \to A_G$
obtained as the adjoint of the natural (non-equivariant) map $A \to
A_G$. 
Because the map $N_{e}^{G}(-)\to \AN_{e}^{G}(-)$ is a monoidal natural
transformation, we obtain a canonical map
from $N_e^G M \sma_{N_e^G A} A_G$ to $\AN_e^G M$; this map is an
isomorphism because it is clearly an isomorphism after forgetting the
equivariance.
\end{proof}

Extending this to $S^1$, if $R$ is an $A$-algebra we have the
following characterization of relative $THH$ as an $S^1$-spectrum
that follows by essentially the same argument.

\begin{proposition}\label{prop:extscal}
Let $R$ be an $A$-algebra in orthogonal spectra.  Then we have an
isomorphism
\[
\AN_e^{S^1} R \cong N_e^{S^1} R \sma_{N_e^{S^1} A} A_{S^1} 
\]
\end{proposition}

We now turn to the homotopical analysis of $\AN_e^{S^1}$.  The
following theorem asserts that
the left derived functor of $\AN_{e}^{S^{1}}$ exists.

\begin{theorem}\label{thm:relTder}
Let $R \to R'$ be a weak equivalence of cofibrant 
$A$-algebras.  Then 
the induced map $\AN_e^{S^1} R \to \AN_e^{S^1} R'$ is an
$\aFin$-equivalence. 
\end{theorem}

To prove this theorem, it suffices to prove the following theorem,
which in particular implies Proposition~\ref{prop:introsmader}.

\begin{theorem}\label{thm:relsmader}
Let $R$ be a cofibrant $A$-algebra.  The smash product
$N_{e}^{S^{1}}R\sma_{N_{e}^{S^{1}}A}A_{S^{1}}$ represents the derived
smash product in the $\aFin$-model structure.
\end{theorem}

\begin{proof}
Let $N$ be a cofibrant $N_{e}^{S^{1}}A$-module approximation of
$N_{e}^{S^{1}}R$; the assertion is that the map
\[
N\sma_{N_{e}^{S^{1}}A}A_{S^{1}}\to N_{e}^{S^{1}}R\sma_{N_{e}^{S^{1}}A}A_{S^{1}}
\]
is a $\aFin$-equivalence.  
We compare to the bar construction:  Let
$B(N,N_{e}^{S^{1}}A,A_{S^{1}})$ be the geometric realization of the simplicial object with
$k$-simplices
\[
N \sma \underbrace{N_{e}^{S^{1}}A\sma\cdots \sma N_{e}^{S^{1}}A}_{k} \sma A_{S^{1}},
\]
and similarly for $B(N_{e}^{S^{1}}R,N_{e}^{S^{1}}A,A_{S^{1}})$.
Then we have a commutative diagram
\[
\xymatrix{%
B(N,N_{e}^{S^{1}}A,A_{S^{1}})\ar[r]\ar[d]
&N\sma_{N_{e}^{S^{1}}A}A_{S^{1}}\ar[d]
\\
B(N_{e}^{S^{1}}R,N_{e}^{S^{1}}A,A_{S^{1}})\ar[r]
&N_{e}^{S^{1}}R\sma_{N_{e}^{S^{1}}A}A_{S^{1}}
}
\]
We want to show that the righthand map is a $\aFin$-equivalence; we
show that the remaining three maps are $\aFin$-equivalences.  We apply
the change of groups functor $\iota^{*}_{C_{n}}$ and show that they
are weak equivalences of orthogonal $C_{n}$-spectra.  Since
$\iota^{*}_{C_{n}}$ commutes with smash product and geometric
realization, we have isomorphisms
\begin{gather*}
\iota^{*}_{C_{n}}B(N,N_{e}^{S^{1}}A,A_{S^{1}})
\iso
B(\iota^{*}_{C_{n}}N,\iota^{*}_{C_{n}}N_{e}^{S^{1}}A,A_{C_{n}})\\
\iota^{*}_{C_{n}}(N\sma_{N_{e}^{S^{1}}A}A_{S^{1}})\iso
\iota^{*}_{C_{n}}N\sma_{\iota^{*}_{C_{n}}N_{e}^{S^{1}}A}A_{C_{n}}
\end{gather*}
and similarly for $N_{e}^{S^{1}}R$ in place of $N$.

Before proceeding, we note that $\iota^{*}_{C_{n}}N_{e}^{S^{1}}A$ and
$\iota^{*}_{C_{n}}N_{e}^{S^{1}}R$ are flat in the sense of
\cite[B.15]{HHR}.  This can be seen as follows.  $N_{e}^{C_{n}}A$ is
flat by~\cite[B.147]{HHR} and $N_{e}^{C_{n}}R$ is flat being the
sequential colimit of pushouts over $h$-cofibrations of flat objects.
Likewise, $\iota^{*}_{C_{n}}N_{e}^{S^{1}}A$,
$\iota^{*}_{C_{n}}N_{e}^{S^{1}}R$, and $\iota^{*}_{C_{n}}N$ are sequential colimits of pushouts
over $h$-cofibrations of objects that are flat, q.v.\
Theorem~\ref{thm:sdcyc} for $N_{e}^{S^{1}}A$ and $N_{e}^{S^{1}}R$. As
an immediate consequence, we see that the map 
\[
B(N,N_{e}^{S^{1}}A,A_{S^{1}})\to 
B(N_{e}^{S^{1}}R,N_{e}^{S^{1}}A,A_{S^{1}})
\]
is a $\aFin$-equivalence as
\[
B(\iota^{*}_{C_{n}}N,\iota^{*}_{C_{n}}N_{e}^{S^{1}}A,\iota^{*}_{C_{n}}A_{S^{1}})\to 
B(\iota^{*}_{C_{n}}N_{e}^{S^{1}}R,\iota^{*}_{C_{n}}N_{e}^{S^{1}}A,\iota^{*}_{C_{n}}A_{S^{1}})
\]
is a weak equivalence on each simplicial level and the simplicial
objects are proper.

To see that $\iota^{*}_{C_{n}}B(N,N_{e}^{S^{1}}A,A_{S^{1}})\to
\iota^{*}_{C_{n}}(N\sma_{N_{e}^{S^{1}}A}A_{S^{1}})$ is a weak equivalence, let $M$ be a
cofibrant $N_{e}^{S^{1}}A$-module approximation of $A_{S^{1}}$.  Since
smash product commutes with geometric realization, we have compatible
isomorphisms
\begin{gather*}
B(N,N_{e}^{S^{1}}A,N_{e}^{S^{1}}A)\sma_{N_{e}^{S^{1}}A}M\iso
B(N,N_{e}^{S^{1}}A,M)\\
B(N,N_{e}^{S^{1}}A,N_{e}^{S^{1}}A)\sma_{N_{e}^{S^{1}}A}A_{S^{1}}\iso
B(N,N_{e}^{S^{1}}A,A_{S^{1}})
\end{gather*}
Now we have a commutative diagram 
\[
\xymatrix{%
B(N,N_{e}^{S^{1}}A,M)\ar@{{}{}{}}[r]|-{\textstyle \cong}\ar[d]
&B(N,N_{e}^{S^{1}}A,N_{e}^{S^{1}}A)\sma_{N_{e}^{S^{1}}A}M\ar[r]\ar[d]
&N\sma_{N_{e}^{S^{1}}A}M\ar[d]\\
B(N,N_{e}^{S^{1}}A,A_{S^{1}})\ar@{{}{}{}}[r]|-{\textstyle \cong}
&B(N,N_{e}^{S^{1}}A,N_{e}^{S^{1}}A)\sma_{N_{e}^{S^{1}}A}A_{S^{1}}\ar[r]
&N\sma_{N_{e}^{S^{1}}A}A_{S^{1}}.
}
\]
with the bottom composite map becoming the map in question after
applying $\iota^{*}_{C_{n}}$.  The lefthand map
becomes a weak equivalence after applying $\iota^{*}_{C_{n}}$ because
both $\iota^{*}_{C_{n}}N$ and $\iota^{*}_{C_{n}}N_{e}^{S^{1}}A$ are
flat. The top map
is a weak equivalence because $(-)\sma_{N_{e}^{S^{1}}A}M$ preserves
the weak equivalence $B(N,N_{e}^{S^{1}}A,N_{e}^{S^{1}}A)\to N$ and the
righthand map is a weak equivalence because $N\sma_{N_{e}S^{1}}(-)$
preserves the weak equivalence $M\to A_{S^{1}}$.

Finally, to see that the map 
\[
\iota^{*}_{C_{n}}B(N_{e}^{S^{1}}R,N_{e}^{S^{1}}A,A_{S^{1}})\to
\iota^{*}_{C_{n}}N_{e}^{S^{1}}R\sma_{N_{e}^{S^{1}}A}A_{S^{1}}
\]
is a weak equivalence, we apply Theorem~\ref{thm:sdcyc} to observe
that it is induced by a map of simplicial objects
\begin{multline*}
B(N^{\cyc}_{\sma} (N_e^{C_n} R,{^{g}N_e^{C_n} R}),
N^{\cyc}_{\sma} (N_e^{C_n} A,{^{g}N_e^{C_n} A}),A_{C_{n}})\\
\to
N^{\cyc}_{\sma}(N_e^{C_n} R,{^{g}N_e^{C_n} R})\sma_{N^{\cyc}_{\sma}(N_e^{C_n} A,{^{g}N_e^{C_n} A})}A_{C_{n}}.
\end{multline*}
Here at the $k$th level, the map is
\begin{multline*}
B((N_e^{C_n} R)^{\sma(k)}\sma {^{g}N_e^{C_n} R},
(N_e^{C_n} A)^{\sma(k)}\sma {^{g}N_e^{C_n} A},A_{C_{n}})
\\
\to
((N_e^{C_n} R)^{\sma(k)}\sma {^{g}N_e^{C_n} R})\sma_{%
(N_e^{C_n} A)^{\sma(k)}\sma {^{g}N_e^{C_n} A}%
}A_{C_{n}},
\end{multline*}
which is a weak equivalence since 
$(N_e^{C_n} R)^{\sma(k)}\sma {^{g}N_e^{C_n} R}$ is flat as a
module over $(N_e^{C_n} A)^{\sma(k)}\sma {^{g}N_e^{C_n} A}$.
\end{proof}

Similarly, we can extend the homotopical statement of
Proposition~\ref{prop:Tcommquillen} to the relative setting.

\begin{proposition}\label{prop:relTcommquillen}
Regarded as a functor on commutative $A$-algebras, the functor
$\AN_e^{S^1}$ is a left Quillen functor with respect to the positive
complete model structure on $A\CAlg$ and the $\aFin$-model structure
on $A_{S^1}\CAlg_{U}^{S^1}$.
\end{proposition}

\begin{proposition}\label{prop:relTdercompder}
Let $R \to R'$ be a weak equivalence of $A$-algebras
where $R$ is cofibrant and $R'$ is a cofibrant commutative
$A$-algebra.  Then the
induced map $\AN_e^{S^1} R \to \AN_e^{S^1} R'$ is an $\aFin$-equivalence.
\end{proposition}

\begin{proof}
By Theorem~\ref{thm:relsmader}, 
\[
\AN_{e}^{S^{1}}R\iso N_{e}^{S^{1}}R\sma_{N_{e}^{S^{1}}A}A_{S_{1}}
\]
represents the derived smash product.  Since $N_{e}^{S^{1}}R'$ is
cofibrant as a commutative $N_{e}^{S^{1}}A$-algebra,
\[
\AN_{e}^{S^{1}}R'\iso N_{e}^{S^{1}}R'\sma_{N_{e}^{S^{1}}A}A_{S_{1}}
\]
also represents the derived smash product.
\end{proof}

\section{When do we have relative cyclotomic structures?}\label{sec:amodcyc}

One application of the perspective of $THH$ as the $S^{1}$-norm is the
construction of relative versions of $TR$ and $TC$ built from
$\AN_e^{S^1} R$, which we discuss in this section.  In previous
drafts, the authors asserted that $\AN_e^{S^1}$ could in general be
endowed with cyclotomic structure or op-pre-cyclotomic structures.
However, as explained below, except for very special choices for $A$
(such as $A=S$), this is not correct.  Some of the difficulties arise
from subtleties of the behavior of the derived functor of change of
universe on commutative ring orthogonal spectra,
q.v. Example~\ref{ex:Isux} above and Example~\ref{ex:Usux} below.
Other difficulties arise from a basic incompatibility of diagonal
maps, as we will explain.

We begin with an example due to Lars Hesselholt that illustrates the
impossibility of a general construction of a nontrivial cyclotomic
structure.

\begin{example}\label{ex:WeSux}
Suppose that we could construct $p$-cyclotomic structures for general
$R$ and $A$, and that the expected naturality holds.  Then in
particular we would have a commutative diagram of ring orthogonal
spectra
\[
\xymatrix{
THH(\bF_p) \ar[r] \ar[d] & THH_{H\bZ}(\bF_p) \ar[d] \\
THH(\bF_p)^{tC_p} \ar[r] & THH_{H\bZ}(\bF_p)^{tC_p}, \\
}
\]
where $(-)^{tC_p}$ denotes the Tate construction.  Passing to homotopy
groups yields a commutative diagram of graded rings
\[
\xymatrix{
S_{\bF_p}(t) \ar[r] \ar[d] & \Gamma_{\bF_p}(t) \ar[d] \\
S_{\bF_p}(t,t^{-1}) \ar[r]^-{=} & S_{\bF_p}(t,t^{-1}), \\
}
\]
where $S_{\bF_p}$ denotes the symmetric algebra and $\Gamma_{\bF_p}$
the divided power algebra.  One then concludes that the composite of
the top horizontal map and the righthand vertical map must be zero in
positive degrees.
\end{example}

In order to understand the situation better, we now describe a natural 
op-precyclotomic structure on $A_{S^1}$.  The geometric fixed point
functor $\Phi^H$ is lax symmetric monoidal, and therefore gives rise
to a functor 
\[
\Phi^H \colon A_{G}\Mod^{G}_{U} \to 
(\Phi^H A_G)\Mod^{G/H}_{U^{H}}
\]
when $H$ is normal in $G$.  In the case of a finite subgroup $C_n <
S^1$, for an $A_{S^{1}}$-module $X$, we have that $\Phi^{C_n} X$ is an
orthogonal $S^1/C_n$-spectrum and a module over
$\Phi^{C_{n}}A_{S^{1}}$.  In fact, it is a module over $A_{S^1/C_n}$.

\begin{proposition}\label{prop:ghtophih}
Let $A$ be a (non-equivariant) cofibrant commutative ring orthogonal
spectrum and let $H$ be a normal subgroup of a compact Lie group $G$.
There is a natural map of commutative ring orthogonal $G/H$-spectra
$A_{G/H}\to \Phi^{H}A_{G}$.
\end{proposition}

\begin{proof}
By adjunction, maps $A_{G/H} \to \Phi^H A_{G}$ are in bijective
correspondence with maps $A \to (\Phi^H A_{G})^{G/H}$.  The natural
map in question can thus be constructed as the adjoint of the composite
\[
A \to (A_G)^G \cong ((A_G)^{H})^{G/H} \to (\Phi^H A_G)^{G/H}.
\]
Alternatively, we can give a direct construction as follows.
Let $X$ be an arbitrary non-equivariant orthogonal spectrum and write
$X_{G}$ for the application of the point-set functor
$\aI_{\bR^{\infty}}^{U}$.  We write $\Phi^{H}X_{G}$ as the coequalizer
\[
\xymatrix@C-1pc{
\bigvee\limits_{V,W<U}
\sJ^{U}_{G}(V,W)^{H}\sma F_{W^{H}}S^{0} \sma (X_{G}(V))^{H}
\ar@<.5ex>[r]\ar@<-.5ex>[r]
&\bigvee\limits_{V<U}F_{V^{H}}S^{0}\sma (X_{G}(V))^{H}
}
\]
in orthogonal $G/H$-spectra.
For $V$ an $H$-fixed $G$-inner product space, we can also regard $V$
as a $G/H$-inner product space, and we have
\[
X_{G/H}(V)\iso X_{G}(V) = (X_{G}(V))^{H}.
\]
Writing $X_{G/H}$ as the coequalizer
\[
\xymatrix@C-1pc{
\bigvee\limits_{V,W<U^{H}}
\sJ^{U^{H}}_{G/H}(V,W)\sma F_{W}S^{0}\sma X_{G/H}(V)
\ar@<.5ex>[r]\ar@<-.5ex>[r]
&\bigvee\limits_{V<U^{H}}F_{V}S^{0}\sma X_{G/H}(V),
}
\]
we get a canonical natural map of orthogonal $G/H$-spectra $\lambda
\colon X_{G/H}\to \Phi^{H}X_{G}$.  The symmetric monoidal
transformation $\Phi^{H}X_{G}\sma \Phi^{H}Y_{G}\to \Phi^{H}(X_{G}\sma
Y_{G})$ is induced by the natural map
\[
F_{V_{1}^{H}}S^{0} \sma (X_{G}(V_{1}))^{H}
\sma F_{V_{2}^{H}}S^{0}\sma (Y_{G}(V_{2}))^{H}
\to
F_{(V_{1}\oplus V_{2})^{H}}S^{0}\sma ((X_{G}\sma
Y_{G})(V_{1}\oplus V_{2}))^{H},
\]
and we see that $\lambda$ is also lax symmetric monoidal.  Applying
these observations to the commutative ring orthogonal spectrum $A$ and
the multiplication map $A\sma A\to A$, we see that $\lambda$ induces a
map of commutative ring orthogonal $G/H$-spectra $A_{G/H}\to
\Phi^{H}A_{G}$, natural in the commutative ring orthogonal spectrum
$A$. 
\end{proof}

We now specialize this to the subgroup $C_n < S^{1}$ and an
$A_{S^{1}}$-module $X$. Pulling back along the $n$th root 
isomorphism $\rho_{n} \colon S^{1}\to S^{1}/C_n$ gives rise to
an orthogonal $S^1$-spectrum $\rho^{*}_{n}\Phi^{C_n}X$ that is a
module over $A_{S^{1}}\iso \rho^{*}_{n}A_{S^{1}/C_n}$.  

\begin{definition}\label{defn:relcyclo}
An op-$p$-precyclotomic spectrum relative to $A$ consists of an 
$A_{S^1}$-module $X$ together with a map of $A_{S^1}$-modules
\[
\gamma \colon X \to \rho_{p}^{*}\Phi^{C_{p}} X.
\]
\end{definition}

Proposition~\ref{prop:ghtophih} thus constructs an op-$p$-precylotomic
spectrum structure on $A_{S^1}$.  However, it is important to be clear
about what this does and doesn't prove: specifically, we do not in
general know that $\Phi^{C_{p}} A_{S^1}$ computes the derived
geometric fixed points.  This problem can be circumvented by working
with $E_\infty$ objects in orthogonal spectra instead of strict
commutative rings; in this case, for a cofibrant $E_\infty$ ring
orthogonal spectrum $A$ we do obtain a version of the map from
Proposition~\ref{prop:ghtophih} that lands in the derived geometric fixed
 points. 

One would now hope to use the same argument as for
Theorem~\ref{thm:cyclodiag} to construct an op-$p$-precyclotomic
structure on $\AN_e^{S^1} R$.  Unfortunately, there is a basic
compatibility issue which we now explain.  It is possible to construct
an $A$-relative version of the diagonal map 
\[
\Delta_A \colon X \to \Phi^{G} \AN_e^G X.
\]
(a special case of the analogue of Proposition~\ref{prop:diag}), which
we can now state using Proposition~\ref{prop:ghtophih}.

\begin{proposition}\label{prop:reldiag}
Let $A$ be a commutative ring orthogonal spectrum and let $X$ be an
$A$-module. For any finite group $G$, there is a natural diagonal map
\[
 \Delta_A \colon  X \to \Phi^G \AN_e^G X.
\]
of $A$-modules, where the $A$-module action on the right is induced by
the composite map $A \to \Phi^G N_e^G A \to \Phi^G \AN_e^G X$. 
\end{proposition}

\begin{proof}
The map itself is constructed as the composite 
\[
 X\overto{\Delta} \Phi^{G} N_{e}^{G}X \to
 \Phi^{G}(A_{G}\sma_{N_{e}^{G}A}N_{e}^{G}X)\iso \Phi^{G}\AN_{e}^{G}X,
\]
where the last isomorphism is Proposition~\ref{prop:basechange}.  To
show that this is a map of $A$-modules as specified, it suffices to
show that the natural transformation $\Id \to \Phi^G N_e^G$ is lax
monoidal, as the second part of the composite clearly is.

Recall from the proof of Theorem~\ref{thm:comnormisobasic} that the
diagonal map can be described as follows: for every (non-equivariant)
inner product space $Z$, the universal property of the indexed smash
product gives a map of based $G$-spaces $N_{e}^{G}(X(Z))\to
(N_{e}^{G}X)(\Ind_{e}^{G} Z)$, which restricts on the diagonal to a
map  
\[
X(Z)\to (N_{e}^{G}X(\Ind_{e}^{G}
Z))^{G}=(\Fix^{G}(N_{e}^{G}X))(\Ind_{e}^{G} Z).
\]
Passing to the left Kan extension $\bP_{\phi}$ along the
fixed point functor $\phi \colon \sJ^{G}_{G}\to
\sJ_{e}$ on the right then induces a map
\[
X(Z)\to (\Phi^{G}(N_{e}^{G}X))((\Ind_{e}^{G} Z)^{G}) =(\Phi^{G}(N_{e}^{G}X))(Z).
\]
Let $X$ and $Y$ be spectra.  In what follows, we write $I_e^G$ for
$\Ind_e^G$ to save space.  A direct space-level verification shows
that the diagram  
\[
\xymatrix@C-1pc{
N_e^{G}(X(V)) \sma N_e^{G}(Y(W)) \ar[r] \ar[dd] & N_e^G(X(V) \sma
Y(W)) \ar[r] & N_e^G\left((X \sma Y)(V \oplus W)\right) \ar[d] \\
&& (N_e^G (X \sma Y))(I_e^G (V \oplus W)) \ar[d]_-{\cong} \\
(N_e^G X)(I_e^G V) \sma (N_e^G Y)(I_e^G W) \ar[rr] && (N_e^G (X \sma
Y))(I_e^G V \oplus I_e^G W) \\
}
\]
commutes.  Restricting to the diagonal we find that the diagram
\[
\xymatrix{
X(V) \sma Y(W) \ar[dd] \ar[r] & (X \sma Y)(V \oplus W) \ar[d] \\
& (\Fix^{G}(N_e^{G}(X \sma Y)))(I_e^{G}(V \oplus W)) \ar[d]_-{\cong} \\
(\Fix^{G}(N_e^{G}X))(I_e^{G}V) \sma
(\Fix^{G}(N_e^{G}Y))(I_e^{G}W) \ar[r] &
(\Fix^{G}(N_e^{G}(X \sma Y)))(I_e^{G}V \oplus I_e^{G}W)) \\
}
\]
commutes.  Finally, since left Kan extension is lax monoidal, passing
to the left Kan extension along the fixed point functor $\phi$ yields
the commutative diagram 
\[
\xymatrix{
X(V) \sma Y(W) \ar[dd]_-{\Delta \sma \Delta} \ar[r] & (X \sma Y)(V \oplus
W) \ar[dd]^-{\Delta} \\
&\\
(\Phi^{G}(N_e^G X))(V) \sma (\Phi^{G}(N_e^G Y))(W) \ar[r] & (\Phi^{G}(N_e^G (X \sma Y)))(V \oplus W). 
}
\]
The top map is clearly the canonical product map, and
inspection of the construction of the lax monoidal structure on 
$\Phi^G$~\cite[V.4.7]{MM} implies that the bottom map is the monoidal
product on $\Phi^G N_e^G$.

\end{proof}

The following example indicates some of the complexity of the behavior
of this diagonal map.

\begin{example}\label{ex:Usux}
In the previous theorem, consider the case when $R=A$ and $A=\bP
F_{\bR}S^{0}$ is the free commutative ring orthogonal spectrum on
$F_{\bR}S^{0}\simeq S^{-1}$.  When $n=2$, 
\begin{align*}
\Fix^{C_{2}}\bP
F_{\bR}S^{0}(W)&=\bigvee_{m}(\sJ_{S^{1}}(\bR^{m},W)/\Sigma_{m})^{C_{2}}\\
&\iso \bigvee_{m}\biggl(\bigvee_{f\colon C_{2}\to \Sigma_{m}} 
\sJ_{S^{1}}(f^{*}\bR^{m},W)^{C_{2}}\biggr)/\Sigma_{m},
\end{align*}
where the inner sum is over homomorphisms $f\colon C_{2}\to \Sigma_{m}$ and
$\Sigma_{m}$ acts on the set of such $f$ by conjugation (as well as
acting on $\bR^{m}$ by permuting coordinates).  Here $f^{*}\bR^{m}$
denotes $\bR^{m}$ with $C_{2}$-action coming from $f$ and the
coordinate permutation action of $\Sigma_{m}$. We can then calculate 
\[
\Phi^{C_{2}}\bP F_{\bR}S^{0}\iso 
\bigvee_{m}\biggl(\bigvee_{f\colon C_{2}\to \Sigma_{m}} 
F_{(f^{*}\bR^{m})^{C_{2}}}S^{0}\biggr)/\Sigma_{m}.
\]
The summands with $f$ the trivial map contribute a summand of
$A_{S^{1}/C_{2}}$, but the remaining summands make non-trivial
contributions of orthogonal $G/H$-spectra of the form
$F_{(\bR^{m})^{\sigma}}S^{0}/Z(\sigma)$ where $\sigma$ is an order $2$
element of $\Sigma_{m}$, $Z(\sigma)$ is its centralizer, and
$(f^{*}\bR^{m})^{\sigma}$ is its fixed points.  In this case we see
that the natural map of Proposition~\ref{prop:ghtophih} is split, and
in general it is split for free commutative ring orthogonal spectra,
but the splitting is not natural and so does not extend to a splitting
for arbitrary commutative ring orthogonal spectra $A$.
\end{example}

Although one might hope to use Proposition~\ref{prop:reldiag} to
construct an op-cyclotomic structure on $\ATHH$, there is an issue
related to the fact that the map of commutative ring orthogonal
spectra
\[
A\to \Phi^{G}N_{e}^G A\to \Phi^{G}A_{G}
\]
inducing the $A$-module structure on the relative diagonal is in
general not the same map as the canonical map given in
Proposition~\ref{prop:ghtophih}.  In order to elucidate the basic
incompatibility, we use the description of $\ATHH$ in terms of
base change given by Proposition~\ref{prop:extscal}.  Since
$\Phi^{C_p}$ commutes with smash product, the required structure
amounts to the data of the following commutative diagram
\[
\xymatrix{
N_e^{S^1} R \ar[d] & N_e^{S^1} A \ar[d]\ar[r]\ar[l] &
A_{S^1} \ar[d] \\
\Phi^{C_p} N_e^{S^1} R & \Phi^{C_p} N_e^{S^1} \ar[r] \ar[l] A
& \Phi^{C_p} A_{S^1}.\\ }
\]
The left-hand square commutes by naturality.  But using the
op-precyclotomic structure from Proposition~\ref{prop:ghtophih}, the
right-hand diagram does not in general commute!

However, this diagram does commute (essentially by hypothesis) in the
case that $A$ is the underlying non-equivariant commutative ring
orthogonal spectrum of a $p$-cyclotomic commutative ring orthogonal
$S^1$-spectrum $\underbar{A}$, the canonical map $N_e^{S^1}
A \to \underbar{A}$ is a map of $p$-cyclotomic spectra, and $R$ is an
$A$-algebra.  Specifically, we can immediately deduce
Theorem~\ref{thm:relwhencyc} from the introduction (using
Theorem~\ref{thm:relsmader} to retain homotopical control).

\begin{theorem}
Let $A$ be a cofibrant commutative ring orthogonal spectrum that is
$\iota_e^* \underbar{A}$ for a cofibrant $p$-cyclotomic commutative
ring orthogonal $S^1$-spectrum $\underbar{A}$.  Moreover, assume that
the canonical counit map $N_e^{S^1} A \to \underbar{A}$ is a
$p$-cyclotomic map.  Let $R$ be a cofibrant $A$-algebra.  Then the
derived smash product
\[
\AN_e^{S^1} R \cong N_e^{S^1} R \sma_{N_e^{S^1} A} \underbar{A}
\]
is a $p$-cyclotomic spectrum.
\end{theorem}

The same argument proves a slightly more general version of this
result, where instead we let $A$ be a commutative ring orthogonal
spectrum, $R$ an $A$-algebra, and $M$ a coefficient spectrum which is an
$N_e^{S^1} A$-module and a $p$-cyclotomic spectrum.  The following is
Theorem~\ref{thm:relwhencycmod} from the introduction.

\begin{theorem}
Let $A$ be a cofibrant commutative ring orthogonal spectrum and $R$ a
cofibrant $A$-algebra.  Let $M$ be a $p$-cyclotomic object in $N_e^{S^1}
A$-modules.  Then the derived smash product 
\[
N_e^{S^1} R \sma_{N_e^{S^1} A} M
\]
is a $p$-cyclotomic spectrum.
\end{theorem}

Under the hypotheses of Theorem~\ref{thm:relwhencyc}, using the
relative analogues of Definitions~\ref{defn:zTR} and~\ref{defn:zTC},
we obtain analogues of $TR$ and $TC$ which we denote $\ATR$ and
$\ATC$.  These constructions are evidently functorial, which proves
Theorem~\ref{thm:trace} from the introduction.

\section{$THH$ of ring $C_{n}$-spectra}\label{sec:cncyclo}

For $G$ a finite group and $H < G$ a subgroup, the norm $N_{H}^{G}$
provides a functor from orthogonal $H$-spectra to orthogonal
$G$-spectra.  In this section, we generalize this construction to a
relative norm $N_{C_{n}}^{S^{1}}$, which we view as a
``$C_{n}$-relative $THH$''.  We begin with an explicit construction in
terms of a cyclic bar construction, which generalizes the simplicial
object studied in Section~\ref{sec:THH} on the edgewise
subdivision of the cyclic bar construction. 

\begin{definition}\label{defn:ncntos1}
Let $R$ be an associative ring orthogonal $C_{n}$-spectrum indexed on
the trivial universe $\bR^{\infty}$.
Let $N_{\sma}^{\cyc,C_{n}}R$ denote the simplicial object that in
degree $q$ is $R^{\sma(q+1)}$, has degeneracy $s_{i}$ (for $0\leq i\leq
q)$ induced by the inclusion of the unit in the $(i+1)$-st factor, has
face maps $d_{i}$ for $0\leq i<q$ induced by multiplication of the $i$th
and $(i+1)$st factors.  The last face map $d_{q}$ is given as
follows.  Let $\alpha_{q}$ be the automorphism of $R^{\sma(q+1)}$ that
cyclically permutes the factors, putting the last factor in the zeroth
position, and then acts on that factor by the generator $g=e^{2\pi
i/n}$ of $C_{n}$.  The last face map is $d_{q}=d_{0}\circ \alpha_{q}$.
\end{definition}

The previous definition constructs a simplicial object but not a
cyclic object.  Nevertheless, it does have extra structure of the same
sort found on the edgewise subdivision of a cyclic object.  The
operator $\alpha_{q}$ in simplicial degree $q$ is the generator of a
$C_{n(q+1)}$-action (the action obtained by regarding $R^{\sma(q+1)}$ as
an indexed smash product for $C_{n}<C_{n(q+1)}$).  The faces,
degeneracies, and operators $\alpha_q$ satisfy the following relations in
addition to the usual simplicial relations:
\begin{align*}
\alpha_{q}^{n(q+1)}&=\id\\
d_{0}\alpha_{q}&=d_{q}\\
d_{i}\alpha_{q}&=\alpha_{q-1}d_{i-1}\text{ for }1\leq i\leq q\\
s_{i}\alpha_{q}&=\alpha_{q+1}s_{i-1}\text{ for }1\leq i\leq q\\
s_{0}\alpha_{q}&=\alpha_{q+1}^{2}s_{q}
\end{align*}
This defines a $\Lambda^{\op}_{n}$-object in the notation of
\cite[1.5]{BHM}.  As explained in \cite[1.6--8]{BHM}, the geometric
realization has an $S^{1}$-action extending the $C_{n}$-action.

\begin{definition}\label{defn:Cnnorm}
Let $R$ be an associative ring orthogonal $C_{n}$-spectrum indexed on
the universe $\tilde U=\iota^{*}_{C_{n}}U$.  The relative norm
$N_{C_{n}}^{S^{1}}R$ is defined as the composite functor
\[
N_{C_{n}}^{S^{1}}R=\aI_{\bR^{\infty}}^{U}
| N^{\cyc,C_{n}}_{\sma}(\aI_{\tilde U}^{\bR^{\infty}}R) |
\]
\end{definition}

When $R$ is a commutative ring orthogonal $C_{n}$-spectrum, we have
the following analogue of Proposition~\ref{prop:commuT}.

\begin{proposition}
The restriction of $N_{C_{n}}^{S^{1}}$ to $\Com^{C_{n}}_{\tilde U}$ lifts
to a functor
\[
N_{C_{n}}^{S^{1}} \colon \Com^{C_{n}}_{\tilde U} \to \Com^{S^1}_{U}
\]
that is left adjoint to the forgetful functor 
\[
\iota^{*} \colon \Com^{S^1}_{U} \to \Com^{C_{n}}_{\tilde U}.
\]
\end{proposition}

We now describe the homotopical properties of the relative norm.  The
following analogue of Theorem~\ref{thm:Tder} has the same proof.

\begin{theorem}
Let $R \to R'$ be a weak equivalence of cofibrant associative ring
orthogonal $C_{n}$-spectra.  Then $N_{C_{n}}^{S^{1}}R \to
N_{C_{n}}^{S^{1}}R'$ is a $\aFin$-equivalence. 
\end{theorem}

In the commutative case, we have the following analogue of
Proposition~\ref{prop:Tcommquillen} (also using an identical proof).

\begin{theorem}
Regarded as a functor on commutative ring orthogonal $C_{n}$-spectra, the
functor $N_{C_{n}}^{S^{1}}$ is a left Quillen functor with respect to
the positive complete model structure on $\Com^{C_{n}}$ and the positive complete
$\aFin$-model structure on $\Com^{S^1}_{U}$.
\end{theorem}

We now turn to the question of the cyclotomic structure.

\begin{theorem}
Let $R$ be a cofibrant associative ring orthogonal $C_{n}$-spectrum.
If $p$ is prime to $n$, then $N_{C_{n}}^{S^{1}}R$
has the natural structure of a $p$-cyclotomic spectrum.
\end{theorem}

\begin{proof}
As in the proof of Theorem~\ref{thm:cyclodiag}, we can identify
$\iota^{*}_{C_{pn}}N_{C_{n}}^{S^{1}} R$ as the geometric realization
of a simplicial orthogonal $C_{pn}$-spectrum of the form
\[
N_{C_{n}}^{C_{pn}}(R^{\sma(\ssdot+1)}).
\]
Since $p$ is prime to $n$, by Proposition \ref{prop:diag} we have a diagonal map $R^{\sma(q+1)}\to
\Phi^{C_{p}}N_{C_{n}}^{C_{pn}} R^{\sma(q+1)}$, which again 
commutes with the simplicial structure and induces a diagonal map
\[
\tau_{p}\colon 
N_{C_{n}}^{S^{1}}R\to 
\rho^{*}_{p}\Phi^{C_{p}}N_{C_{n}}^{S^{1}}R.
\]
Under the hypothesis that $R$ is cofibrant as an orthogonal
$C_{n}$-spectrum, Theorem~\ref{thm:normiso} shows that the diagonal map
$R^{\sma(q+1)}\to \Phi^{C_{p}}N_{C_{n}}^{C_{pn}}R^{\sma(q+1)}$ is an
isomorphism, and it follows that $\tau_{p}$ is an isomorphism.  The
inverse gives the $p$-cyclotomic structure map.
\end{proof}

As usual, we can construct $TR_{C_{n}} R$ and $TC_{C_{n}} R$ from the
cyclotomic structure on $N_{C_{n}}^{S^{1}} R$.  And as before, when
$R$ only has the homotopy type of a cofibrant object, application of
Proposition~\ref{prop:zconsist} allows us to work with $\opTR_{C_{n}}$
and $\opTC_{C_{n}}$.

When $p$ divides $n$, the diagonal map is of the form
\[
N_{C_{n/p}}^{S^{1}}\Phi^{C_{p}}R\to \Phi^{C_{p}}N_{C_{n}}^{S^{1}} R,
\]
and is an isomorphism when $R$ is cofibrant as an orthogonal
$C_{n}$-spectrum or as a commutative ring orthogonal
$C_{n}$-spectrum.  In these cases, we can get a $p$-cyclotomic structure
map if we have one on $R$ of the following form.

\begin{definition}\label{defn:cncyclo}
For $p \mid n$, a $C_{n}$ $p$-cyclotomic spectrum consists of an orthogonal
$C_{n}$-spectrum $X$ together with a map of orthogonal $C_{n}$-spectra
\[
t\colon N_{C_{n/p}}^{C_{n}}\Phi^{C_{p}}X\to X
\]
that induces a weak equivalence to $X$ from the derived composite functor.
\end{definition}

\begin{proposition}
Assume $p \mid n$ and 
let $R$ be an associative ring orthogonal $C_{n}$-spectrum with a
$C_{n}$ $p$-cyclotomic structure such that the structure map $t$ is a
ring map.  Then $N_{C_{n}}^{S^{1}}R$ has the natural
structure of a $p$-cyclotomic spectrum.
\end{proposition}

At present, we do not know if the previous proposition is interesting.
For any (non-equivariant) ring orthogonal spectrum $R'$,
$R=N_{e}^{C_{n}}R'$ satisfies the hypothesis of the previous
proposition, and $N_{C_{n}}^{S^{1}}R\iso N_{e}^{S^{1}}R'$. The
assumptions on $t$ imply that such an $R$ must be a norm from a
subgroup of order relatively prime to $p$; however, the structure map
$t$ does not necessarily preserve it.

\section{Spectral sequences for $\ATR$}\label{sec:spectralsequences}

In this section we present four spectral sequences for computing
$\ATR$.  In each case we actually have two spectral sequences, one
graded over the integers and a second graded over $RO(S^{1})$.  We
follow the modern convention of denoting an integral grading with $*$
and an $RO(S^{1})$-grading with $\star$.
Although the two look formally similar, they are very different
computationally, for reasons explained in the introduction to
\cite{LewisMandell2}: the Tor terms are computed using very different
notions of projective module.  Specifically, for $V$ a non-trivial
representation $\upi^{(-)}_{*}(\Sigma^{V}R)$ cannot be expected to be
projective as a $\upi^{(-)}_{*}R$ Mackey functor module; however, 
$\upi^{(-)}_{\star}(\Sigma^{V}R)$ is of course 
projective as a $\upi^{(-)}_{\star}R$ Mackey functor module, being
just a shift of the free module $\upi^{(-)}_{\star}R$.

\subsection{The absolute to relative spectral sequence}

The equivariant homotopy groups $\pi^{C_{n}}_{*}(N_{e}^{S^{1}}R)$ are the
$TR$-groups $TR^{n}_{*}(R)$ and so $\pi^{C_{n}}_{*}(\AN_{e}^{S^{1}}R)$
are by definition the relative $TR$-groups $\ATR^{n}_{*}(R)$.  

\begin{notation}Let
\begin{align*}
TR^{(-)}_{*}(R)&=\upi^{(-)}_{*}(N_{e}^{S^{1}}(R))
&TR^{(-)}_{\star}(R)&=\upi^{(-)}_{\star}(N_{e}^{S^{1}}(R))\\
\ATR^{(-)}_{*}(R)&=\upi^{(-)}_{*}(\AN_{e}^{S^{1}}(R))
&\ATR^{(-)}_{\star}(R)&=\upi^{(-)}_{\star}(\AN_{e}^{S^{1}}(R))\\
\end{align*}
\end{notation}

Using the isomorphism of Proposition~\ref{prop:extscal}
\[
\AN_{e}^{S^{1}}(R)\iso N_{e}^{S^{1}}(R)\sma_{N_{e}^{S^{1}}A}A_{S^{1}},
\]
we can apply the K\"unneth spectral sequences of \cite{LewisMandell2}
to compute the relative $TR$-groups from the absolute $TR$-groups and
Mackey functor $\uTor$.  Technically, to apply \cite{LewisMandell2}
and for ease of statement, we restrict to a finite subgroup
$H<S^{1}$.  Recall that for a commutative ring orthogonal spectrum
$A$, $A_H$ denotes $I_{\bR^{\infty}}^{\tilde{U}} A$ where $\tilde{U}$
is the complete $S^{1}$-universe regarded as a complete $H$-universe,
and we regard $A$ as an $H$-trivial orthogonal $H$-spectrum. 

\begin{theorem}
Let $A$ be a cofibrant commutative ring orthogonal spectrum and let
$R$ be a cofibrant associative $A$-algebra or cofibrant commutative
$A$-algebra.  For each finite subgroup $H<S^{1}$, there is a natural
strongly convergent spectral sequence of $H$-Mackey functors
\[
\uTor^{TR^{(-)}_{*}(A)}_{*,*}
(TR^{(-)}_{*}(R),\upi_{*}^{(-)}(A_{H}))
\quad\Longrightarrow\quad \ATR^{(-)}_{*}(R),
\]
compatible with restriction among finite subgroups of $S^{1}$.
\end{theorem}

Compatibility with restriction among finite subgroups of $S^{1}$
refers to the fact that for $H<K$, the restriction of the $K$-Mackey
functor $\uTor$ to an $H$-Mackey functor is canonically isomorphic to
the $H$-Mackey functor $\uTor$ and the corresponding isomorphism on
$E^{\infty}$-terms induces the same filtration on $\upi_{*}$.
(Free $K$-Mackey functor modules restrict to free
$H$-Mackey functor modules essentially because finite $K$-sets
restrict to finite $H$-sets.)

We also have corresponding K\"unneth spectral sequences graded on
$RO(H)$ for $H < S^{1}$ or $RO(S^{1})$.  We choose to state our
results in terms of the $RO(S^{1})$-grading because this makes the
behavior of the restriction among subgroups easier to describe; the
restriction maps $RO(S^1) \to RO(H)$ are surjective, and as a result
$\uTor$-groups calculated in $RO(H)$-graded homological algebra
restrict naturally to $\uTor$-groups calculated in
$RO(S^1)$-graded homological algebra. In the following theorem,
$\star$ denotes the $RO(S^{1})$-grading.

\begin{theorem}
Let $A$ be a cofibrant commutative ring orthogonal spectrum and let
$R$ be a cofibrant associative $A$-algebra or cofibrant commutative
$A$-algebra.  For each finite subgroup $H<S^{1}$, there is a natural
strongly convergent spectral sequence of $H$-Mackey functors
\[
\uTor^{TR^{(-)}_{\star}(A)}_{*,\star}
(TR^{(-)}_{\star}(R),\upi_{\star}^{(-)}(A_{H}))
\quad\Longrightarrow\quad \ATR^{(-)}_{\star}(R),
\]
compatible with restriction among finite subgroups of $S^{1}$.
\end{theorem}

\subsection{The simplicial filtration spectral sequence}

The spectral sequence of the preceding subsection essentially gives a
computation of the relative theory in terms of absolute theory.  More
often we expect to use the relative theory to compute the absolute
theory.  Non-equivariantly, the isomorphism
\begin{equation}\label{eq:changebaseringtxt}
THH(R) \sma A \cong \ATHH(R \sma A)
\end{equation}
gives rise to a K\"unneth spectral sequence
\[
\Tor^{A_*(R \sma_{S} R^{\op})}_{*,*}(A_*(R), A_*(R))
\quad\Longrightarrow\quad A_* (THH(R)).
\]
As employed by B\"okstedt, an Adams spectral sequence can then in
practice be used to compute the homotopy groups of $THH(R)$.  For
formal reasons, the isomorphism~\eqref{eq:changebaseringtxt} still
holds equivariantly, but now we have three different versions of the
non-equivariant K\"unneth spectral sequence (none of which have quite
as elegant an $E^{2}$-term) which we use in conjunction with
equation~\eqref{eq:changebaseringtxt}.

The first equivariant spectral sequence generalizes the K\"unneth
spectral sequence in the special case when $\pi_{*}A$ is a field.
Non-equivariantly, it derives from the simplicial filtration of the
cyclic bar construction; equivariantly, we restrict to a finite
subgroup $H<S^{1}$ and look at the simplicial filtration on the
$n$th edgewise subdivision (described in the proof of
Theorem~\ref{thm:Tder}).

\begin{theorem}
Let $A$ be a cofibrant commutative ring orthogonal spectrum and let
$R$ be a cofibrant associative $A$-algebra or cofibrant commutative
$A$-algebra.  Let $H$ be a finite subgroup of $S^{1}$.
\begin{enumerate}
\item There is a natural
spectral sequence strongly converging to the integer graded $H$-Mackey functor
$\ATR^{(-)}_{*}(R)$ with $E^{1}$-term
\[
\uE^{1}_{s,t}=\upi_{t}(\AN_{e}^{H}(R^{\sma(s+1)})).
\]
\item There is a natural
spectral sequence strongly converging to the $RO(S^{1})-$
graded $H$-Mackey functor 
$\ATR^{(-)}_{\star}(R)$ with $E^{1}$-term
\[
\uE^{1}_{s,\tau}=\upi_{\tau}(\AN_{e}^{H}(R^{\sma(s+1)})).
\]
\end{enumerate}
The $E^{2}$-terms of both spectral sequences are compatible with
restriction among finite subgroups of $S^{1}$.
\end{theorem}

To see the compatibility with restriction among subgroups, we note
that for $H=C_{mn}$, the $E^{2}$-term $(E^{2}_{*,\tau})^{C_{m}}$ is
the homology of the simplicial object 
\[
\sd_{n} \pi_{\star}^{C_{m}}((N_{e}^{C_{m}}A)^{\sma(\ssdot+1)}).
\]
For $H<K$, the subdivision operators then induce an isomorphism on
$E_{2}$-terms. 

In general, we do not know how to describe the $E^{2}$-term of these
spectral sequences.  One can formulate box-flatness hypotheses that
would permit the identification of the $E^{2}$-term as a kind of
Mackey functor Hochschild homology~\cite{AlgHoch}; however, such
hypotheses will rarely hold in practice.  On the other hand, when
$A=H\bF$ for $\bF$ a field, for formal reasons, the $E^{1}$-term is a
purely algebraic functor of the graded vector space $\pi_{*}R$.  We
conjecture that the $E^{2}$-term is a functor of the graded
$\bF$-algebra $\pi_{*}R$.

\subsection{The cyclic filtration spectral sequence}

We have a second spectral sequence arising from the filtration on
cyclic objects constructed by Fiedorowicz and Gajda
\cite{FiedorowiczGajda}.  Although they work in the context of spaces,
their arguments generalize to provide an $\aFin$-equivalence
\[
|EX\subdot|\to |X\subdot|
\]
for cyclic orthogonal spectra, where $E$ is the evident orthogonal
spectrum generalization of the construction in their Definition~1:
\[
EX\subdot = \int_{[m]\in \Lambda_{\text{face}}}X_{m}\sma \Lambda(\ssdot,[m])_{+}
\]
The proof of their Proposition~1 (which in fact only gives an
$\aFin$-equivalence for spaces) also applies in the orthogonal spectrum
context, substituting geometric fixed points for fixed points, to
prove the $\aFin$-equivalence for orthogonal spectra.  Change of
universe $\aI_{\bR^{\infty}}^{U}$ commutes with geometric realization,
and we use the coend filtration of $EX\subdot$ for
$X\subdot=N^{\cyc}_{\sma_{A}}R$ to obtain the following
Fiedorowicz-Gajda cyclic filtration spectral sequences.

\begin{theorem}
Let $A$ be a cofibrant commutative ring orthogonal spectrum and let
$R$ be a cofibrant associative $A$-algebra or cofibrant commutative
$A$-algebra.  Let $H$ be a finite subgroup of $S^{1}$.
\begin{enumerate}
\item There is a natural
spectral sequence of integer graded $H$-Mackey functors strongly converging to
$\ATR^{(-)}_{*}(R)$ with $E^{1}$-term 
\[
\uE^{1}_{s,t}=\upi_{t}(\aI_{\bR^{\infty}}^{U}(S^{1}_{+}\sma_{C_{s+1}}
R^{\sma(s+1)})).
\]
\item There is a natural
spectral sequence of $RO(S^{1})$-graded $H$-Mackey functors strongly
converging to 
$\ATR^{(-)}_{\star}(R)$ with $E^{1}$-term
\[
\uE^{1}_{s,\tau}=\upi_{\tau}(\aI_{\bR^{\infty}}^{U}(S^{1}_{+}\sma_{C_{s+1}}
R^{\sma(s+1)})).
\]
\end{enumerate}
The $E^{1}$-terms are compatible with restriction among finite subgroups
of $S^{1}$.
\end{theorem}

\subsection{The relative cyclic bar construction spectral sequence}

The third spectral sequence directly involves Mackey functor $\uTor$.
For an $A$-algebra $R$, let $\gAN_e^{C_n} R$ denote the ($\AN^{C_n}_e
R$,$\AN^{C_n}_e R$)-bimodule obtained by twisting the left action of
$\AN_e^{C_n} R$ on $\AN_e^{C_n} R$ by the generator $g=e^{2\pi i/n}$
of $C_n$.  We can identify the $C_n$-homotopy type of $\AN_e^{S^1} R$
in terms of this bimodule,
\[
\AN_{e}^{S^{1}}R\iso \aI_{\tilde U}^{U}N^{\cyc}_{\sma_{A}}
(\AN_e^{C_n} R, \gAN_e^{C_n} R),
\]
where the cyclic bar construction on the right is taken in the
symmetric monoidal category of $A$-modules in orthogonal
$C_{n}$-spectra and $\tilde U=\iota^{*}_{C_{n}}U$ denotes $U$ viewed
as a complete $C_{n}$-universe.  A consequence of this description is
that the main theorem of~\cite{LewisMandell2} constructing the
equivariant K\"unneth spectral sequence applies:

\begin{theorem}
Let $A$ be a cofibrant commutative ring orthogonal spectrum and let
$R$ be a cofibrant associative $A$-algebra or cofibrant commutative
$A$-algebra.  Fix $n>0$.
\begin{enumerate}
\item There is a
natural strongly convergent spectral sequence of integer graded
$C_{n}$-Mackey functors
\[
E^{2}_{*,*}=\uTor_{*,*}^{N_{e}^{C_{n}}(R\sma_{A}
R^{\op})}(\upi_{*}\AN_{e}^{C_{n}}R, \upi_{*}\gAN_{e}^{C_{n}} R)
\Longrightarrow \ATR^{(-)}_{*}(R).
\]
\item There is a
natural strongly convergent spectral sequence of $RO(S^{1})$-graded
$C_{n}$-Mackey functors
\[
E^{2}_{*,\star}=\uTor_{*,\star}^{N_{e}^{C_{n}}(R\sma_{A}
R^{\op})}(\upi_{\star}\AN_{e}^{C_{n}}R, \upi_{\star}\gAN_{e}^{C_{n}} R)
\Longrightarrow \ATR^{(-)}_{\star}(R).
\]
\end{enumerate}
\end{theorem}

We see no reason why the $E^{2}$-terms for the spectral sequences of
the previous theorem should be compatible under restriction among
finite subgroups of $S^{1}$.

\section{Adams operations}\label{sec:adams}

In this section, we study the circle power operations on $THH(R)$ for
a commutative ring $R$ and on $\ATHH(R)$ for a commutative $A$-algebra
$R$.  Such operations were first defined on Hochschild homology by
Loday~\cite{Loday} and Gerstenhaber-Schack~\cite{GerstenhaberSchack}
and explained by McCarthy \cite{McCarthy} in terms of covering maps of
the circle and extended to $THH$ by \cite{MSV}.
Following~\cite[4.5.3]{BCD}, we refer to these as \term{Adams
operations} and denote as $\psi^{r}$ (though in older
literature~\cite[4.5.16]{LodayCyclic}, the Adams operations differ by
a factor of the operation number $r$).  Specifically, we study how the
operations interact with the equivariance, and we show that when $r$
is prime to $p$, $\psi^{r}$ descends to an operation on $TR(R)$,
$TC(R)$, cf.~\cite[\S7]{BCD}.  We study the effect of $\psi^{r}$ on
$TR_{0}(R)$ and $TC_{0}(R)$, where we shown it is the identity on
$TR_{0}(R)$ when $R$ is connective.

We recall the construction of McCarthy's Adams operations, which
ultimately derives from the identification of $N^{\cyc}_{\sma_{A}}R$
as the tensor $R\otimes S^{1}$ in the category of commutative
$A$-algebras.  Using the standard model for the circle as the
geometric realization of a simplicial set $S^{1}\subdot$ (with one
$0$-simplex and one non-degenerate $1$-simplex), the tensor
identification is just observing that $N^{\cyc}_{\sma_A}R$ is the
simplicial object obtained by taking $S^{1}\subdot$ coproduct factors
of $R$ in simplicial degree $\ssdot$,
\[
N^{\cyc}_{\sma_{A}}R=R\otimes S^{1}\subdot.
\]
The operation $\psi^{r}$ is induced by the $r$-fold covering map
\[
q_{r}\colon S^{1}\to S^{1},\qquad 
e^{i\theta}\mapsto e^{ri\theta}.
\]
after tensoring with $R$.

\begin{definition}\label{def:Adams}
Let $A$ be a commutative ring orthogonal spectrum and $R$ a commutative
$A$-algebra.  For $r\neq0$, the Adams operation 
\[
\psi^r\colon \ATHH(R) \to \ATHH(R)
\]
is the map of (non-equivariant) commutative $A$-algebras obtained as
the tensor of $R$ with the covering map $q_{r}\colon S^{1}\to S^{1}$.
\end{definition}

We will study the equivariance of $\psi^{r}$ using the $C_{n}$-action
that arises on the edgewise subdivision $\sd_{n}$ of a cyclic set.  To
make this section more self-contained, we again recall from~\cite[\S1]{BHM}
how this works.
There are natural homeomorphisms
\[
\delta_{n}\colon |\sd_{n}X|\to |X|
\]
for the $n$-fold edgewise subdivision of a simplicial space or
simplicial orthogonal spectrum, and
canonical isomorphisms of simplicial objects
$\sd_r \sd_s X\to \sd_{rs} X$, which together make the following
diagram commute~\cite[1.12]{BHM}:
\begin{equation}
  \label{eq:subdiv-compat}
\begin{gathered}
\xymatrix{
|\sd_r \sd_s X| \ar[r] \ar[d]_{\delta_r} &
|\sd_{rs} X| \ar[d]^{\delta_{rs}} \\
|\sd_s X| \ar[r]_{\delta_s} &
|X|.
}
\end{gathered}
\end{equation}
When $X$ has a cyclic structure, $\sd_n X$ comes with a
natural $C_n$-equivariant structure which on the geometric realization
is the restriction to $C_{n}$ of the natural $S^{1}$-action; moreover,
in the diagram above, the left hand isomorphism is
$C_s$-equivariant~\cite[1.7--8]{BHM}.  

We have a simplicial model of $\psi^{r}$ by McCarthy's observation
that $q_{r}$ is the geometric realization of a quotient map of simplicial
sets $\sd_{r}S^{1}\subdot\to S^{1}\subdot$.  By naturality, 
diagram~\eqref{eq:subdiv-compat} is compatible with this quotient map.

\begin{proposition}\label{thm:Adamsequiv}
Let $A$ be a commutative ring orthogonal spectrum and $R$ a
commutative $A$-algebra.  For $r\neq 0$ and $n$ relatively prime to
$r$, the restriction of $q_{r}$ is the multiplication by $r$
isomorphism $C_{n}\to C_{n}$ and the Adams operations $\psi^{r}$ is a
map of commutative ring orthogonal $C_{n}$-spectra
\[
\psi^{r}\colon  \iota^{*}_{C_{n}}\AN_{e}^{S^{1}}R\to
q_{r}^{*}\iota^{*}_{C_{n}}\AN_{e}^{S^{1}}R.
\]
Moreover, for $s$ relatively prime to $n$, the formula
\[
(q_r)^*(\psi^s) \circ \psi^r = \psi^{rs}\colon
\iota^{*}_{C_{n}}\AN_{e}^{S^{1}}R\to
q_{rs}^{*}\iota^{*}_{C_{n}}\AN_{e}^{S^{1}}R.
\]
holds.
\end{proposition}

\begin{proof}
As above, the $r$-fold covering map defining the Adams operations
becomes a $C_n$-equivariant map
\[
\sd_n (\sd_r S^1) \to (q_r\mid_{C_{n}})^*(\sd_n S^1).
\]
Tensoring levelwise and applying $\aI_{\bR^{\infty}}^{\tilde U}$, we
obtain a map of simplicial commutative $A$-algebras
\[
\aI_{\bR^{\infty}}^{\tilde U}(R \otimes (\sd_n \sd_r S^1)) \to q_r^*
\aI_{\bR^{\infty}}^{\tilde U}(R \otimes \sd_n S^1).
\]
The result now follows from diagram~(\ref{eq:subdiv-compat}) and its
compatibility with the covering projections $q_r$.
\end{proof}

In the case when $p \nmid r$, the previous proposition shows that in
particular the operation $\psi^{r}$ should pass to categorical
$C_{p^{n}}$-fixed points (in the derived category of $A$).  Taking
fibrant replacements, we get a map (of non-equivariant $A$-modules)
\[
\psi^{r}\colon (\AN_{e}^{S^{1}}R)_{f}^{C_{p^{n}}}\to (\AN_{e}^{S^{1}}R)_{f}^{C_{p^{n}}}
\]
making the diagram
\[
\xymatrix{%
(\AN_{e}^{S^{1}}R)_{f}^{C_{p^{n+1}}}\ar[r]^{\psi^{r}}\ar[d]_{F}
&(\AN_{e}^{S^{1}}R)_{f}^{C_{p^{n+1}}}\ar[d]^{F}\\
(\AN_{e}^{S^{1}}R)_{f}^{C_{p^{n}}}\ar[r]_{\psi^{r}}
&(\AN_{e}^{S^{1}}R)_{f}^{C_{p^{n}}} }
\]
commute, where $F$ is the natural inclusion of fixed-points.  Passing
to the homotopy limit, we get an Adams operation $\psi^{r}$ on
$\ATF(R)$.

We next argue that for $p \nmid r$, the Adams operation $\psi^{r}$
descends to $TR(R)$ and $TC(R)$.

\begin{theorem}\label{thm:AdamsTRTC}
Let $R$ be a commutative ring orthogonal spectrum.  For $p \nmid r$, the Adams operation
$\psi^{r}$ induces maps
\[
\psi^{r}\colon TR(R)\to TR(R)
\]
and
\[
\psi^{r}\colon TC(R)\to TC(R).
\]

\end{theorem}

\begin{proof}
It suffices to consider the case when $R$ is cofibrant and to show
that $\psi^{r}$ commutes with the op-$p$-cyclotomic structure map
\[
\gamma =\tau_{p}\colon N_{e}^{S^{1}}R\to
\rho^{*}_{p}\Phi^{C_{p}}\aI_{\bR^{\infty}}^{\tilde U}|\sd_{p} N^{\cyc} R|.
\]
This is clear from the naturality of~\eqref{eq:subdiv-compat}.
\end{proof}

Finally, we provide the following computation for the action of the
Adams operations on $TR_{0}$ and $TC_{0}$.

\begin{theorem}\label{thm:Adamspi0}
Let $R$ be a cofibrant commutative ring orthogonal spectrum.  Assume
that $R$ is connective. Then for $p \nmid r$, the Adams operation
$\psi^{r}$ acts by the identity on $TR_{0}(R)$.
\end{theorem}

\begin{proof}

Writing $R_{0}=\pi_{0}R$, the hypothesis of connectivity
implies that
\[
\pi_{0}TR(R)\iso \pi_{0}TR(R_0),
\]
and so it suffices to consider the case when
$R=HR_0$. By~\cite[Addendum~3.3]{HMfinitealgs}, we 
have a canonical isomorphism of $TR_{0}(R)$ with the $p$-typical Witt
ring $W(R_{0})$ and canonical isomorphisms of $\pi_{0}^{C_{p^{n}}} N_e^{S^1} R$
with $W_{n+1}(R_{0})$, the $p$-typical Witt vectors of length $n+1$.
Letting $R_{0}$ vary over all commutative rings, $\psi^{r}$ then
restricts to natural transformations $\psi^{r}_{n+1}$ of rings $W_{n+1}(-)\to
W_{n+1}(-)$, compatible with the restriction maps. We complete the
proof by arguing that such a natural transformation must be the identity. 

Since $W_{n+1}$ is representable, it suffices to prove that
$\psi^{r}_{n+1}$ is the identity when $R_{0}$ is the representing
object $\bZ[x_{0},\dotsc,x_{n}]$, or, since this is torsion free, when
$R_{0}=\bQ[x_{0},\dotsc,x_{n}]$. \textit{A fortiori}, it suffices to prove
$\psi^{r}_{n+1}$ is the identity when $R_{0}$ is a $\bQ$-algebra.
Since for a $\bQ$-algebra $W_{n+1}(R_{0})$ is isomorphic as a ring to
the Cartesian product of $n+1$ copies of $R_{0}$ via the ghost
coordinates, the only possible natural ring endomorphisms of $W_{n+1}$
are the maps that permute the factors.  Since $\psi^{r}$ commutes with
the restriction map $R$ on $TR(R)$, and on the ghost coordinates the
restriction map induces the projection onto the first $n$ factors, it
follows by induction that $\psi^{r}_{n+1}$ is the identity.
\end{proof}

\begin{corollary}\label{cor:AdamsTCpi0}
Let $R$ be a commutative ring orthogonal spectrum.  Assume that $R$ is
connective and that $p \nmid r$.  Then $TC_{0}(R)$ has the Frobenius
invariants of $W(\pi_{0}R)$ as a quotient and the action of $\psi^{r}$
descends to the identity map on this quotient.
\end{corollary}

\begin{example}\label{ex:cpinf}
When we take $R=S$ to be the sphere spectrum, \cite[\S5]{BHM}
identifies $TC(S)\phat$ as $(S\vee \Sigma \bC P^{\infty}_{-1})\phat$,
where $\bC P^{\infty}_{-1}$ denotes the Thom spectrum of the virtual
bundle $-L$, where $L$ denotes the tautological line bundle.  More to
the point, $\Sigma \bC P^{\infty}_{-1}$ is the homotopy fiber of the
$S^{1}$-transfer $\Sigma \Sigma^{\infty}_{+}\bC P^{\infty}\to S$.  The
tom Dieck splitting identifies 
\[
TR^{n}(S)\phat\htp 
   \prod_{0\leq m\leq n} (\Sigma^{\infty}_{+}B(C_{p^{n}}/C_{p^{m}}))\phat
\iso 
\prod_{0\leq k\leq n} (\Sigma^{\infty}_{+}B(C_{p^{k}}))\phat.
\]
For $p\nmid r$, $\psi^{r}$ acts on $THH(S)$ as the identity (on the
point set level), and so acts on the $C_{p^{n}}$-fixed points via the
multiplication by $r$ map $C_{p^{n}}\to C_{p^{n}}$.  It therefore
induces the corresponding multiplication by $r$ map on each
classifying space $B(C_{p^{n}}/C_{p^{m}})$ in each factor in
$TR^{n}(S)$; note that multiplication by $r$ on $C_{p^{n}}/C_{p^{m}}$
is multiplication by $r$ on $C_{p^{k}}$ (under the canonical
isomorphism).  This allows us to determine the action of $\psi^{r}$ on
$TC(S)$.  The computation of $TC(S)$ in \cite[\S5]{BHM} and
\cite[\S4.4]{MadsenTraces} uses a weak equivalence
\[
(\Sigma \Sigma^{\infty}_{+}\bC P^{\infty})\phat \simeq 
\holim (\Sigma^{\infty}_{+}BC_{p^{k}})\phat,
\]
and the action of $\psi^{r}$ on $BC_{p^{k}}$ is compatible with
the action of $\psi^{r}$ on $(\Sigma \Sigma^{\infty}_{+}\bC
P^{\infty})\phat$ given by multiplication by $r$ on the suspension and
the action on $\bC P^{\infty}\simeq K(\bZ,2)$ induced by the
multiplication by $r$ on $\bZ$. The fiber sequence 
\[
\Sigma \bC P^{\infty}_{-1}\to \Sigma \Sigma^{\infty}_{+}\bC
P^{\infty}\to S
\]
has a consistent action of $\psi^{r}$ (where we use the trivial action
on $S$).  After $p$-completion, the action of $\{r\mid\   p \nmid r\}$
extends to an action of the units of $\bZ\phat$.  The Teichm\"uller
character then gives an action of $(\bZ/p)^{\times}$ and (since $p-1$
is invertible in $\bZ\phat$) a splitting into $p-1$ ``eigenspectra''
wedge summands.  This decomposition of $TC(S)\phat$ is well-known and
plays a role in Rognes' cohomological analysis of $Wh(*)\phat$ at
regular primes \cite[\S5]{Rognes-WHP}. 
\end{example}

\section{Madsen's remarks}

In his CDM notes \cite[p.~218]{MadsenTraces}, Madsen describes the
restriction map, and notes that the inverse is not as readily
accessible even in the algebraic setting since ``$\Delta(r)=r\otimes
\dotsb \otimes r$ is not linear''.  Yet in our framework, we naturally get
the inverse to the cyclotomic structure map, rather than the
cyclotomic structure map itself. At first blush, this seems to pose a
curious contradiction.  The answer arises from the transfer: $v\mapsto
v^{\otimes p}$ is linear modulo the ideal generated by the transfer,
and this is exactly the ideal killed by $L\Phi^{H}$.

The observation that the ideal killed by $L\Phi^{H}$ coincides with
the ideal generated by the transfer is essentially a formal
consequence of the definition of the derived geometric fixed point
functor: $L\Phi^{H}(X)=(X\sma \widetilde{E\aP})^{H}$ is a composite of
the categorical fixed points with the localization killing cells of
the form $S^{1}/K$ for $K$ a proper subgroup of $H$. Computationally,
this means that all transfers from proper subgroups of $H$ are killed.

The observation that the algebraic diagonal map is linear modulo the
transfer is more interesting.  In particular, this question highlights
the issue of constructing an algebraic model of the norm functor that
correctly reflects the homotopy theory.  We first consider the naive
smash power which is simply the $C_{p}$-module $(\mathbb
Z\{x,y\})^{\otimes p},$ where $\mathbb Z\{x,y\}$ is the free abelian group on the set
$\{x,y\}$. Inside is the element $(x+y)^{\otimes p}$, which is
obviously in the fixed points of the $C_{p}$-action.  In this context,
Madsen's remark boils down to the fact that $(x+y)^{\otimes p}$ is not 
$x^{\otimes p}+y^{\otimes p}$.  We can expand $(x+y)^{\otimes p}$
using a non-commutative version of the binomial theorem as follows.  Observing that the full
symmetric group $\Sigma_{p}$ acts on the tensor power (and the $C_{p}$-action is just the obvious restriction), if we group all terms with $i$
tensor factors of $x$ and $p-i$ tensor factors of $y$, then we see
that the symmetric group permutes these and a subgroup conjugate to
$\Sigma_{i}\times\Sigma_{p-i}$ stabilizes each element.  We therefore
see that the sum of all of such terms for a fixed $i$ can be expressed
as the transfer 
\[
\Tr_{\Sigma_{i}\times\Sigma_{p-i}}^{\Sigma_{p}} x^{\otimes i}\otimes
y^{\otimes (p-i)}.
\]
Letting $i$ vary and summing the terms (and then restricting back
to $C_{p}$) shows that 
\[
(x+y)^{\otimes p}=x^{\otimes p}+y^{\otimes
p}+\Res_{C_{p}}^{\Sigma_{p}}\big(\sum_{i=1}^{p-1}\Tr_{\Sigma_{i}\times\Sigma_{p-i}}^{\Sigma_{p}}x^{\otimes
i}y^{\otimes (p-i)}\big).
\]
All of the terms involving transfers are in the ideal generated by
transfers by definition, and so we conclude that the
$p$\textsuperscript{th} power map is linear modulo these.

However, this algebraic model is not the correct analogue of the norm.
First, when we reduce modulo the transfer from proper subgroups in the
$p$th tensor power of a ring, then we also kill the transfer of the
element $1$. This then takes us from $\mathbb Z$-modules to $\mathbb
Z/p$-modules. Second, the fixed point Mackey functor associated to the
$p$th tensor power functor is not the right algebraic version of the
norm.

There are now several constructions of a norm functor in the category
of Mackey functors that exhibit the correct homotopy-theoretic
behavior.  Mazur describes one for cyclic $p$-groups \cite{Mazur},
Hill-Hopkins gives one for a general finite group by stepping through
the norm in spectra \cite{HillHopkinsLocalization}, and subsequently
Hoyer gave a purely algebraic definition for all finite groups and
showed it to be equivalent to the others \cite{HoyerThesis}.  One of
the basic properties of the algebraic norm is that the norm from
$H$-Mackey functors to $G$-Mackey functors is the functor underlying
the left adjoint to the forgetful functor from $G$-Tambara functors to
$H$-Tambara functors. In particular, since $\pi_0(R)$ for $R$ a
commutative ring $G$-spectrum is a $G$-Tambara functor~\cite{Brun}, the
algebraic norm precisely mirrors the multiplicative behavior of the
norm in spectra.  A more detailed exposition of the connection between
the algebraic norm and $THH$ will appear in~\cite{AlgHoch}.

In this context, if $R$ is a commutative ring, then the inverse map considered by Madsen is exactly the
universally defined norm map 
\[
N_{e}^{C_{p}}\colon R\to N_{e}^{C_{p}}(R)(C_{p}/C_{p})
\]
underlying the Tambara functor structure.  While this map is not
linear, it is so modulo the transfer~\cite{Tambara}.  In fact, just as
in topology, this map is a right inverse to the ``geometric fixed
points'' functor $\Phi^{C_{p}}$ on Mackey functors, the map which takes a Mackey
functor $\underline{M}$ and returns the quotient group
$\underline{M}(G/G)/\im(\Tr)$, where $\im(\Tr)$ denotes the image of
the transfer: $\Phi^{C_{p}}\circ N_{e}^{C_{p}} = Id$.

We close by illustrating this all with an example which shows the
failure of the ``naive'' tensor power approach and the strength (and
relative computability) of the Tambara functor approach to the
algebraic norm.  Let $p=2$, and let $R=\mathbb Z[x]$.  Then the
two-fold tensor power, $C_{2}$-equivariantly, is
\[
\mathbb Z[C_{2}\cdot x]=\mathbb Z[x,gx].
\]
The transfer ideal is generated by $2$ and $x+gx$, and modulo $2$ and $x+gx$, the map $x\mapsto x\cdot gx$ induces the canonical surjection
\[
\mathbb Z[x]\to \mathbb Z/2[x\cdot gx].
\]
In this example, the map from $R$ to the quotient of the fixed points of $R^{\otimes 2}$ by the ideal given by the transfer is not an isomorphism; we can interpret the failure to be an isomorphism as a failure to correctly interpret the transfer of the element $1$. In particular, restricting to the submodule generated by $1$ we implicitly computed
\[
N_{e}^{C_{2}}\mathbb Z=\mathbb Z,
\]
endowed with the trivial action. This is not what the algebraic norm computes for us!

For $G=C_{2}$ and for $R=\mathbb Z[x]$, the fixed points of $N_{e}^{C_{2}}(\mathbb Z[x])$ are the ring
\[
\mathbb Z[t,y,x\cdot gx]/(t^{2}-2t, ty-2y),
\]
with the elements $t$ and $y$ the transfers of $1$ and $x$ respectively (the restriction map takes $t$ to $2$, $y$ to $x+gx$ and $x\cdot gx$ to itself). In particular, we observe that the unit $1$ generates not a copy of $\mathbb Z$ but rather a copy of the Burnside ring $\mathbb Z[t]/t^{2}-2t$. Thus, modulo the image of the transfer, this ring is simply $\mathbb Z[x\cdot gx]$, and the norm map $x\mapsto x\cdot gx$ is an isomorphism.


\begin{thebibliography}{00}

\bibitem{AlgHoch} V.~Angeltveit, A.~J.~Blumberg, T.~Gerhardt, M.~A.~Hill, T.~Lawson, and M.~A.~Mandell.
\newblock Algebraic Hochschild homology of Mackey functors.
\newblock Preprint, 2014.

\bibitem{BlumbergHill} A.~J.~Blumberg and M.~A.~Hill.
\newblock Operadic multiplications in equivariant spectra, norms, and
transfers.
\newblock arXiv:1309.1750.

\bibitem{BM2} A.~J.~Blumberg and M.~A.~Mandell.
\newblock Localization theorems in topological {H}ochschild homology
and topological cyclic homology.
\newblock Geom. and Top. {\bf 16} (2012), 1053--1120.

\bibitem{BMtw} A.~J.~Blumberg and M.~A.~Mandell.
\newblock Localization for {$THH(ku)$} and the topological {H}ochschild
and cyclic homology of {W}aldhausen categories.
\newblock arXiv:1111.4003.

\bibitem{BM} A.~J.~Blumberg and M.~A.~Mandell.
\newblock The homotopy theory of cyclotomic spectra.
\newblock arXiv:1303.1694.

\bibitem{Bokstedt} M.~B{\"o}kstedt.
\newblock Topological {H}ochschild homology.
\newblock Preprint, 1990.

\bibitem{BHM} M.~B{\"o}kstedt and W.C.~Hsiang and
I.~Madsen.
\newblock The cyclotomic trace and algebraic {$K$}-theory of spaces.
\newblock Invent. Math. {\bf 111}(3) (1993), 465--539.

\bibitem{Brun} M. Brun.
\newblock Witt vectors and {T}ambara functors.
\newblock Adv. Math. {\bf 193} (2005), no.~2, 233--256.

\bibitem{BCD} M. Brun, G. Carlsson and B. I. Dundas.
\newblock Covering homology.
\newblock Adv. Math. {\bf 225} (2010), no.~6, 3166--3213.

\bibitem{bohmann} Anna Marie Bohmann.
\newblock (appendix by A.~Bohmann and E.~Riehl.)
\newblock A comparison of norm maps.
\newblock arXiv:1201.6277.

\bibitem{Dundas} B.~I.~Dundas.
\newblock Relative {$K$}-theory and topological cyclic homology.
\newblock Acta Math. {\bf 179} (2) (1997), 223--242.

\bibitem{EKMM}
A.~D. Elmendorf, I.~Kriz, M.~A. Mandell, and J.~P. May, 
\newblock {\em Rings, modules, and algebras in stable homotopy theory},
  volume~47 of {\em Mathematical Surveys and Monographs}.
\newblock American Mathematical Society, Providence, RI, 1997.
\newblock With an appendix by M. Cole.

\bibitem{Evens} L.~Evens.
\newblock The cohomology of groups.
\newblock Oxford University Press, 1991.

\bibitem{FiedorowiczGajda}
Z. Fiedorwicz and W.~Gajda.
\newblock The {$S^1$}-CW decomposition of the geometric realization of a cyclic set.
\newblock Fund. Math. 145 (1) (1994), 91--100. 

\bibitem{GerstenhaberSchack} M.~Gerstenhaber and S.~D.~Schack.
\newblock A {H}odge-type decomposition for commutative algebra cohomology.
\newblock J. Pure. Appl. Alg. {\bf 48} (1987), 229--247.

\bibitem{GreenleesMayLoc} J.~P.~C.~Greenlees and J.~P.~May.
\newblock Localization and completion theorems for {$MU$}-module spectra.
\newblock Ann. of Math. {\bf 146} (1997), 509--544.

\bibitem{HMfinitealgs} L.~Hesselholt and I.~Madsen.
\newblock On the {$K$}-theory of finite algebras over {W}itt vectors of perfect
  fields.
\newblock \emph{Topology}, 36\penalty0 (1):\penalty0 29--101, 1997.

\bibitem{HMannals} L.~Hesselholt and I.~Madsen,
\newblock On the {$K$}-theory of local fields.
\newblock Ann. of Math., {\bf 158}(2) (2003), 1--113.

\bibitem{HillHopkinsLocalization}
Michael~A. Hill and Michael~J. Hopkins.
\newblock Equivariant localization.
\newblock 2012.

\bibitem{HHR} M.~A.~Hill and M.~J.~Hopkins and D.~C.~Ravenel.
\newblock On the non-existence of elements of {K}ervaire invariant one.
\newblock arXiv:0908.3724.

\bibitem{HoyerThesis}
Rolf Hoyer.
\newblock {\em Two topics in stable homotopy theory}.
\newblock PhD thesis, University of Chicago, 6 2014.

\bibitem{JonesCyclic} J.~D.~S.~Jones.
\newblock Cyclic homology and equivariant homology.
\newblock Invent. math. {\bf 87} (1987), 403--423.

\bibitem{LewisMandell2} L.~G.~Lewis, Jr. and M.~A.~Mandell. 
\newblock Equivariant universal coefficient and {K}unneth spectral sequences.
\newblock Proc. London Math. Soc. {\bf 92} (2006), no. 2, 505-544.

\bibitem{Loday} J.~L.~Loday. 
\newblock Operations sur l'homologie cyclique des algebres commutatives.
\newblock Invent. math. {\bf 96} (1989), 205--230.

\bibitem{LodayCyclic} J.~L.~Loday.
\newblock Cyclic homology.
\newblock Springer (1998).

\bibitem{LMS}
L.~G. Lewis, Jr., J.~P. May, M.~Steinberger, and J.~E. McClure.
\newblock {\em Equivariant stable homotopy theory}, volume 1213 of {\em Lecture
  Notes in Mathematics}.
\newblock Springer-Verlag, Berlin, 1986.
\newblock With contributions by J. E. McClure.

\bibitem{HTT} J. Lurie.
\newblock {\em Higher topos theory}
\newblock Annals of Mathematics Studies, vol. 170, Princeton University Press, Princeton, NJ, 2009.

\bibitem{MadsenTraces} I.~Madsen.
\newblock Algebraic {$K$}-theory and traces.
\newblock Curr. Dev. in Math., Internat. Press, Cambridge, MA, 1994, 191--321.

\bibitem{MM} M.~A.~Mandell and J.~P.~May.
\newblock Equivariant orthogonal spectra and {$S$}-modules.
\newblock Mem. of the Amer. Math. Soc. {\bf 159} (755), 2002.

\bibitem{MalkiewichDX}
C.~Malkiewich.
\newblock On the topological Hochschild homology of {DX}.
\newblock arXiv:1505.06778.

\bibitem{Mazur}
Kristen Mazur.
\newblock An equivariant tensor product on {M}ackey functors.
\newblock arXiv:1508.04062.


\bibitem{McCarthy} R.~McCarthy.
\newblock Relative algebraic {$K$}-theory and topological cyclic homology.
\newblock Acta Math., {\bf 179}(2) (1997), 197--222.

\bibitem{McCarthyMinasian} R.~McCarthy and V.~Minasian.
\newblock {HKR} theorem for smooth {$S$}-algebras.
\newblock J. Pure Appl. Algebra {\bf 185} (2003), 239--258.

\bibitem{MSV} J.~E.~McClure and R.~Schwanzl and R.~Vogt.
\newblock {$THH(R) \cong R \otimes S^1$} for {$E_\infty$} ring spectra.
\newblock J. Pure Appl. Algebra {\bf 121}(2) (1997), 137--159.

\bibitem{Rognes-WHP}
John Rognes.
\newblock The smooth {W}hitehead spectrum of a point at odd regular primes.
\newblock {\em Geom. Topol.}, 7:155--184 (electronic), 2003.

\bibitem{schwequi} S.~Schwede.
\newblock Lectures on equivariant stable homotopy theory.
\newblock
Preprint, \url{http://www.math.uni-bonn.de/~schwede/equivariant.pdf}, 2013.

\bibitem{Stolz} M.~Stolz.
\newblock Equivariant structures on smash powers of commutative ring spectra.
\newblock Doctoral thesis, University of Bergen, 2011.

\bibitem{Tambara}
D.~Tambara.
\newblock On multiplicative transfer.
\newblock {\em Comm. Algebra}, 21(4):1393--1420, 1993.

\bibitem{Ullman} J.~Ullmann.
\newblock On the regular slice spectral sequence.
\newblock Doctoral thesis, MIT, 2013.

\end{thebibliography}
\end{document}